\documentclass[final,leqno]{amsart}
\usepackage{amsmath,amssymb,amsfonts,latexsym,stmaryrd}
\usepackage{graphicx,float,epsfig} 
\usepackage{mathrsfs} 
\usepackage{float} 
\usepackage{color}
\usepackage{setspace} 

\usepackage{soul}
\usepackage[dvipsnames]{xcolor}

\DeclareGraphicsExtensions{ .eps,.ps} \usepackage{subfig}
\usepackage{multirow}
\usepackage{url}
\usepackage[linesnumbered,algoruled,boxed]{algorithm2e}

\usepackage[colorlinks=false, linkcolor=blue,  citecolor=OliveGreen,  pdfstartview={}{}]{hyperref} 
\usepackage[dvipsnames]{xcolor}

\usepackage{amsmath}

\def\O{\Omega}

\newtheorem{prob}{Problem}

\newtheorem{remark}{Remark}[section]
\newtheorem{lemma}{Lemma}[section]
\newtheorem{corollary}{Corollary}[section]
\newtheorem{thm}{Theorem}[section]

\newtheorem{definition}{Definition}
\usepackage{booktabs}
\usepackage{float}

\newcommand\bu{\boldsymbol{u}}

\def\CT{{\mathcal T}}

\newcommand\bPi{\boldsymbol{\Pi}}

\renewcommand\H{\mathrm{H}}

\renewcommand\O{\Omega}

\newcommand\rot{\mathop{\mathrm{rot}}\nolimits}

\newcommand{\vertiii}[1]{{\left\vert\kern-0.25ex\left\vert\kern-0.25ex\left\vert #1 
    \right\vert\kern-0.25ex\right\vert\kern-0.25ex\right\vert}}

\begin{document}

\title[VEM for the elasticity problem with small edges]
{A virtual element method for the elasticity problem allowing small edges}

\author{Danilo Amigo}
\address{GIMNAP-Departamento de Matem\'atica, Universidad del B\'io - B\'io, Casilla 5-C, Concepci\'on, Chile.}
\email{danilo.amigo2101@alumnos.ubiobio.cl}
\author{Felipe Lepe}
\address{GIMNAP-Departamento de Matem\'atica, Universidad del B\'io - B\'io, Casilla 5-C, Concepci\'on, Chile.}
\email{flepe@ubiobio.cl}
\thanks{The first and second authors were partially supported by
DIUBB through project 2120173 GI/C Universidad del B\'io-B\'io and ANID-Chile through FONDECYT project 11200529 (Chile).}

\author{Gonzalo Rivera}
\address{Departamento de Ciencias Exactas,
Universidad de Los Lagos, Casilla 933, Osorno, Chile.}
\email{gonzalo.rivera@ulagos.cl}
\thanks{The third author was partially supported by
 Universidad de Los Lagos through regular project R02/21.}




\begin{abstract}
In this paper we analyze a virtual element method for the two dimensional elasticity problem allowing small edges. With this approach, the classic  assumptions on the geometrical features of the polygonal meshes can be relaxed. In particular, we consider only star-shaped polygons for the meshes. Suitable error estimates are presented, where a rigorous analysis on the influence of  the Lam\'e constants in each estimate  is presented. We report numerical tests to assess the performance of the method.
\end{abstract}

\maketitle

\section{Introduction}\label{sec:intro}
The virtual element method (VEM) is a numerical approach that has taken relevance in the community 
of numerical analysis in the nowadays, since t allows to approximate the solutions of partial  differential equations
with accuracy together with a computational  implementation that results to be easy to handle. Moreover, the flexibility of this method
on the geometrical assumptions on the elements of the meshes, which permits arbitrary polygonal elements such as convex and nonconvex polygons, and hanging nodes, that are not allowed in the finite element approach, just for mention the most relevant. These considerations, lead the VEM to be a suitable alternative to approximate physical phenomenons that the standard finite element method is not possible to consider, such as domains with fractures, cuspids, holes, etc. A full treatment of the basic 
elements of the VEM is available in \cite{MR2997471, MR3200242}.

In recent years, different studies have been carried out that demonstrate the good performance of VEM on problems of different natute, such as evolutionary problems, spectral problems, nonlinear problems, among others. For a better knowledge of the recent state of art we resort to  \cite{ABM2022}.  Now, in the continuous advance in the research of virtual methods, a new approach has been  developed that allows the use of even more general and flexible meshes. This  new approach for the VEM is now related to consider small edges for the polygonal elements on the meshes, as is presented in \cite{MR3714637,MR3815658, MR3857893}. More precisely, depending on the regularity of the solution of the partial differential equations, is possible to only consider  star shaped polygonal elements, or a bounded number of edges for the polygons. In both cases
these assumptions are enough to allow sufficiently small sides on the elements. To make matter precise, in \cite{MR3714637}
the authors have developed this new strategy considering the Laplace operator in two dimensions, whereas in \cite{MR3815658} have been capable to extend the ideas for three dimensions, proving that VEM with small edges also 
works when more complicated domains are considered.

The aim now is to analyze the performance of the VEM with small edges in other contexts, in fact, this research is in ongoing progress.  For instance, in \cite{MR4284360} an application of this new approach in eigenvalue problems has shown the accuracy of the approximation of the spectrum for the Steklov eigenvalue problem, in \cite{MR4461634} for elliptic interface problems, \cite{MR4359996} for three dimensional problems considering polytopal meshes, etc.

In the present paper, our contribution is to apply the methodology of small edges for the two dimensional linear elasticity problem. The bibliography related to numerical methods to approximate the displacement of some elastic structure is abundant, where different method as been proposed, as  \cite{MR2293249, MR3860570, MR2754580, MR3453481, MR4497821, MR3376135, MR4477479}, where mixed finite element methods, mixed virtual element methods, stabilized methods, discontinuous Galerkin methods,  just for mention some of them, have been considered. Regarding the study of VEM applied to elasticity problems, we can cite the following works \cite{MR3715385,MR4126310,MR4050542,MR3845217,MR4119968,MR3767447,MR3881593}. In particular,  a VEM for the  primal formulation of the elasticity eigenvalue problem as been proposed in \cite{MR4050542},  where the proposed virtual spaces on this reference are a suitable alternative for the load problem. These virtual spaces for the elasticity equations, together with the small edges approach, are the core of our contribution.

In our paper, we are interested in the primal formulation of the linear elasticity equations with vanishing Dirichlet boundary condition. This boundary condition is needed for our purposes, since lead to the precise regularity that the small edges approach needs to be performed. 

It is well known that the primal formulation for linear elasticity leads to the so called locking phenomenon when finite element families to discretize $\H^1$ are implemented, which not occurs when mixed formulations are considered. This locking is due to the effects of the Poisson ratio, which we denote by $\nu$,  on the computation of the Lam\'e coefficient $\lambda_S$: if $\nu$ tends to $1/2$, then $\lambda_S$ goes to infinity. Then, it is important to have a correct control and knowledge  on the constants that appear on the estimates when the mathematical analysis is performed. Moreover, the numerical tests that we will present lead to the conclusion that,  regardless of the  family of polygonal mesh under consideration, the convergence errors are not deteriorated; however, when $\nu$ tends to $1/2$, the deterioration of the error curves is observed.

With this consideration, for the small edges approach we need to take care on this Lam\'e coefficient as with  any other family of elements that discretizes $\H^1$. Moreover, the small edges approach that we will consider to run the analysis is based in the results of \cite{MR3815658}. Hence, we need to adapt the theory, already available for the Laplace operator, for the elasticity operator. This obligates to have a spacial treatment on this context due the Lam\'e constants, specially $\lambda_S$.

The outline of this manuscript is as follows. In Section \ref{sec:model_problem}  the model problem that we consider in our paper, summarizing the results that are needed along our studies, together with definitions and notations to develop the mathematical analysis. The core of our manuscript is contained in  Section \ref{sec:vem}, where  we introduce the virtual element method of our interest. Under the assumption that the meshes allow small edges for the elements, we prove a series of technical results that are needed for the elasticity operator. These results contain important information related to how the physical parameters, inherent on the linear elasticity equations, are involved on the estimates associated to projections, interpolant operators and error estimates, among others.  Finally, in Section \ref{sec:numerics}, we report numerical tests, that illustrate the theory and exhibit the performance of the proposed method.


%
\section{The model problem}
\label{sec:model_problem}
Let  $\Omega \subset \mathbb{R}^{2}$ be a convex, bounded, and open Lipschitz domain with boundary $\partial \Omega$.  The linear elasticity problem is as follows
\begin{equation}
\label{elast}
\left\{\begin{array}{cccc}
\textbf{div}(\boldsymbol{\sigma}(\textbf{u})) &=& -\varrho \textbf{f} \quad \text{in} \; \Omega, \\
\textbf{u} &=&  \boldsymbol{0} \quad \text{on} \; \partial\O,
\end{array}\right.
\end{equation}
where  $\textbf{u}$ represents the displacement of the solid, $\textbf{f}$ is an external load, $\varrho$ represents the density of the material that we assume constant,  $\boldsymbol{\varepsilon}$ represents the strain tensor, and  $\boldsymbol{\sigma}$  is the Cauchy tensor. The strain and Cauchy tensors are related through the operator $\boldsymbol{\mathcal{C}}$, where $\boldsymbol{\sigma}(\textbf{u}) = \boldsymbol{\mathcal{C}}\boldsymbol{\varepsilon}(\textbf{u})$ and,  which according to Hooke's law,   $\boldsymbol{\mathcal{C}}$ is defined by 
\begin{equation*}
\boldsymbol{\mathcal{C}}\boldsymbol{\tau} := 2\mu_{S}\boldsymbol{\tau} + \lambda_{S}\text{tr}(\boldsymbol{\tau})\mathbb{I},
\end{equation*}
where $\mathbb{I}\in\mathbb{R}^{2\times 2}$ is the identity matrix, and  $\mu_S$ and $\lambda_S$ are the Lam\'e coefficients defined by
\begin{equation*}
\mu_{S} = \dfrac{E}{2(1 + \nu)} \quad	\text{and}\quad \lambda_{S} = \dfrac{E\nu}{(1 + \nu)(1 - 2\nu)},
\end{equation*}
where  $E$ is the Young's modulus and $\nu$ is the Poisson ratio. Let us remark that when $\nu\rightarrow 1/2$, then $\lambda_S\rightarrow \infty$, which is an important issue that we need to take in consideration  for our purposes.
The variational formulation of  \eqref{elast} is the following: Given $\textbf{f} \in \textbf{L}^{2}(\Omega)$, find $\textbf{u} \in \textbf{H}_{0}^{1}(\Omega)$ such that 
\begin{equation}
\label{eq:variational1}
\displaystyle{\int_{\Omega}} \boldsymbol{\sigma}(\textbf{u}) : \boldsymbol{\varepsilon}(\textbf{v}) = \displaystyle{\int_{\Omega} \varrho \textbf{f}\cdot \textbf{v}} \quad \forall \; \textbf{v} \in \textbf{H}^{1}_{0}(\Omega).
\end{equation}
\noindent Let us define the symmetric and  continuous bilinear form
\begin{equation}\label{a}
a :  \textbf{H}_{0}^{1}(\Omega) \times \textbf{H}_{0}^{1}(\Omega) \longrightarrow \mathbb{R}, \quad a(\textbf{w}, \textbf{v}) := \displaystyle{\int_{\Omega}} \boldsymbol{\sigma}(\textbf{w}) : \boldsymbol{\varepsilon}(\textbf{v}) \quad \forall \textbf{w}, \textbf{v} \in \textbf{H}_{0}^{1}(\Omega).
\end{equation}
and the linear functional
\begin{equation}\label{F}
F : \textbf{H}_{0}^{1}(\Omega) \longrightarrow \mathbb{R}, \quad F(\textbf{v}) := \displaystyle{\int_{\Omega} \varrho \textbf{f}\cdot \textbf{v}} \quad \forall \textbf{v} \in \textbf{H}_{0}^{1}(\Omega).
\end{equation}
Note that $|a(\textbf{w}, \textbf{v})| \lesssim \max\{\lambda_{S},\mu_{S}\}\|\textbf{w}\|_{1,\O}\|\textbf{v}\|_{1,\O}$, and $F(\textbf{v}) \leq \varrho\|\textbf{f}\|_{0,\O}\|\textbf{v}\|_{1,\O}$, proving that $a(\cdot,\cdot)$ and $F(\cdot)$ are bounded. Now, with  \eqref{a} and \eqref{F} at hand, we rewrite \eqref{eq:variational1} as follows
\begin{prob}\label{continuous} 
Given $\textbf{f} \in \textbf{L}^{2}(\Omega)$, find $\textbf{u} \in \textbf{H}_{0}^{1}(\Omega)$ such that 
\begin{equation*}
\label{eq:source_ab}
a(\textbf{u}, \textbf{v}) = F(\textbf{v}), \quad \forall \; \textbf{v} \in \textbf{H}_{0}^{1}(\Omega).
\end{equation*}
\end{prob}
From Korn's inequality and Lax-Milgram's lemma we have that Problem \ref{continuous} is well posed and its solution satisfies  $\|\textbf{u}\|_{1, \Omega} \leq \dfrac{1}{\widehat{C}}\|\textbf{f}\|_{0, \Omega}$, where  $\widehat{C}>0$ is independent of  $\lambda_{S}$. Moreover, from 
 \cite{MR840970}, the following regularity holds.
\begin{lemma}\label{reg}
For every $\textbf{f} \in \textbf{L}^{2}(\Omega)$, the solution of Problem \ref{continuous} is such that $\textbf{u} \in \textbf{H}^{2}(\Omega)$. Moreover, there exists $C > 0$ such that $\|\textbf{u}\|_{2,\Omega} \leq C\|\textbf{f}\|_{0,\Omega}$.
\end{lemma}
\section{The virtual element method}
\label{sec:vem}
In the present section we introduce the virtual element method that we consider to approximate the solution of Problem \ref{continuous}. To do this task, we will consider a more relaxed   conditions compared with those introduced in  \cite{MR2997471} for the classic VEM, where there is no possible to assume more general polygonal meshes allowing  arbitrary edges, more precisely, small edges. Hence, and inspired in  \cite{MR3714637}, if $\{\mathcal{T}_h\}_{h>0}$ represents a family of polygonal meshes to discretize $\O$,  $E\in\CT_h$ is an arbitrary element of the mesh, and  $h := \underset{E \in \mathcal{T}_{h}}{\max} \; h_{E}$ represents the mesh size.  On the other hand, let $\texttt{a},\texttt{b}\in\mathbb{R}$. If $\texttt{a}\leq C\texttt{b}$ and $\texttt{b}\leq C\texttt{a}$ with $C\in\mathbb{R}^+$, we define the relation  $\texttt{a}\simeq\texttt{b}$. Then, from now and on, we will assume that for $E\in \CT_h$, there holds that  $|E|\simeq h_E^2$ and $|\partial E|\simeq h_E$.

Let us assume the following assumption on $\CT_h$:
\begin{itemize}
\item[\textbf{A1.}] There exists  $\gamma \in \mathbb{R}^{+}$ such that each polygon $E \in \{\mathcal{T}_{h}\}_{h>0}$ is star-shaped with respect to a ball  $B_{E}$ with center  $\textbf{x}_{E}$ and radius  $\rho_{E} \geq \gamma h_{E}$.
\end{itemize}
We denote by $N$ the number of vertices of $\mathcal{T}_{h}$. \\

Let us write the bilinear form $a(\cdot, \cdot)$ and the functional $F(\cdot)$ as follows
\begin{center}
$a(\textbf{w}, \textbf{v}) = \displaystyle{\sum_{E \in \mathcal{T}_{h}} a^{E}(\textbf{w}, \textbf{v})}$ \quad where $a^{E}(\textbf{w}, \textbf{v}) := \displaystyle{\int_{E} \boldsymbol{\sigma}(\textbf{w}) : \boldsymbol{\varepsilon}(\textbf{v})}$ \quad $\forall \; \textbf{w}, \textbf{v} \in \textbf{H}^{1}_{0}(\Omega)$,
\end{center}
\begin{center}
$F(\textbf{v}) = \displaystyle{\sum_{E \in \mathcal{T}_{h}} F^{E}(\textbf{v})}$ \quad where $F^{E}(\textbf{v}) :=  \displaystyle{\int_{E} \varrho\textbf{f}\cdot\textbf{v}}$ \quad $\forall \; \textbf{v} \in \textbf{H}^{1}_{0}(\Omega)$.
\end{center}
\subsection{Virtual spaces}
Now we introduce the virtual spaces of our interest. Following \cite{MR4567234} and \cite{MR3714637}, we introduce the following local spaces
\begin{center}
$\mathbb{B}_{\partial E} := \{\textbf{v}_{h} \in \boldsymbol{\mathcal{C}}^{0}(\partial E) \; : \; \textbf{v}_{h}\lvert_{e} \in \mathbb{P}_{k}(e) \; \forall e \subset \partial E\}$,
\end{center}
\begin{center}
$\boldsymbol{\mathcal{W}}_{h}^{E} := \{\textbf{v}_{h} \in \textbf{H}^{1}(E) \; : \; \Delta \textbf{v}_{h} \in [\mathbb{P}_{k}(E)]^{2} \; \text{y} \; \textbf{v}_{h}\lvert_{\partial E} \; \in \mathbb{B}_{\partial E}\}$.
\end{center}
For each,  $E \in \{\mathcal{T}_{h}\}_{h>0}$, we introduce the projection $\boldsymbol{\Pi}_{k,E} : \boldsymbol{\mathcal{W}}_{h}^{E} \longrightarrow [\mathbb{P}_{k}(E)]^{2},$
defined for every  $\textbf{v}_{h} \in \boldsymbol{\mathcal{W}}_{h}^{E}$ as the solution of  (see \cite{MR3881593} for instance)
\begin{equation*}
\label{elastt}
\left\{\begin{array}{ccl}
a^{E}(\textbf{v}_{h} - \boldsymbol{\Pi}_{k,E}\textbf{v}_{h}, \boldsymbol{p}) &=&  0 \quad \forall \; \boldsymbol{p} \in [\mathbb{P}_{k}(E)]^{2}, \\
\displaystyle{\int_{\partial E} \textbf{v}_{h} - \boldsymbol{\Pi}_{k,E}\textbf{v}_{h}} &=& 0, \\
\displaystyle{\int_{E} \rot(\textbf{v}_{h}) - \rot\left(\boldsymbol{\Pi}_{k,E}\textbf{v}_{h}\right)} &=& 0
\end{array}\right.
\end{equation*}
\begin{definition}
We define the local virtual space by 
\begin{center}
$\boldsymbol{\mathcal{V}}_{h}^{E} := \left\lbrace \textbf{v}_{h} \in \boldsymbol{\mathcal{W}}_{h}^{E} \; : \; \displaystyle{\int_{E} \boldsymbol{p}\cdot (\textbf{v}_{h} - \boldsymbol{\Pi}_{k,E}\textbf{v}_{h}) = 0, \forall \; \boldsymbol{p} \in [\mathbb{P}_{k}(E)]^{2}/[\mathbb{P}_{k-2}(E)]^{2}} \right\rbrace$,
\end{center}
where  the space $[\mathbb{P}_{k}(E)]^{2}/[\mathbb{P}_{k-2}(E)]^{2}$ denotes the polynomials in $[\mathbb{P}_{k}(E)]^{2}$ in   which are orthogonal to  $[\mathbb{P}_{k-2}(E)]^{2}$ with respect to the $\textbf{L}^{2}(E)$ product.
We choose the same degrees of freedom as those  in \cite[Section 4.1]{MR2997471} for the local virtual space defined above.
\end{definition}
Now we are in position to introduce the global virtual space which we define by 
\begin{equation*}
\boldsymbol{\mathcal{V}}_{h} := \{ \textbf{v}_{h} \in \textbf{H}_{0}^{1}(\Omega) \; : \; \textbf{v}_{h}\lvert_{E} \; \in \boldsymbol{\mathcal{V}}_{h}^{E}\}.
\end{equation*}
For the analysis of the VEM  allowing small edges, we introduce the following seminorm which is induced by the stability term  $\vertiii{\cdot}_{k,E} : \textbf{H}^{1}(E) \rightarrow \mathbb{R}$ defined by (see \cite[equation (3.9)]{MR3815658})
\begin{equation}\label{triplenorm}
\vertiii{\boldsymbol{\zeta}}_{k,E}^{2} := \|\bPi_{k-2,E}^{0}\boldsymbol{\zeta}\|_{0,E}^{2} + h_{E}\displaystyle{\sum_{e \in  \mathcal{E}_{E}} \|\bPi_{k-1,e}^{0}\boldsymbol{\zeta}\|_{0,e}^{2}} \quad \forall \boldsymbol{\zeta} \in \textbf{H}^{1}(E),
\end{equation} 
where  $\mathcal{E}_{E}$ represents the set of edges of $E$, $\bPi_{k-1,e}^{0}$ is the  $\textbf{L}^{2}(e)$ orthogonal projection onto the space $[\mathbb{P}_{k-1}(e)]^{2}$ and $\bPi_{k-2,E}^{0}$ is the $\textbf{L}^{2}(E)$ orthogonal projection onto the space $[\mathbb{P}_{k-2}(E)]^{2}$. This seminorm will appear naturally when  we operate to  control the approximation errors of the virtual solutions, in particular to bound the norms of the polynomial projectors and  the interpolations of the solutions. 
From the following trace inequality
\begin{equation}\label{traceineq0}
\|\boldsymbol{\zeta}\|_{0,\partial E}^{2} \lesssim h_{E}^{-1}\|\boldsymbol{\zeta}\|_{0,E}^{2} + h_{E}|\boldsymbol{\zeta}|_{1,E}^{2} \quad \forall \boldsymbol{\zeta} \in \textbf{H}^{1}(E),
\end{equation}
and  \eqref{triplenorm}, there holds that  
\begin{equation*}
\vertiii{\boldsymbol{\zeta}}_{k,E} \lesssim \|\boldsymbol{\zeta}\|_{0,E} + h_{E}|\boldsymbol{\zeta}|_{1,E} \quad \forall \boldsymbol{\zeta} \in \textbf{H}^{1}(E).
\end{equation*}
On the other hand, from Korn's inequality and the fact that $a(\cdot,\cdot)$ is bounded, we have
\begin{equation}\label{cota1}
|\boldsymbol{\Pi}_{k,E}\textbf{v}_{h}|_{1,E} \lesssim \max\{\lambda_{S}\mu_{S}^{-1}, 1\}|\textbf{v}_{h}|_{1,E} \quad \forall \textbf{v}_{h} \in \boldsymbol{\mathcal{V}}_{h}^{E}.
\end{equation}
The following result is an adaptation of \cite[Lemma 3.4]{MR3815658} to our case, and gives an estimate for $\boldsymbol{\Pi}_{k,E}$ in terms of $\vertiii{\cdot}_{k,E}$.
\begin{lemma}\label{lema34}
There holds 
\begin{equation*}
\|\boldsymbol{\Pi}_{k,E}\boldsymbol{\zeta}\|_{0,E} \lesssim \max\{\lambda_{S}\mu_{S}^{-1}, 1\}\vertiii{\boldsymbol{\zeta}}_{k,E} \quad \forall \boldsymbol{\zeta} \in \textbf{H}^{1}(E),
\end{equation*}
where the hidden constant is independent of $h_E$.
\end{lemma}
\begin{proof}
Let  $\boldsymbol{\zeta} \in \textbf{H}^{1}(E)$. Then, using the definition of $\boldsymbol{\Pi}_{k,E}$, integration by parts, the definitions of $\boldsymbol{\Pi}_{k-2,E}^{0}$ and $\boldsymbol{\Pi}_{k-1,e}^{0}$, inverse estimates and trace inequalities, we have
\begin{align*}
|\boldsymbol{\Pi}_{k,E}\boldsymbol{\zeta}|_{1,E}^{2} &\lesssim \displaystyle{\int_{E} \boldsymbol{\varepsilon}(\boldsymbol{\Pi}_{k,E}\boldsymbol{\zeta}) : \boldsymbol{\varepsilon}(\boldsymbol{\Pi}_{k,E}\boldsymbol{\zeta})} \\ 
&\lesssim \mu_{S}^{-1}\displaystyle{\int_{E} \boldsymbol{\sigma}(\boldsymbol{\Pi}_{k,E}\boldsymbol{\zeta}) : \boldsymbol{\varepsilon}(\boldsymbol{\Pi}_{k,E}\boldsymbol{\zeta})} = \mu_{S}^{-1}\displaystyle{\int_{E} \boldsymbol{\sigma}(\boldsymbol{\Pi}_{k,E}\boldsymbol{\zeta}) : \boldsymbol{\varepsilon}(\boldsymbol{\zeta})} \\
&= \mu_{S}^{-1}\left(-\displaystyle{\int_{E} \boldsymbol{\Pi}_{k-2,E}^{0}\boldsymbol{\zeta}\cdot\textbf{div}(\boldsymbol{\sigma}(\boldsymbol{\Pi}_{k,E}\boldsymbol{\zeta})) + \sum_{e \in \mathcal{E}_{E}} \int_{e} (\boldsymbol{\sigma}(\boldsymbol{\Pi}_{k,E}\boldsymbol{\zeta}) \cdot \textbf{n}_{E})\cdot\boldsymbol{\Pi}_{k-1,e}^{0}\boldsymbol{\zeta}} \right) \\
&\leq \mu_{S}^{-1}\left( \|\boldsymbol{\Pi}_{k-2,E}^{0}\boldsymbol{\zeta}\|_{0,E}\|\textbf{div}(\boldsymbol{\sigma}(\boldsymbol{\Pi}_{k,E}\boldsymbol{\zeta}))\|_{0,E} + \displaystyle{\sum_{e \in \mathcal{E}_{E}} \|\boldsymbol{\Pi}_{k-1,e}^{0}\boldsymbol{\zeta}\|_{0,e}\|\boldsymbol{\sigma}(\boldsymbol{\Pi}_{k,E}\boldsymbol{\zeta})\|_{0,e}} \right) \\
&\lesssim h_{E}^{-1}\mu_{S}^{-1}\|\boldsymbol{\sigma}(\boldsymbol{\Pi}_{k,E}\boldsymbol{\zeta})\|_{0,E}\left( \|\boldsymbol{\Pi}_{k-2,E}^{0}\boldsymbol{\zeta}\|_{0,E} + h_{E}^{1/2}\left(\displaystyle{\sum_{e \in \mathcal{E}_{E}} \|\boldsymbol{\Pi}_{k-1,e}^{0}\boldsymbol{\zeta}\|_{0,e}^{2}}\right)^{1/2} \right)\\
&\lesssim h_{E}^{-1}\mu_{S}^{-1}\|\boldsymbol{\sigma}(\boldsymbol{\Pi}_{k,E}\boldsymbol{\zeta})\|_{0,E}\vertiii{\boldsymbol{\zeta}}_{k,E}.
\end{align*}
From the definition of $\boldsymbol{\sigma}$ we have 
\begin{equation*}
\|\boldsymbol{\sigma}(\boldsymbol{\Pi}_{k,E}\boldsymbol{\zeta})\|_{0,E}
 \lesssim \max\{\lambda_{S}, \mu_{S}\}\|\boldsymbol{\varepsilon}(\boldsymbol{\Pi}_{k,E}\boldsymbol{\zeta})\|_{0,E},
\end{equation*}
and invoking the estimate  $\|\boldsymbol{\varepsilon}(\boldsymbol{\Pi}_{k,E}\boldsymbol{\zeta})\|_{0,E} \lesssim |\boldsymbol{\Pi}_{k,E}\boldsymbol{\zeta}|_{1,E}$ we obtain directly the following estimate for the seminorm of the projection
\begin{equation}\label{parte1}
|\boldsymbol{\Pi}_{k,E}\boldsymbol{\zeta}|_{1,E} \lesssim h_{E}^{-1}\max\{\lambda_{S}\mu_{S}^{-1},1\}\vertiii{\boldsymbol{\zeta}}_{k,E}.
\end{equation}
On the other hand, notice that from the definition of $\vertiii{\cdot}_{k,E}$, the definition of $\boldsymbol{\Pi}_{k,E}$ and the fact that $|\partial E| \simeq h_{E}$, we have
\begin{equation}\label{parte2}
\left\lvert \displaystyle{\int_{\partial E} \boldsymbol{\Pi}_{k,E}\boldsymbol{\zeta} } \right\lvert = \left\lvert \displaystyle{\int_{\partial E}\boldsymbol{\zeta} } \right\lvert = \left\lvert \displaystyle{\sum_{e \in \mathcal{E}_{E}} \int_{e} \boldsymbol{\Pi}_{0,e}^{0}\boldsymbol{\zeta} } \right\lvert \lesssim \left(h_{E}\displaystyle{\sum_{e \in \mathcal{E}_{E}} \|\boldsymbol{\Pi}_{k-1,e}^{0}\boldsymbol{\zeta}\|_{0,e}^{2}} \right)^{1/2} \leq \vertiii{\boldsymbol{\zeta}}_{k,E},
\end{equation}
where in the first inequality, we  have used Cauchy-Schwarz inequality and the fact that $\|\boldsymbol{\Pi}_{0,e}^{0}\boldsymbol{\zeta}\|_{0,e} \leq \|\boldsymbol{\Pi}_{k-1,e}^{0}\boldsymbol{\zeta}\|_{0,e}$. Hence, from the following  Poincar\'e-Friedrich inequality 
\begin{equation}\label{FP}
\|\boldsymbol{\zeta}\|_{0,E} \lesssim \left\lvert \displaystyle{\int_{\partial E} \boldsymbol{\zeta} } \right\lvert + h_{E}|\boldsymbol{\zeta}|_{1,E} \quad \forall \boldsymbol{\zeta} \in \textbf{H}^{1}(E),
\end{equation}
and using \eqref{parte1} and \eqref{parte2}, we obtain 
\begin{equation*}
\|\boldsymbol{\Pi}_{k,E}\boldsymbol{\zeta}\|_{0,E} \lesssim \left\lvert \displaystyle{\int_{\partial E} \boldsymbol{\Pi}_{k,E}\boldsymbol{\zeta} } \right\lvert + h_{E}|\boldsymbol{\Pi}_{k,E}\boldsymbol{\zeta}|_{1,E} \lesssim \max\{\lambda_{S}\mu_{S}^{-1},1\}\vertiii{\boldsymbol{\zeta}}_{k,E}.
\end{equation*}
This concludes the proof.
\end{proof}
Now we need approximation properties for $\boldsymbol{\Pi}_{k,E}$. In the following result we derive such an estimates, that are an adaptation of those in \cite[Lemma 3.5]{MR3815658}, in which we emphasize that the estimates depends on the Lam\'e coefficients.
\begin{lemma}\label{lema35}
The following estimates hold
\begin{itemize}
\item[i)] For all $\boldsymbol{\zeta} \in \textbf{H}^{\ell + 1}(E)$ and  $0 \leq \ell \leq k$, there holds 
$$\|\boldsymbol{\zeta} - \boldsymbol{\Pi}_{k,E}\boldsymbol{\zeta}\|_{0,E} \lesssim \max\{\lambda_{S}\mu_{S}^{-1}, 1\}h_{E}^{\ell + 1}|\boldsymbol{\zeta}|_{\ell + 1,E},$$
\item[ii)]  For all  $\boldsymbol{\zeta} \in \textbf{H}^{\ell + 1}(E)$ and  $1 \leq \ell \leq k$, there holds 
$$|\boldsymbol{\zeta} - \boldsymbol{\Pi}_{k,E}\boldsymbol{\zeta}|_{1,E} \lesssim \max\{\lambda_{S}\mu_{S}^{-1},1\}h_{E}^{\ell}|\boldsymbol{\zeta}|_{\ell + 1,E},$$
\item[iii)]  For all $\boldsymbol{\zeta} \in \textbf{H}^{\ell + 1}(E)$ and  $1 \leq \ell \leq k$, there holds
$$|\boldsymbol{\zeta} - \boldsymbol{\Pi}_{k,E}\boldsymbol{\zeta}|_{2,E} \lesssim \max\{\lambda_{S}\mu_{S}^{-1},1\}h_{E}^{\ell - 1}|\boldsymbol{\zeta}|_{\ell + 1,E},$$
\end{itemize}
where in each estimate,  the hidden constants are independent of $h_E$.
\end{lemma}
\begin{proof}
Let us recall the following Bramble-Hilbert estimate 
\begin{equation}\label{bramblehilbert}
\underset{\boldsymbol{q} \in [\mathbb{P}_{\ell}(E)]^{2}}{\inf} |\boldsymbol{\zeta} - \boldsymbol{q}|_{m,E} \lesssim h_{E}^{\ell + 1 - m}|\boldsymbol{\zeta}|_{\ell + 1,E} \quad \forall \boldsymbol{\zeta} \in \textbf{H}^{\ell + 1}(E).
\end{equation}
We begin by proving  $\text{ii)}$. To do this task,  given $\boldsymbol{q} \in [\mathbb{P}_{\ell}(E)]^{2}$ arbitrary, and invoking estimate \eqref{cota1} we have 
\begin{equation*}
|\boldsymbol{\Pi}_{k,E}\boldsymbol{\zeta} - \boldsymbol{q}|_{1,E} = |\boldsymbol{\Pi}_{k,E}(\boldsymbol{\zeta} - \boldsymbol{q})|_{1,E} \leq \max\{\lambda_{S}\mu_{S}^{-1},1\}|\boldsymbol{\zeta} - \boldsymbol{q}|_{1,E}.
\end{equation*}
From triangle inequality and   \eqref{bramblehilbert},  we have 
\begin{equation*}
\begin{split}
|\boldsymbol{\zeta} - \boldsymbol{\Pi}_{k,E}\boldsymbol{\zeta}|_{1,E} 
&\lesssim \max\{\lambda_{S}\mu_{S}^{-1},1\}h_{E}^{\ell}|\boldsymbol{\zeta}|_{\ell + 1,E}.
\end{split}
\end{equation*}
Now to derive  $\text{i)}$, invoking  \eqref{FP},  $\text{ii)}$ and the definition of $\boldsymbol{\Pi}_{k,E}$, for each  $\boldsymbol{\zeta} \in \textbf{H}^{\ell + 1}(E)$ we have 
\begin{multline*}
\|\boldsymbol{\zeta} - \boldsymbol{\Pi}_{k,E}\boldsymbol{\zeta}\|_{0,E} \lesssim \left\lvert \displaystyle{\int_{\partial E} \left( \boldsymbol{\zeta} - \boldsymbol{\Pi}_{k,E}\boldsymbol{\zeta}\right)}\right\lvert + h_{E}|\boldsymbol{\zeta} - \boldsymbol{\Pi}_{k,E}\boldsymbol{\zeta}|_{1,E} \\
= h_{E}|\boldsymbol{\zeta} - \boldsymbol{\Pi}_{k,E}\boldsymbol{\zeta}|_{1,E} \lesssim \max\{\lambda_{S}\mu_{S}^{-1},1\}h_{E}^{\ell + 1}|\boldsymbol{\zeta}|_{\ell + 1,E}.
\end{multline*}
Finally, from the following polynomial approximation property (see \cite[(4.5.3) Lemma]{MR2373954})
\begin{equation}\label{approxpoly3}
|\boldsymbol{p}|_{1,\Omega} \lesssim h_{E}^{-1}\|\boldsymbol{p}\|_{0,E} \quad \forall \boldsymbol{p} \in [\mathbb{P}_{k}(E)]^{2},
\end{equation}
together with  \eqref{cota1}, yields to 
\begin{equation*}
|\boldsymbol{\Pi}_{k,E}\boldsymbol{\zeta}|_{2,E} \lesssim h_{E}^{-1}|\boldsymbol{\Pi}_{k,E}\boldsymbol{\zeta}|_{1,E} \lesssim \max\{\lambda_{S}\mu_{S}^{-1},1\}h_{E}^{-1}|\boldsymbol{\zeta}|_{1,E}.
\end{equation*}
Then, given $\boldsymbol{q} \in [\mathbb{P}_{k}(E)]^{2}$ arbitrary, triangle inequality  and  \eqref{bramblehilbert}, we obtain
\begin{equation*}
\begin{split}
|\boldsymbol{\zeta} - \boldsymbol{\Pi}_{k,E}\boldsymbol{\zeta}|_{2,E} &\lesssim \max\{\lambda_{S}\mu_{S}^{-1},1\}h_{E}^{\ell - 1}|\boldsymbol{\zeta}|_{\ell + 1,E},
\end{split}
\end{equation*}
concluding $\text{iii)}$ and hence,  the  proof.
\end{proof}
In the previous result, we have made clear the dependence of the Lam\'e constants on the estimates. This is important to take into account, since in particular $\lambda_S$ leads to numerical locking. Now, we are going to analyze stability and approximation properties for the $\textbf{L}^{2}(E)$ orthogonal projector $\boldsymbol{\Pi}_{k,E}^{0}$ onto the space $[\mathbb{P}_{k}(E)]^{2}$.
\begin{remark}
%
%
The following estimate for  $\boldsymbol{\Pi}_{k,E}^{0}$ holds 
\begin{equation}\label{stimaproye1}
\|\boldsymbol{\Pi}_{k,E}^{0}\boldsymbol{\zeta}\|_{0,E} \leq \|\boldsymbol{\zeta}\|_{0,E}.
\end{equation}
Moreover, since  $\boldsymbol{\Pi}_{k,E}^{0}\boldsymbol{p} = \boldsymbol{p}$ for all $\boldsymbol{p} \in [\mathbb{P}_{k}(E)]^{2}$, as a consequence of  \eqref{bramblehilbert} there holds 
\begin{equation}\label{stimaproye2}
\|\boldsymbol{\zeta} - \boldsymbol{\Pi}_{k,E}^{0}\boldsymbol{\zeta}\|_{0,E} \lesssim h_{E}^{\ell + 1}|\boldsymbol{\zeta}|_{\ell + 1,E} \quad \forall \boldsymbol{\zeta} \in \textbf{H}^{\ell + 1}(E), \quad 0 \leq \ell \leq k.
\end{equation}
\end{remark}
The following lemma is adapted from \cite[Lemma 3.6]{MR3815658} to our case, and states an estimate for the seminorm $|\cdot|_{1,E}$ of $\boldsymbol{\Pi}_{k,E}^{0}$. 
\begin{lemma}\label{lema36}
There holds
\begin{equation*}\label{stimaproye3}
|\boldsymbol{\Pi}_{k,E}^{0}\boldsymbol{\zeta}|_{1,E} \lesssim \max\{\lambda_{S}\mu_{S}^{-1},1\}|\boldsymbol{\zeta}|_{1,E} \quad \forall \boldsymbol{\zeta} \in \textbf{H}^{1}(E),
\end{equation*}
where the hidden constant is independent of $h_E$.
\end{lemma}
\begin{proof}
Let  $\boldsymbol{\zeta} \in \textbf{H}^{1}(E)$. From  triangle inequality, \eqref{approxpoly3}, \eqref{cota1}, and item  $\text{i)}$ of  Lemma \ref{lema35}, we obtain 
\begin{equation*}
\begin{split}
|\boldsymbol{\Pi}_{k,E}^{0}\boldsymbol{\zeta}|_{1,E} &\leq |\boldsymbol{\Pi}_{k,E}^{0}\boldsymbol{\zeta} - \boldsymbol{\Pi}_{k,E}\boldsymbol{\zeta}|_{1,E} + |\boldsymbol{\Pi}_{k,E}\boldsymbol{\zeta}|_{1,E} \\
&\lesssim h_{E}^{-1}\|\boldsymbol{\Pi}_{k,E}^{0}\boldsymbol{\zeta} - \boldsymbol{\Pi}_{k,E}\boldsymbol{\zeta}\|_{0,E} + \max\{\lambda_{S}\mu_{S}^{-1},1\}|\boldsymbol{\zeta}|_{1,E} \\
&\lesssim h_{E}^{-1}(\|\boldsymbol{\Pi}_{k,E}^{0}\boldsymbol{\zeta} - \boldsymbol{\zeta}\|_{0,E} +  \|\boldsymbol{\zeta} - \boldsymbol{\Pi}_{k,E}\boldsymbol{\zeta}\|_{0,E}) + \max\{\lambda_{S}\mu_{S}^{-1},1\}|\boldsymbol{\zeta}|_{1,E} \\
&\lesssim h_{E}^{-1}h_{E}|\boldsymbol{\zeta}|_{1,E} + \max\{\lambda_{S}\mu_{S}^{-1},1\}|\boldsymbol{\zeta}|_{1,E} \lesssim \max\{\lambda_{S}\mu_{S}^{-1},1\}|\boldsymbol{\zeta}|_{1,E},
\end{split}
\end{equation*}
which concludes the proof.
\end{proof}
The following result is adapted from \cite[Lemma 3.7]{MR3815658}, and gives us an approximation for $\boldsymbol{\Pi}_{k,E}^{0}$ in seminorm $|\cdot|_{1,E}$.
\begin{lemma}\label{lema37}
The following estimate holds
\begin{equation*}
|\boldsymbol{\zeta} - \boldsymbol{\Pi}_{k,E}^{0}\boldsymbol{\zeta}|_{1,E} \lesssim \max\{\lambda_{S}\mu_{S}^{-1},1\}h_{E}^{\ell}|\boldsymbol{\zeta}|_{1+\ell,E}, \quad 1 \leq \ell \leq k,
\end{equation*}
where the hidden constant is independent of $h_E$.
\end{lemma}
\begin{proof}
Let $1\leq\ell\leq k$  and  $\boldsymbol{\zeta} \in \textbf{H}^{\ell + 1}(E)$. Then, we obtain the result for a  given  $\boldsymbol{q} \in [\mathbb{P}_{\ell}(E)]^{2}$, triangle inequality, \eqref{bramblehilbert} and Lemma \ref{lema36}.
\end{proof}
Now we prove a stability result for $\textbf{L}^2$ norm of the projection $\boldsymbol{\Pi}_{k,E}^{0}$ with respect to the triple norm $\vertiii{\cdot}_{k,E}$, for virtual functions. A similar result can be found in  \cite[Lemma 3.8]{MR3815658} for the Laplace problem.
\begin{lemma}\label{lema38}
There holds 
\begin{equation*}
\|\boldsymbol{\Pi}_{k,E}^{0}\textbf{v}_{h}\|_{0,E} \lesssim \max\{\lambda_{S}\mu_{S}^{-1},1\}\vertiii{\textbf{v}_{h}}_{k,E} \quad \forall \textbf{v}_{h} \in \boldsymbol{\mathcal{V}}_{h}^{E},
\end{equation*}
where the hidden constant is independent of $h_{E}$.
\end{lemma}
\begin{proof}
Given  $\textbf{v}_{h} \in \boldsymbol{\mathcal{V}}_{h}^{E}$, from the definition of $\boldsymbol{\Pi}_{k-2,E}^{0}$ we have 
\begin{equation*}
\begin{split}
\displaystyle{\int_{E} \boldsymbol{\Pi}_{k-2,E}^{0}\textbf{v}_{h}\cdot(\boldsymbol{\Pi}_{k-2,E}^{0}\textbf{v}_{h} - \boldsymbol{\Pi}_{k,E}^{0}\textbf{v}_{h})} = 0.
\end{split}
\end{equation*}
Hence, from  Pithagoras Theorem, the definition of  $\boldsymbol{\mathcal{V}}_{h}^{E}$, \eqref{triplenorm}, and  Lemma \ref{lema34} we obtain
\begin{multline*}
\|\boldsymbol{\Pi}_{k,E}^{0}\textbf{v}_{h}\|_{0,E}^{2} = \|\boldsymbol{\Pi}_{k-2,E}^{0}\textbf{v}_{h}\|_{0,E}^{2} + \|(\boldsymbol{\Pi}_{k,E}^{0} - \boldsymbol{\Pi}_{k-2,E}^{0})\textbf{v}_{h}\|_{0,E}^{2} \\
= \|\boldsymbol{\Pi}_{k-2,E}^{0}\textbf{v}_{h}\|_{0,E}^{2} + \|(\boldsymbol{\Pi}_{k,E}^{0} - \boldsymbol{\Pi}_{k-2,E}^{0})\boldsymbol{\Pi}_{k,E}\textbf{v}_{h}\|_{0,E}^{2} \\
\lesssim \vertiii{\textbf{v}_{h}}_{k,E}^{2} + \max\{\lambda_{S}\mu_{S}^{-1},1\}^{2}\vertiii{\textbf{v}_{h}}_{k,E}^{2} 
\lesssim \max\{\lambda_{S}\mu_{S}^{-1},1\}^{2}\vertiii{\textbf{v}_{h}}_{k,E}^{2},
\end{multline*}
concluding the proof. 
\end{proof}
To prove the following result, we need to adapt  the arguments of the proof of \cite[Lemma 3.9]{MR3815658} for the vectorial case, taking into account that now, in the estimates, the Lam\'e constants appear naturally.
\begin{lemma}\label{lema39}
Assume that \textbf{A1} holds. Then, there exists a constant  $C > 0$ depending on  $\rho_{E}$ \textcolor{blue}{and}  $k$ such that 
\begin{equation*}\label{stima1}
|\textbf{v}_{h}|_{1,E} \leq C\max\{\lambda_{S}\mu_{S}^{-1},1\}(h_{E}^{-1}\vertiii{\textbf{v}_{h}}_{k,E} + |\textbf{v}_{h}|_{1/2,\partial E}) \quad \forall \textbf{v}_{h} \in \boldsymbol{\mathcal{V}}_{h}^{E}.
\end{equation*}
\end{lemma}
\begin{proof}
The proof is similar to what was done in \cite[Lemma 3.9]{MR3815658}, since the two-dimensional case is direct.
Then, we have
\begin{equation*}
\|\boldsymbol{p}\|_{0,E} \lesssim \|\boldsymbol{\Pi}_{k,E}^{0}\textbf{v}_{h} - \textbf{w}_{h}\|_{0,E},
\end{equation*}
where $\textbf{w}_{h} := \Upsilon(\textbf{v}_{h})$, the operator $\Upsilon$, called \textit{lifting operator}, was introduced in \cite[Section 2.7]{MR3815658}, and $\boldsymbol{p}$ was defined in the proof of \cite[Lemma 3.9]{MR3815658}.
Therefore, from triangle inequality and Lemma \ref{lema38} we have
\begin{equation}\label{lifting}
\|\boldsymbol{p}\|_{0,E} 
\lesssim \max\{\lambda_{S}\mu_{S}^{-1},1\}(\vertiii{\textbf{v}_{h}}_{k,E} + |\textbf{v}_{h}|_{1/2,\partial E}).
\end{equation}
Hence, we conclude that
\begin{equation}\label{lifting2}
\begin{split}
|\textbf{v}_{h}|_{1,E} \leq |\boldsymbol{\zeta}|_{1,E} &\lesssim |\textbf{v}_{h}|_{1/2, \partial E} + \|\boldsymbol{p}\|_{0,E} \\
 &\lesssim \max\{\lambda_{S}\mu_{S}^{-1},1\}(\vertiii{\textbf{v}_{h}}_{k,E} + |\textbf{v}_{h}|_{1/2,\partial E}),
\end{split}
\end{equation}
where $\boldsymbol{\zeta}$ was defined in the proof of \cite[Lemma 3.9]{MR3815658}, and \eqref{lifting2} is obtained from \cite[Lemma 3.9]{MR3815658}, together with  \eqref{lifting}.
\end{proof}
As a direct consequence of lemma above, we have the following result.
\begin{corollary}\label{lema310}
Assume that \textbf{A1} holds. Then, there exists  $C > 0$ depending on  $\rho_{E}$ y $k$, such that 
\begin{equation*}\label{stima2}
|\textbf{v}_{h}|_{1,E} \leq C\max\{\lambda_{S}\mu_{S}^{-1},1\}\left(h_{E}^{-1}\vertiii{\textbf{v}_{h}}_{k,E} + h_{E}^{1/2}\|\partial_{s}\textbf{v}_{h}\|_{0,\partial E}\right) \quad \forall \textbf{v}_{h} \in \boldsymbol{\mathcal{V}}_{h}^{E}.
\end{equation*}
where  $\partial_{s}\textbf{v}_{h}$ denotes the tangential derivative of  $\textbf{v}_{h}$.
\end{corollary}
\begin{proof}
To derive the result, it is enough to apply  Lemma \ref{lema39}, together with the inverse estimate 
\begin{equation*}\label{approx1}
|\boldsymbol{\zeta}|_{1/2,\partial E} \lesssim h_{E}^{1/2}\|\partial_{s}\boldsymbol{\zeta}\|_{0,\partial E} \quad \forall \boldsymbol{\zeta} \in \textbf{H}^{1}(\partial E).
\end{equation*}
\end{proof}
\subsection{The interpolation operator $\mathbf{I}_{k,E}$}
In order to obtain error estimates for our numerical method, and interpolation operator is needed. This operator is defined locally, and its construction is based in the references \cite{MR3340705} and \cite{MR3714637}, where this last reference makes the construction of this operator for the small edges approach.

We consider for $s > 1$ the interpolation operator $\mathbf{I}_{k,E} : \textbf{H}^{s}(E) \longrightarrow \boldsymbol{\mathcal{V}}_{h}^{E}$, which satisfies that the degrees of freedom of $\boldsymbol{\zeta}$ and $\mathbf{I}_{k,E}\boldsymbol{\zeta}$ are the same, together with the identity $\boldsymbol{\Pi}_{k-2,E}^{0}\textbf{I}_{k,E}\boldsymbol{\zeta} = \boldsymbol{\Pi}_{k-2,E}^{0}\boldsymbol{\zeta}$
and $\mathbf{I}_{k,E}\boldsymbol{q} = \boldsymbol{q}\quad \forall \; \boldsymbol{q} \in [\mathbb{P}_{k}(E)]^{2}.$

The following result is a stability estimate for $\textbf{I}_{k,E}$ in $\vertiii{\cdot}_{k,E}$ norm and its proof is  straightforward from \cite[Lemma 3.12]{MR3815658}.
\begin{lemma}\label{lema312}
Assume that \textbf{A1} holds. Then, there holds 
\begin{equation*}
\vertiii{\mathbf{I}_{k,E}\boldsymbol{\zeta}}_{k,E} \lesssim \left(\|\boldsymbol{\zeta}\|_{0,E} + h_{E}|\boldsymbol{\zeta}|_{1,E} + h_{E}^{2}|\boldsymbol{\zeta}|_{2,E}\right) \quad \forall \boldsymbol{\zeta} \in \textbf{H}^{2}(E),
\end{equation*}
where the hidden constant depends on $\rho_{E}$ and $k$,  but not on $h_E$.
\end{lemma}

We have the following estimate for $|\mathbf{I}_{k,E}\boldsymbol{\zeta}|_{1,E}$. This result is adapted from \cite[Lemma 3.13]{MR3815658} to our case.

\begin{lemma}\label{lema313}
The following estimate hold
\begin{equation*}
|\mathbf{I}_{k,E}\boldsymbol{\zeta}|_{1,E} \lesssim \max\{\lambda_{S}\mu_{S}^{-1},1\}\left(|\boldsymbol{\zeta}|_{1,E} + h_{E}|\boldsymbol{\zeta}|_{2,E}\right) \quad \forall \boldsymbol{\zeta} \in \textbf{H}^{2}(E),
\end{equation*}
where the hidden constant depends on $\rho_{E}$ and $k$.
\end{lemma}

\begin{proof}
Given $\boldsymbol{\zeta} \in \textbf{H}^{2}(E)$, let us define 
\begin{equation*}
\overline{\boldsymbol{\zeta}}_{E} := \dfrac{1}{|\partial E|}\displaystyle{\int_{\partial E} \boldsymbol{\zeta}}.
\end{equation*}
Invoking Corollary~\ref{lema310} and Lemma~\ref{lema312} we obtain
\begin{multline*}
|\mathbf{I}_{k,E}\boldsymbol{\zeta}|_{1,E}^{2} = |\mathbf{I}_{k,E}(\boldsymbol{\zeta} - \overline{\boldsymbol{\zeta}}_{E})|_{1,E}^{2} \\
\lesssim \max\{\lambda_{S}\mu_{S}^{-1},1\}^{2}\left( h_{E}^{-2}\vertiii{\mathbf{I}_{k,E}(\boldsymbol{\zeta} - \overline{\boldsymbol{\zeta}}_{E})}_{k,E}^{2} + h_{E}\|\partial_{s}\mathbf{I}_{k,E}\boldsymbol{\zeta}\|_{0,E}^{2} \right) \\
= \max\{\lambda_{S}\mu_{S}^{-1},1\}^{2}\left(h_{E}^{-2}\|\boldsymbol{\zeta} - \overline{\boldsymbol{\zeta}}_{E}\|_{0,E}^{2} + |\boldsymbol{\zeta}|_{1,E}^{2} + h_{E}^{2}|\boldsymbol{\zeta}|_{2,E} + h_{E}\|\partial_{s}\mathbf{I}_{k,E}\boldsymbol{\zeta}\|_{0,\partial E}^{2}\right).
\end{multline*}
Now, from \eqref{traceineq0}
and standard interpolant estimates  we obtain
\begin{equation*}
h_{E}\|\partial_{s}\mathbf{I}_{k,E}\boldsymbol{\zeta}\|_{0,\partial E}^{2} \lesssim \displaystyle{\sum_{e \in \mathcal{E}_{E}} h_{E}\|\partial_{s}\boldsymbol{\zeta}\|_{0,e}^{2}} \lesssim |\boldsymbol{\zeta}|_{1,E}^{2} + h_{E}^{2}|\boldsymbol{\zeta}|_{2,E}.
\end{equation*}
Hence, using the estimate $\|\boldsymbol{\zeta} - \overline{\boldsymbol{\zeta}}_{E}\|_{0,E} \lesssim h_{E}|\boldsymbol{\zeta}|_{1,E}$, we obtain
\begin{equation*}
|\mathbf{I}_{k,E}\boldsymbol{\zeta}|_{1,E}^{2} \lesssim \max\{\lambda_{S}\mu_{S}^{-1},1\}^{2}\left(|\boldsymbol{\zeta}|_{1,E}^{2} + h_{E}^{2}|\boldsymbol{\zeta}|_{2,E}^{2}\right).
\end{equation*}
This concludes the proof.
\end{proof}

We also have a stability result for $\mathbf{I}_{k,E}\boldsymbol{\zeta}$ in $\|\cdot\|_{0,E}$ norm, which is adapted from \cite[Lemma 3.14]{MR3815658}.

\begin{lemma}\label{lema314}
The following estimate hold 
\begin{equation*}
\|\mathbf{I}_{k,E}\boldsymbol{\zeta}\|_{0,E} \lesssim \max\{\lambda_{S}\mu_{S}^{-1},1\}^{2}\left(\|\boldsymbol{\zeta}\|_{0,E} + h_{E}|\boldsymbol{\zeta}|_{1,E} + h_{E}^{2}|\boldsymbol{\zeta}|_{2,E}\right) \quad \forall \boldsymbol{\zeta} \in \textbf{H}^{2}(E),
\end{equation*}
where the hidden constant depends on $\rho_{E}$ and $k$.
\end{lemma}
\begin{proof}
Let $\boldsymbol{\zeta} \in \textbf{H}^{2}(\Omega)$. From triangle inequality we obtain
\begin{equation}\label{dt}
\|\mathbf{I}_{k,E}\boldsymbol{\zeta}\|_{0,E} \leq \|\mathbf{I}_{k,E}\boldsymbol{\zeta} - \bPi_{k,E}\mathbf{I}_{k,E}\boldsymbol{\zeta}\|_{0,E} + \|\bPi_{k,E}\mathbf{I}_{k,E}\boldsymbol{\zeta}\|_{0,E}. 
\end{equation}
Now the aim is to estimate each term on the right hand side of \eqref{dt}. Invoking item  $\text{i)}$ of Lemma \ref{lema35} with  $\ell=0$ and Lemma \ref{lema313}, we have
\begin{equation}\label{x1}
\begin{split}
\|\mathbf{I}_{k,E}\boldsymbol{\zeta} - \bPi_{k,E}\mathbf{I}_{k,E}\boldsymbol{\zeta}\|_{0,E} &\lesssim \max\{\lambda_{S}\mu_{S}^{-1},1\}h_{E}|\mathbf{I}_{k,E}\boldsymbol{\zeta}|_{1,E} \\
&\lesssim \max\{\lambda_{S}\mu_{S}^{-1},1\}^{2}\left(h_{E}|\boldsymbol{\zeta}|_{1,E} + h_{E}^{2}|\boldsymbol{\zeta}|_{2,E}\right).
\end{split}
\end{equation}
On the other hand, from Lemmas \ref{lema34} and \ref{lema312} we have
\begin{equation}\label{x2}
\begin{split}
\|\bPi_{k,E}\mathbf{I}_{k,E}\boldsymbol{\zeta}\|_{0,E} &\lesssim \max\{\lambda_{S}\mu_{S}^{-1},1\}\vertiii{\mathbf{I}_{k,E}\boldsymbol{\zeta}}_{k,E} \\
&\lesssim \max\{\lambda_{S}\mu_{S}^{-1},1\}\left(\|\boldsymbol{\zeta}\|_{0,E} + h_{E}|\boldsymbol{\zeta}|_{1,E} + h_{E}^{2}|\boldsymbol{\zeta}|_{2,E}\right).
\end{split}
\end{equation}
Finally, replacing \eqref{x1} and \eqref{x2} in \eqref{dt}, we conclude the proof.
\end{proof}

The following result establishes approximation properties for  $\mathbf{I}_{k,E}$. This results are adapted from those in \cite[Lemma 3.15]{MR3815658}, and we again emphasize that the Lam\'e coefficients appears in our estimates.
\begin{lemma}\label{lema315}
For $1 \leq \ell \leq k$ and  $\boldsymbol{\zeta} \in \textbf{H}^{\ell + 1}(E)$, the following estimates hold 
\begin{itemize}
\item[i)] $\|\boldsymbol{\zeta} - \mathbf{I}_{k,E}\boldsymbol{\zeta}\|_{0,E} + \|\boldsymbol{\zeta} - \boldsymbol{\Pi}_{k,E}\mathbf{I}_{k,E}\boldsymbol{\zeta}\|_{0,E} \lesssim \max\{\lambda_{S}\mu_{S}^{-1},1\}^{2}h_{E}^{\ell + 1}|\boldsymbol{\zeta}|_{\ell + 1,E}$,
\item[ii)] $|\boldsymbol{\zeta} - \mathbf{I}_{k,E}\boldsymbol{\zeta}|_{1,E} + |\boldsymbol{\zeta} - \boldsymbol{\Pi}_{k,E}\mathbf{I}_{k,E}\boldsymbol{\zeta}|_{1,E} \lesssim \max\{\lambda_{S}\mu_{S}^{-1},1\}h_{E}^{\ell}|\boldsymbol{\zeta}|_{\ell + 1,\Omega}$,
\item[iii)] $|\boldsymbol{\zeta} - \boldsymbol{\Pi}_{k,E}\mathbf{I}_{k,E}\boldsymbol{\zeta}|_{2,E} \lesssim \max\{\lambda_{S}\mu_{S}^{-1},1\}^{2}h_{E}^{\ell - 1}|\boldsymbol{\zeta}|_{\ell + 1,E}$,
\end{itemize}
where the hidden constants depend on $\rho_{E}$ and $k$. 
\end{lemma}

\begin{proof}
First we prove $\text{i)}$. We observe that, from Lemmas \ref{lema34}, \ref{lema312} and \ref{lema313}, we have the estimate 
\begin{equation*}
\begin{split}
\|\mathbf{I}_{k,E}\boldsymbol{\zeta}\|_{0,E} + \|\boldsymbol{\Pi}_{k,E}\mathbf{I}_{k,E}\boldsymbol{\zeta}\|_{0,E} &\lesssim \|\mathbf{I}_{k,E}\boldsymbol{\zeta}\|_{0,E} + \max\{\lambda_{S}\mu_{S}^{-1},1\}\vertiii{\mathbf{I}_{k,E}\boldsymbol{\zeta}}_{k,E} \\
&\lesssim \max\{\lambda_{S}\mu_{S}^{-1},1\}^{2}\left(\|\boldsymbol{\zeta}\|_{0,E} + h_{E}|\boldsymbol{\zeta}|_{1,E} + h_{E}^{2}|\boldsymbol{\zeta}|_{2,E}\right).
\end{split}
\end{equation*}
Therefore, applying triangular inequality and invoking \eqref{bramblehilbert}, we conclude that
\begin{equation*}
\|\boldsymbol{\zeta} - \mathbf{I}_{k,E}\boldsymbol{\zeta}\|_{0,E} + \|\boldsymbol{\zeta} - \boldsymbol{\Pi}_{k,E}\mathbf{I}_{k,E}\boldsymbol{\zeta}\|_{0,E} \lesssim \max\{\lambda_{S}\mu_{S}^{-1},1\}^{2}h_{E}^{\ell + 1}|\boldsymbol{\zeta}|_{\ell + 1,E}.
\end{equation*}
To  prove $\text{ii)}$, we invoke  \eqref{cota1} and Lemma \ref{lema313}, in order to obtain 
\begin{equation*}
|\mathbf{I}_{k,E}\boldsymbol{\zeta}|_{1,E} + |\boldsymbol{\Pi}_{k,E}\mathbf{I}_{k,E}\boldsymbol{\zeta}|_{1,E} 
\lesssim \max\{\lambda_{S}\mu_{S}^{-1},1\}\left(|\boldsymbol{\zeta}|_{1,E} + h_{E}|\boldsymbol{\zeta}|_{2,E}\right).
\end{equation*}
Hence, applying triangular inequality and invoking \eqref{bramblehilbert}, we obtain
\begin{equation*}
|\boldsymbol{\zeta} - \mathbf{I}_{k,E}\boldsymbol{\zeta}|_{1,E} + |\boldsymbol{\zeta} - \boldsymbol{\Pi}_{k,E}\mathbf{I}_{k,E}\boldsymbol{\zeta}|_{1,E} \lesssim \max\{\lambda_{S}\mu_{S}^{-1},1\}h_{E}^{\ell}|\boldsymbol{\zeta}|_{\ell + 1,E}.
\end{equation*}
Finally, for  $\text{iii)}$, we use  \eqref{approxpoly3}, \eqref{cota1} and Lemma \ref{lema313} to obtain 
\begin{equation*}
|\boldsymbol{\Pi}_{k,E}\mathbf{I}_{k,E}\boldsymbol{\zeta}|_{2,E} \lesssim h_{E}^{-1}|\boldsymbol{\Pi}_{k,E}\mathbf{I}_{k,E}\boldsymbol{\zeta}|_{1,E}
\lesssim \max\{\lambda_{S}\mu_{S}^{-1},1\}^{2}\left(h_{E}^{-1}|\boldsymbol{\zeta}|_{1,E} + |\boldsymbol{\zeta}|_{2,E}\right),
\end{equation*}
Therefore, applying triangular inequality and \eqref{bramblehilbert}, we have
\begin{equation*}
|\boldsymbol{\zeta} - \boldsymbol{\Pi}_{k,E}\mathbf{I}_{k,E}\boldsymbol{\zeta}|_{2,E} \lesssim \max\{\lambda_{S}\mu_{S}^{-1},1\}^{2}h_{E}^{\ell - 1}|\boldsymbol{\zeta}|_{\ell + 1,E}.
\end{equation*}
This concludes the proof.

\end{proof}
The following result is adapted from \cite[Lemma 3.16]{MR3815658} to our case, and gives us an approximation property for the interpolant operator   $\mathbf{I}_{k,E}\boldsymbol{\zeta}$, when  $\boldsymbol{\zeta} \in \textbf{H}^{2}(\Omega)$. We will need this results to demonstrate properties about the operator $F(\cdot)$ and its discrete counterpart.
\begin{lemma}\label{lema316}
The following estimate holds
\begin{equation*}
\|\mathbf{I}_{k,E}\boldsymbol{\zeta} - \boldsymbol{\Pi}_{1,E}^{0}\mathbf{I}_{k,E}\boldsymbol{\zeta}\|_{0,E} \lesssim  \max\{\lambda_{S}\mu_{S}^{-1},1\}^{2}h_{E}^{2}|\boldsymbol{\zeta}|_{2,E} \quad \forall \boldsymbol{\zeta} \in \textbf{H}^{2}(\Omega),
\end{equation*}
where the hidden constant depends on $\rho_{E}$ and $k$.
\end{lemma}
\begin{proof}
Let $\boldsymbol{\zeta} \in \textbf{H}^{2}(E)$. From Lemma \ref{lema314} and  \eqref{stimaproye1} we obtain
\begin{equation*}
\begin{split}
\|\textbf{I}_{k,E}\boldsymbol{\zeta}\|_{0,E} + \|\boldsymbol{\Pi}_{1,E}^{0}\textbf{I}_{k,E}\boldsymbol{\zeta}\|_{0,E} 
\lesssim \max\{\lambda_{S}\mu_{S}^{-1},1\}^{2}\left(\|\boldsymbol{\zeta}\|_{0,E} + h_{E}|\boldsymbol{\zeta}|_{1,E} + h_{E}^{2}|\boldsymbol{\zeta}|_{2,E}\right).
\end{split}
\end{equation*}
Hence, using triangle inequality and \eqref{bramblehilbert} with $\ell = 1$, we obtain
\begin{equation*}
\|\textbf{I}_{k,E}\boldsymbol{\zeta} - \boldsymbol{\Pi}_{1,E}^{0}\textbf{I}_{k,E}\boldsymbol{\zeta}\|_{0,E} \lesssim \max\{\lambda_{S}\mu_{S}^{-1},1\}^{2}h^{2}|\boldsymbol{\zeta}|_{2,E},
\end{equation*}
which concludes the proof.
\end{proof}

The following result is straightforward from \cite[Lemma 3.18]{MR3815658} which is essential  to prove that the discrete bilinear form $a_{h}(\cdot,\cdot)$ is elliptic in $\boldsymbol{\mathcal{V}}_{h}$.
\begin{lemma}\label{lema318}
The following estimate holds  
\begin{equation*}
\vertiii{\textbf{v}}_{k,E}^{2} \lesssim h_{E}\displaystyle{\sum_{e \in \mathcal{E}_{E}} \|\bPi_{k-1,e}^{0}\textbf{v}\|_{0,e}^{2}} \quad \forall \textbf{v} \; \text{such that} \,\,\boldsymbol{\Pi}_{k,E}\textbf{v} = 0,
\end{equation*}
where the hidden constant is independent of $h_E$
\end{lemma}
The following result establishes an estimate for  $\|\textbf{v}_{h}\|_{0,\partial E}$, with $\textbf{v}_{h} \in \mathbb{B}_{\partial E}$, and its proof is naturally extended to our vectorial case, from the proof of \cite[Lemma 3.19]{MR3815658}.
\begin{lemma}\label{lema319}
For every  $\textbf{v}_{h} \in \mathbb{B}_{\partial E}$ that vanishes in some part of  $\partial E$, there holds 
\begin{equation*}\label{stima5}
\|\textbf{v}_{h}\|_{0,\partial E} \lesssim h_{E}\|\partial_{s}\textbf{v}_{h}\|_{0,\partial E},
\end{equation*}
where  the hidden constant  depends only on  $k$.
\end{lemma}
\begin{remark}
Let  $\textbf{v} \in \textbf{H}^{1}(E)$ be such that  $\boldsymbol{\Pi}_{k,E}\textbf{v} = 0$. Then
\begin{center}
$\displaystyle{\int_{\partial E} \textbf{v}} = 0$,
\end{center}
which implies that  $\textbf{v}$ vanishes in some part of  $\partial E$. Then, invoking Corollary \ref{lema310}, Lemmas \ref{lema318} and  \ref{lema319}, we obtain that 
\begin{equation}\label{stimaxx}
|\textbf{v}|_{1,E} \lesssim  \max\{\lambda_{S}\mu_{S}^{-1},1\}h_{E}^{1/2}\|\partial_{s}\textbf{v}\|_{0,\partial E}.
\end{equation}
\end{remark}

Let us introduce the following stabilization term  $S^{E}(\cdot, \cdot)$ defined for $\mathbf{w}_{h},\mathbf{v}_{h}\in\boldsymbol{\mathcal{V}}_h$ by 
\begin{equation*}
S^{E}(\textbf{w}_{h}, \textbf{v}_{h}) := h_{E}\displaystyle{\int_{\partial E} \partial_{s}\textbf{w}_{h}\partial_{s}\textbf{v}_{h}},
\end{equation*}
which corresponds to a scaled inner product between $\partial_{s}\textbf{w}_{h}$ and $\partial_{s}\textbf{v}_{h}$ in $\textbf{L}^{2}(\partial E)$. Let us introduce the discrete bilinear form  $a_{h}(\cdot, \cdot):\boldsymbol{\mathcal{V}}_h\times \boldsymbol{\mathcal{V}}_h\rightarrow\mathbb{R}$ defined by 
\begin{equation*}
a_{h}(\textbf{w}_{h}, \textbf{v}_{h}) := \displaystyle{\sum_{E \in \mathcal{T}_{h}} \left[a^{E}(\boldsymbol{\Pi}_{k,E}\textbf{w}_{h}, \boldsymbol{\Pi}_{k,E}\textbf{v}_{h}) + S^{E}(\textbf{w}_{h} - \boldsymbol{\Pi}_{k,E}\textbf{w}_{h}, \textbf{v}_{h} - \boldsymbol{\Pi}_{k,E}\textbf{v}_{h})\right]}.
\end{equation*}
Now we  define the following discrete operator
\begin{equation}\label{Fdiscrete}
F_{h}(\textbf{v}_{h}) := \displaystyle{\int_{\Omega} \varrho\textbf{f}\cdot \Xi_{h}\textbf{v}_{h}},
\end{equation}
where $\Xi_{h} : \boldsymbol{\mathcal{V}}_{h} \longrightarrow \boldsymbol{\mathcal{P}}_{k,h}$ is defined by 
\begin{equation*}
\Xi_{h} = \left\{\begin{array}{lll}
\boldsymbol{\Pi}_{1,h}^{0} \quad &\text{if} \; k=1,2, \\
\boldsymbol{\Pi}_{k-2,h}^{0} \quad &\text{if} \; k \geq 3.
\end{array}\right.
\end{equation*}
where   $\boldsymbol{\mathcal{P}}_{k,h}$ is the piecewise discontinuous polynomial space of degree   $k$  respect to  $\mathcal{T}_{h}$ and $\boldsymbol{\Pi}_{k,h}^{0}$ is the global $\textbf{L}^{2}(\Omega)$-projector, such that $\boldsymbol{\Pi}_{k,h}^{0}|_{E} = \boldsymbol{\Pi}_{k,E}^{0}$. Finally, $\boldsymbol{\Pi}_{k,h}$ is the global projector with respect to $a(\cdot, \cdot)$, such that $\boldsymbol{\Pi}_{k,h}|_{E} = \boldsymbol{\Pi}_{k,E}$.

Now, we can write the virtual element discretization of Problem \ref{continuous}. 
\begin{prob}\label{discrette}
Given $\textbf{f} \in \textbf{L}^{2}(\Omega)$, find $\textbf{u}_{h} \in \boldsymbol{\mathcal{V}}_{h}$ such that

\begin{equation*}\label{discrete}
a_{h}(\textbf{u}_{h}, \textbf{v}_{h}) = F_{h}(\textbf{v}_{h}) \quad \forall \; \textbf{v}_{h} \in \boldsymbol{\mathcal{V}}_{h}.
\end{equation*}
\end{prob}

\begin{remark}
Observe that from triangle inequality and \eqref{stimaxx}, there holds 
\begin{equation*}
\begin{split}
|\textbf{v}_{h}|_{1,E}^{2} 
\lesssim \max\{\lambda_{S}^{2}\mu_{S}^{-2},\mu_{S}^{-1},1\}a_{h}^{E}(\textbf{v}_{h}, \textbf{v}_{h}).
\end{split} 
\end{equation*}
Moreover, the following estimate holds
\begin{equation}\label{stima8}
|\textbf{v}_{h}|_{1,\O}^{2} \lesssim \max\{\lambda_{S}^{2}\mu_{S}^{-2},\mu_{S}^{-1},1\}a_{h}(\textbf{v}_{h}, \textbf{v}_{h}).
\end{equation}
\end{remark}
It is easy to check that $a_{h}(\cdot,\cdot)$ is elliptic in $\boldsymbol{\mathcal{V}}_{h}$ and hence, we deduce that Problem \ref{discrette} has unique solution as a consequence of Lax-Milgram's Lemma.

Now for the functional $F(\cdot)$ defined in  \eqref{F} and its discrete counterpart  $F_{h}(\cdot)$ defined in  \eqref{Fdiscrete}, we have the following approximation result (see \cite[Lemma 4.1]{MR3815658}).
\begin{lemma}\label{lema41}
Let $1 \leq \ell \leq k$. The following estimates hold true 
\begin{itemize}
\item[i)]  For all $\textbf{f} \in \textbf{H}^{\ell - 1}(\Omega)$ and for all  $\textbf{w}_{h} \in \boldsymbol{\mathcal{V}}_{h}$,
$$|F(\textbf{w}_{h}) - F_{h}(\textbf{w}_{h})| \lesssim h^{\ell}|\textbf{f}|_{\ell - 1, \Omega}|\textbf{w}_{h}|_{1,\Omega}.$$
\item[ii)] For all $\textbf{f} \in \textbf{H}^{\ell - 1}(\Omega)$ and for all $\boldsymbol{\zeta} \in \textbf{H}^{2}(\Omega)$,
$$|F(\mathbf{I}_{k,h}\boldsymbol{\zeta}) - F_{h}(\mathbf{I}_{k,h}\boldsymbol{\zeta})| \lesssim \max\{\lambda_{S}\mu_{S}^{-1},1\}^{2}h^{\ell + 1}|\textbf{f}|_{\ell - 1, \Omega}|\boldsymbol{\zeta}|_{2,\Omega},$$
\end{itemize}
where the hidden constants are independent of $h$.
\end{lemma}
\begin{proof} 
The result follows from adapting the proof of \cite[Lemma 4.1]{MR3815658} together with the aid of Lemma \ref{lema316}.\end{proof}

To conclude this section, we introduce the broken $\textbf{H}^{1}$-seminorm with respect to $\mathcal{T}_{h}$, given by

\begin{equation*}
|\textbf{v}|_{1,h}^{2} := \displaystyle{\sum_{E \in \mathcal{T}_{h}} |\textbf{v}|_{1,E}^{2}} \quad \forall \; \textbf{v}_{h} \in \textbf{L}^{2}(\Omega) \; \text{such that} \; \textbf{v}|_{E} \in \textbf{H}^{1}(E).
\end{equation*}
\subsection{Error estimates} 
Now our aim is to obtain error estimates for the proposed method. Clearly, due the previous results, the error estimates will depend on the Lam\'e coefficients.  This fact gives us the hint that the error estimates that we can derive will not be optimal when the Poisson ratio is close to $1/2$.
%
%

We introduce the energy norm $\|\cdot\|_{h}$ defined by $\|\textbf{v}_{h}\|_{h}^{2} := a_{h}(\textbf{v}_{h},\textbf{v}_{h})$ $\forall \textbf{v}_{h} \in \boldsymbol{\mathcal{V}}_{h}$, and let us recall the following standard estimate (see  \cite{MR2373954} for instance)
\begin{equation}\label{stima10}
\|\textbf{u} - \textbf{u}_{h}\|_{h} \leq \underset{\textbf{v}_{h} \in \boldsymbol{\mathcal{V}}_{h}}{\inf} \|\textbf{u} - \textbf{v}_h\|_{h} + \underset{\textbf{v}_{h} \in \boldsymbol{\mathcal{V}}_{h}}{\sup} \dfrac{a_{h}(\textbf{u}, \textbf{v}_h) - F_{h}(\textbf{v}_h)}{\|\textbf{v}_h\|_{h}}.
\end{equation}
The main task now is to estimate correctly the supremum in \eqref{stima10}. With this aim in mind, we begin with the following result that states  an approximation for the elliptic projector $\boldsymbol{\Pi}_{k,h}$ in the energy norm.
\begin{remark}
For any  $\textbf{v}_{h} \in \boldsymbol{\mathcal{V}}_{h}$, from the definition of $\boldsymbol{\Pi}_{k,h}$ and $a_{h}(\cdot,\cdot)$, the following estimate holds
\begin{equation}\label{lemadanilo}
\|\textbf{v}_{h} - \boldsymbol{\Pi}_{k,h}\textbf{v}_{h}\|_{h} \leq \|\textbf{v}_{h}\|_{h} \quad \forall \textbf{v}_{h} \in \boldsymbol{\mathcal{V}}_{h}.
\end{equation}
\end{remark}
%
%
Let us recall estimate \eqref{stima10}. Now our task is to estimate each of the terms on its right hand side. With this goal in mind, let us begin by recalling that  from the definition of $\boldsymbol{\Pi}_{k,E}$ and for each  $\textbf{v}_{h} \in \boldsymbol{\mathcal{V}}_{h}$ we have 
\begin{equation*}
a_{h}(\textbf{u},\textbf{v}_{h}) = \displaystyle \sum_{E \in \mathcal{T}_{h}} \left[a^{E}(\boldsymbol{\Pi}_{k,E}\textbf{u} - \textbf{u}, \textbf{v}_{h} - \boldsymbol{\Pi}_{k,E}\textbf{v}_{h})
+ S^{E}(\textbf{u} - \boldsymbol{\Pi}_{k,E}\textbf{u}, \textbf{v}_{h} - \boldsymbol{\Pi}_{k,E}\textbf{v}_{h})\right] + F(\textbf{v}_{h}).
\end{equation*}
Therefore, applying the Cauchy Schwarz inequality, \eqref{lemadanilo} and the continuity of $a^{E}(\cdot,\cdot)$, we obtain
\begin{multline*}
a_{h}(\textbf{u},\textbf{v}_{h}) - F_{h}(\textbf{v}_{h}) = \displaystyle{\sum_{E \in \mathcal{T}_{h}} \left[a^{E}(\boldsymbol{\Pi}_{k,E}\textbf{u} - \textbf{u}, \textbf{v}_{h} - \boldsymbol{\Pi}_{k,E}\textbf{v}_{h})\right.} \\
+ \left. S^{E}(\textbf{u} - \boldsymbol{\Pi}_{k,E}\textbf{u}, \textbf{v}_{h} - \boldsymbol{\Pi}_{k,E}\textbf{v}_{h})\right] + F(\textbf{v}_{h}) - F_{h}(\textbf{v}_{h}) \\
\lesssim \max\{\lambda_{S},\mu_{S}\}\max\{\lambda_{S}\mu_{S}^{-1}\mu_{S}^{-1/2},1\}\left(|\boldsymbol{\Pi}_{k,E}\textbf{u} - \textbf{u}|_{1,h}\|\textbf{v}_{h}\|_{h}\right) + \|\textbf{u} - \bPi_{k,h}\textbf{u}\|_{h}\|\textbf{v}_{h}\|_{h} \\
+ \max\{\lambda_{S}\mu_{S}^{-1},\mu_{S}^{-1/2},1\}\boldsymbol{\mathcal{R}}_{h}\|\textbf{v}_{h}\|_{h}, 
\end{multline*}
where $\boldsymbol{\mathcal{R}}_{h} := \underset{\textbf{v}_{h} \in \boldsymbol{\mathcal{V}}_{h}}{\sup} \dfrac{F(\textbf{v}_h) - F_{h}(\textbf{v}_h)}{|\textbf{v}_h|_{1,\Omega}}$. Hence, we have
\begin{equation*}
\dfrac{a_{h}(\textbf{u}, \textbf{v}_{h}) - F_{h}(\textbf{v}_{h})}{\|\textbf{v}_{h}\|_{h}} \lesssim C_{1}(\lambda_{S},\mu_{S})\left(|\boldsymbol{\Pi}_{k,E}\textbf{u} - \textbf{u}|_{1,h} + \|\textbf{u} - \bPi_{k,h}\textbf{u}\|_{h} + \boldsymbol{\mathcal{R}}_{h}\right),
\end{equation*}
with $C_{1}(\lambda_{S},\mu_{S}) := \max\{\lambda_{S}\mu_{S}^{-1}\mu_{S}^{-1/2},1\}\max\{\lambda_{S},\mu_{S},1\}$. On the other hand, there holds
\begin{equation*}
\underset{\textbf{v}_{h} \in \boldsymbol{\mathcal{V}}_{h}}{\inf} \|\textbf{u} - \textbf{v}_{h}\|_{h} \leq \|\textbf{u} - \textbf{I}_{k,h}\textbf{u}\|_{h},
\end{equation*}
and therefore we obtain
\begin{equation}\label{stima11}
\|\textbf{u} - \textbf{u}_{h}\|_{h} \lesssim C_{1}(\lambda_{S},\mu_{S})\left(|\boldsymbol{\Pi}_{k,h}\textbf{u} - \textbf{u}|_{1,h} + \|\textbf{u} - \bPi_{k,h}\textbf{u}\|_{h} 
+ \boldsymbol{\mathcal{R}}_{h} + \|\textbf{u} - \textbf{I}_{k,h}\textbf{u}\|_{h}\right).
\end{equation}
Now, using the results of the previous sections, we can give an error estimate in $\|\cdot\|_{h}$ norm, under the assumption that $\textbf{u}$ belongs to $\textbf{H}^{\ell + 1}(\Omega)$, with $1 \leq \ell \leq k$. From $\text{i)}$ of Lemma \ref{lema41} and \ref{elast}, we have
\begin{equation}\label{stima12}
\boldsymbol{\mathcal{R}}_{h} \lesssim h^{\ell}|\textbf{div}(\boldsymbol{\sigma}(\textbf{u}))|_{\ell - 1,\Omega} \lesssim h^{\ell}|\boldsymbol{\sigma}(\textbf{u})|_{\ell,\Omega}
\lesssim \max\{\lambda_{S},\mu_{S}\}h^{\ell}|\textbf{u}|_{\ell + 1, \Omega}.
\end{equation}
On the other hand, from $\text{ii})$ of Lemma \ref{lema35}, there holds 
\begin{equation*}
|\textbf{u} - \boldsymbol{\Pi}_{k,h}\textbf{u}|_{1,h}^{2} \lesssim \max\{\lambda_{S}\mu_{S}^{-1}, 1\}^{2}\displaystyle{\sum_{E \in \mathcal{T}_{h}} h_{E}^{2\ell}|\textbf{u}|_{\ell + 1,E}^{2}} \lesssim \max\{\lambda_{S}\mu_{S}^{-1}, 1\}^{2}h^{2\ell}|\textbf{u}|_{\ell + 1,\Omega}^{2}, 
\end{equation*}
in which we derive
\begin{equation}\label{stima13}
|\textbf{u} - \boldsymbol{\Pi}_{k,h}\textbf{u}|_{1,h} \lesssim \max\{\lambda_{S}\mu_{S}^{-1}, 1\}h^{\ell}|\textbf{u}|_{\ell + 1,\Omega}.
\end{equation}
Now, to estimate $\|\textbf{u} - \textbf{I}_{k,h}\textbf{u}\|_{h}$, we put $\textbf{j} := \textbf{u} - \textbf{I}_{k,h}\textbf{u}$ and $\textbf{j}_{E} := \textbf{u} - \textbf{I}_{k,E}\textbf{u}$. Then, we have
\begin{multline*}
\|\textbf{j}\|_{h}^{2} = \displaystyle{\sum_{E \in \mathcal{T}_{h}} \left[a^{E}(\boldsymbol{\Pi}_{k,E}\textbf{j}_{E}, \boldsymbol{\Pi}_{k,E}\textbf{j}_{E}) + S^{E}(\textbf{j}_{E} - \boldsymbol{\Pi}_{k,E}\textbf{j}_{E}, \textbf{j}_{E} - \boldsymbol{\Pi}_{k,E}\textbf{j}_{E})\right]} \\
\lesssim \displaystyle{\sum_{E \in \mathcal{T}_{h}} \max\{\lambda_{S},\mu_{S}\}|\boldsymbol{\Pi}_{k,E}\textbf{j}_{E}|_{1,E}^{2}} + \displaystyle{\sum_{E \in \mathcal{T}_{h}} h_{E}\|\partial_{s}\textbf{j}_{E}\|_{0,\partial E}^{2}} 
+ \displaystyle{\sum_{E \in \mathcal{T}_{h}} h_{E}\|\partial_{s}\boldsymbol{\Pi}_{k,E}\textbf{j}_{E}\|_{0, \partial E}^{2}}.
\end{multline*}
First, from  item $\text{ii})$ of Lemma \ref{lema315} and \eqref{cota1}, we obtain
\begin{equation*}
\displaystyle{\sum_{E \in \mathcal{T}_{h}} \max\{\lambda_{S},\mu_{S}\}|\boldsymbol{\Pi}_{k,E}\textbf{j}_{E}|_{1,E}^{2}} \lesssim \max\{\lambda_{S},\mu_{S}\}\max\{\lambda_{S}\mu_{S}^{-1},1\}^{3}h^{2\ell}|\textbf{u}|_{\ell + 1,\Omega}^{2}.
\end{equation*}
%
%
On the other hand, using the estimate
\begin{equation*}
|\boldsymbol{\zeta}|_{1/2,\partial E} \lesssim |\boldsymbol{\zeta}|_{1,E} \quad \forall \boldsymbol{\zeta} \in \textbf{H}^{1}(E),
\end{equation*}
applied to the $\ell$-order derivatives of $\textbf{u}$, and standard interpolation estimates, we obtain
\begin{equation*}
\begin{split}
\displaystyle{\sum_{E \in \mathcal{T}_{h}} h_{E}\|\partial_{s}\textbf{s}_{E}\|_{0,\partial E}^{2}} = \displaystyle{\sum_{E \in \mathcal{T}_{h}} h_{E}\|\partial_{s}(\textbf{u} - \textbf{I}_{k,E}\textbf{u})\|_{0,\partial E}^{2}} &\lesssim \displaystyle{\sum_{E \in \mathcal{T}_{h}} h_{E}\sum_{e \in \mathcal{E}_{E}} h_{e}^{2\ell  - 1}|\partial_{s}^{\ell}\textbf{u}|_{1/2,e}^{2}} \\
&\lesssim h^{2\ell}|\textbf{u}|_{\ell + 1,\Omega}^{2}.
\end{split}
\end{equation*}
Finally, applying the polynomial estimate $\|\textbf{p}\|_{0,\partial E}^{2} \lesssim h_{E}^{-1}\|\textbf{p}\|_{0,E}^{2}$ together with \eqref{cota1} and $\text{ii)}$ of Lemma \ref{lema315}, we obtain
\begin{multline*}
\displaystyle{\sum_{E \in \mathcal{T}_{h}} h_{E}\|\partial_{s}\boldsymbol{\Pi}_{k,E}\textbf{s}_{E}\|_{0, \partial E}^{2}} = \displaystyle{\sum_{E \in \mathcal{T}_{h}} h_{E}\|\partial_{s}\boldsymbol{\Pi}_{k,E}(\textbf{u} - \textbf{I}_{k,E}\textbf{u})\|_{0, \partial E}^{2}}\\
 \lesssim \displaystyle{\sum_{E \in \mathcal{T}_{h}} |\boldsymbol{\Pi}_{k,E}(\textbf{u} - \textbf{I}_{k,E}\textbf{u})|_{1,E}^{2}} 
\lesssim \max\{\lambda_{S}\mu_{S}^{-1},1\}^{2}\displaystyle{\sum_{E \in \mathcal{T}_{h}} |\textbf{u} - \textbf{I}_{k,E}\textbf{u}|_{1,E}^{2}}\\ \lesssim \max\{\lambda_{S}\mu_{S}^{-1},1\}^{3}h^{2\ell}|\textbf{u}|_{\ell + 1,\Omega}^{2}.
\end{multline*}
Then, there holds
\begin{equation}\label{stima14}
\|\textbf{u} - \textbf{I}_{k,h}\textbf{u}\|_{h} \lesssim C_{2}(\lambda_{S},\mu_{S})h^{\ell}|\textbf{u}|_{\ell + 1,\Omega},
\end{equation}
with $C_{2}(\lambda_{S},\mu_{S}) := \max\{\lambda_{S}\mu_{S}^{-1},1\}^{3/2}\max\{\lambda_{S}^{1/2},\mu_{S}^{1/2},1\}$. \\
As a final task, we provide an  estimate for $\|\textbf{u} - \boldsymbol{\Pi}_{k,h}\textbf{u}\|_{h}$. From the definition of $\|\cdot\|_{h}$ and $S^{E}(\cdot, \cdot)$, \eqref{traceineq0} and Lemma \ref{lema35}, we have
\begin{multline*}
\|\textbf{u} - \boldsymbol{\Pi}_{k,h}\textbf{u}\|_{h}^{2} = \displaystyle{\sum_{E \in \mathcal{T}_{h}} S^{E}(\textbf{u} - \boldsymbol{\Pi}_{k,h}\textbf{u}, \textbf{u} - \boldsymbol{\Pi}_{k,h}\textbf{u})} = \displaystyle{\sum_{E \in \mathcal{T}_{h}} h_{E}\|\partial_{s}(\textbf{u} - \boldsymbol{\Pi}_{k,E}\textbf{u})\|_{0,\partial E}^{2}} \\
\lesssim \displaystyle{\sum_{E \in \mathcal{T}_{h}} \left[|\textbf{u} - \boldsymbol{\Pi}_{k,E}\textbf{u}|_{1,E}^{2} + h_{E}^{2}|\textbf{u} - \boldsymbol{\Pi}_{k,E}\textbf{u}|_{2,E}^{2}\right]}
 \lesssim \max\{\lambda_{S}\mu_{S}^{-1},1\}^{2}h^{2\ell}|\textbf{u}|_{\ell+1,\Omega}^{2}.
\end{multline*}
Hence, we derive the estimate
\begin{equation}\label{stima15}
\|\textbf{u} - \boldsymbol{\Pi}_{k,h}\textbf{u}\|_{h} \lesssim \max\{\lambda_{S}\mu_{S}^{-1},1\}h^{\ell}|\textbf{u}|_{\ell+1,\Omega}.
\end{equation}
The above estimates give us the following result, that is an adaptation of \cite[Theorem 4.1]{MR3815658}, and we emphasize that the Lam\'e coefficients appears.
\begin{thm}\label{teorema41}
Assuming that the solution $\textbf{u}$ of  Problem \ref{continuous} belongs to $\textbf{H}^{\ell + 1}(\Omega)$ for $1 \leq \ell \leq k$. Then, the following estimate holds 
\begin{equation*}
\|\textbf{u} - \textbf{u}_{h}\|_{h} \lesssim C(\lambda_{S},\mu_{S})h^{\ell}|\textbf{u}|_{\ell + 1,\Omega},
\end{equation*}
where the constants $C(\lambda_{S},\mu_{S})$ is given by 
\begin{equation*}
C(\lambda_{S},\mu_{S}) := C_{1}(\lambda_{S},\mu_{S})\widetilde{C}_{2}(\lambda_{S},\mu_{S}), \,\,\,\text{and}\quad
\widetilde{C}_{2}(\lambda_{S},\mu_{S}) := \max\{C_{2}(\lambda_{S},\mu_{S}),\lambda_{S}\mu_{S}^{-1},1\},
\end{equation*}
and the hidden constant is independent of $h$.
\end{thm}
\begin{proof}
The proof follows by replacing  \eqref{stima12}, \eqref{stima13}, \eqref{stima14} and \eqref{stima15} in  \eqref{stima11}.
\end{proof}
We also have the following result, that give us an error estimate for $\boldsymbol{\Pi}_{k,h}\textbf{u}_{h}$ and $\boldsymbol{\Pi}_{k,h}^{0}\textbf{u}$ in $|\cdot|_{1,\Omega}$ seminorm. This result is adapted from \cite[Theorem 4.2]{MR3815658} to our case.
\begin{thm}\label{teorema42}
Assuming that the solution $\textbf{u}$ of Problem \ref{continuous} belongs to $\textbf{H}^{\ell + 1}(\Omega)$, $1 \leq \ell \leq k$. Then, there holds
\begin{equation*}
|\textbf{u} - \textbf{u}_{h}|_{1,\Omega} + |\textbf{u} - \boldsymbol{\Pi}_{k,h}\textbf{u}_{h}|_{1,h} + |\textbf{u} - \boldsymbol{\Pi}_{k,h}^{0}\textbf{u}|_{1,h} \lesssim K(\lambda_{S},\mu_{S})h^{\ell}|\textbf{u}|_{\ell+1, \Omega},
\end{equation*}
where $K(\lambda_{S},\mu_{S})$ is a positive constant depending on the Lam\'e coefficients.
\end{thm}
\begin{proof}
Note that, from the definition of $\|\cdot\|_{h}$ and Theorem \ref{teorema41}, we have
\begin{multline*}
|\boldsymbol{\Pi}_{k,h}(\textbf{u} - \textbf{u}_{h})|_{1,h}^{2} \leq \displaystyle{\sum_{E \in \mathcal{T}_{h}} \mu_{S}^{-1}a^{E}(\boldsymbol{\Pi}_{k,E}(\textbf{u} - \textbf{u}_{h}),\boldsymbol{\Pi}_{k,E}(\textbf{u} - \textbf{u}_{h}))}\\
 \lesssim \mu_{S}^{-1}\|\textbf{u} - \textbf{u}_{h}\|_{h}^{2} 
\lesssim \mu_{S}^{-1}C(\lambda_{S},\mu_{S})^{2}h^{2\ell}|\textbf{u}|_{\ell + 1,\Omega}^{2}.
\end{multline*}
This gives us the following estimate
\begin{equation*}
|\boldsymbol{\Pi}_{k,h}(\textbf{u} - \textbf{u}_{h})|_{1,h} \lesssim \mu_{S}^{-1/2}C(\lambda_{S},\mu_{S})h^{\ell}|\textbf{u}|_{\ell + 1,\Omega},
\end{equation*}
Therefore, applying triangular inequality and \eqref{stima13}, we obtain
\begin{equation}\label{stima16}
|\textbf{u} - \boldsymbol{\Pi}_{k,h}\textbf{u}_{h}|_{1,h} \leq |\textbf{u} - \boldsymbol{\Pi}_{k,h}\textbf{u}|_{1,h} + |\boldsymbol{\Pi}_{k,h}(\textbf{u} - \textbf{u}_{h})|_{1,h} \lesssim C_{3}(\lambda_{S},\mu_{S})h^{\ell}|\textbf{u}|_{\ell + 1,\Omega},
\end{equation}
where $C_{3}(\lambda_{S},\mu_{S}) := \max\{\lambda_{S}\mu_{S}^{-1},1, \mu_{S}^{-1/2}C(\lambda_{S},\mu_{S})\}$. On the other hand, applying triangular inequality, Lemma \ref{lema36} and \eqref{stima13}, we deduce that
\begin{equation}\label{stima17}
|\textbf{u} - \boldsymbol{\Pi}_{k,h}^{0}\textbf{u}|_{1,h} \leq |\textbf{u} - \boldsymbol{\Pi}_{k,h}\textbf{u}|_{1,h} + |\boldsymbol{\Pi}_{k,h}^{0}(\boldsymbol{\Pi}_{k,h}\textbf{u} - \textbf{u})|_{1,h} \lesssim \max\{\lambda_{S}^{2}\mu_{S}^{-2},1\}h^{\ell}|\textbf{u}|_{\ell + 1,\Omega}.
\end{equation}
Finally, from triangular inequality, \eqref{stima8}, \eqref{stima14}, Theorem \ref{teorema41}, and $\text{ii)}$ of Lemma \ref{lema315}, we obtain
\begin{equation}\label{stima18}
\begin{split}
|\textbf{u} - \textbf{u}_{h}|_{1,\Omega} &\lesssim |\textbf{u} - \textbf{I}_{k,h}\textbf{u}|_{1,\Omega} + \max\{\lambda_{S}\mu_{S}^{-1},\mu_{S}^{-1/2},1\}\|\textbf{I}_{k,h}\textbf{u} - \textbf{u}_{h}\|_{h} \\
&\lesssim |\textbf{u} - \textbf{I}_{k,h}\textbf{u}|_{1,\Omega} + \max\{\lambda_{S}\mu_{S}^{-1},\mu_{S}^{-1/2},1\}\left(\|\textbf{I}_{k,h}\textbf{u} - \textbf{u}\|_{h} + \|\textbf{u} - \textbf{u}_{h}\|_{h}\right) \\
&\lesssim C_{4}(\lambda_{S},\mu_{S})h^{\ell}|\textbf{u}|_{\ell + 1,\Omega},
\end{split}
\end{equation}
with
\begin{equation*}
\begin{split}
C_{4}(\lambda_{S},\mu_{S}) &:= \max\{\lambda_{S}\mu_{S}^{-1},1,\widetilde{C}_{4}(\lambda_{S},\mu_{S})\}, \\
\widetilde{C}_{4}(\lambda_{S},\mu_{S}) &:= \max\{\lambda_{S}\mu_{S}^{-1},\mu_{S}^{-1/2},1\}\max\{C(\lambda_{S},\mu_{S}),C_{2}(\lambda_{S},\mu_{S})\}.
\end{split}
\end{equation*}
Therefore, we conclude the proof from \eqref{stima16}, \eqref{stima17} y \eqref{stima18}.
\end{proof}
Now, we will prove two results involving the stability term $S(\cdot,\cdot)$. These results are adapted respectively from \cite[Lemma 4.2]{MR3815658} and \cite[Lemma 4.3]{MR3815658}, to our case.
\begin{lemma}\label{lema42}
For all $\boldsymbol{\zeta} \in \textbf{H}_{0}^{1}(\Omega) \cap \textbf{H}^{2}(\Omega)$, there holds 
\begin{equation*}
\displaystyle{\sum_{E \in \mathcal{T}_{h}} S^{E}(\boldsymbol{\zeta} - \boldsymbol{\Pi}_{k,E}\textbf{I}_{k,E}\boldsymbol{\zeta},\boldsymbol{\zeta} - \boldsymbol{\Pi}_{k,E}\textbf{I}_{k,E}\boldsymbol{\zeta})} \lesssim \max\{\lambda_{S}\mu_{S}^{-1},1\}^{2}h^{2}|\boldsymbol{\zeta}|_{2,\Omega}^{2}.
\end{equation*}
\end{lemma}
\begin{proof}
From the definition of $S^{E}(\cdot, \cdot)$, \eqref{traceineq0}, and Lemma \ref{lema315} with $\ell = 1$, we have
\begin{multline*}
\displaystyle{\sum_{E \in \mathcal{T}_{h}} S^{E}(\boldsymbol{\zeta} - \boldsymbol{\Pi}_{k,E}\textbf{I}_{k,E}\boldsymbol{\zeta},\boldsymbol{\zeta} - \boldsymbol{\Pi}_{k,E}\textbf{I}_{k,E}\boldsymbol{\zeta})} \lesssim \displaystyle{\sum_{E \in \mathcal{T}_{h}} h_{E}\|\partial_{s}\left(\boldsymbol{\zeta} - \boldsymbol{\Pi}_{k,E}\textbf{I}_{k,E}\boldsymbol{\zeta}\right)\|_{0,\partial E}^{2}} \\
\lesssim \displaystyle{\sum_{E \in \mathcal{T}_{h}} \left(|\boldsymbol{\zeta} - \boldsymbol{\Pi}_{k,E}\textbf{I}_{k,E}\boldsymbol{\zeta}|_{1,E}^{2} + h_{E}^{2}|\boldsymbol{\zeta} - \boldsymbol{\Pi}_{k,E}\textbf{I}_{k,E}\boldsymbol{\zeta}|_{2,E}^{2}\right)} \lesssim \max\{\lambda_{S}^{2}\mu_{S}^{-2},1\}h^{2}|\boldsymbol{\zeta}|_{2,\Omega}^{2}.
\end{multline*}
This concludes the proof.
\end{proof}
\begin{lemma}\label{lema43}
Assume that $\textbf{u} \in \textbf{H}^{\ell + 1}(\Omega)$, $1 \leq \ell \leq k$. Then there holds
\begin{equation*}
\displaystyle{\sum_{E \in \mathcal{T}_{h}} S^{E}(\textbf{u}_{h} - \boldsymbol{\Pi}_{k,E}\textbf{u}_{h},\textbf{u}_{h} - \boldsymbol{\Pi}_{k,E}\textbf{u}_{h})} \lesssim \mathcal{C}(\lambda_{S},\mu_{S})h^{2\ell}|\textbf{u}|_{\ell + 1,\Omega}^{2},
\end{equation*}
where $\mathcal{C}(\lambda_{S},\mu_{S})$ is a positive constant depending on the Lam\'e coefficients.
\end{lemma}
\begin{proof}
From triangle inequality, the continuity of $a^{E}(\cdot, \cdot)$, Theorem \ref{teorema41}, and \eqref{stima15} we obtain
\begin{multline}\label{stima19}
\displaystyle{\sum_{E \in \mathcal{T}_{h}} S^{E}(\textbf{u}_{h} - \boldsymbol{\Pi}_{k,E}\textbf{u}_{h},\textbf{u}_{h} - \boldsymbol{\Pi}_{k,E}\textbf{u}_{h})} = \|\textbf{u}_{h} - \boldsymbol{\Pi}_{k,h}\textbf{u}_{h}\|_{h}^{2} \\
\lesssim \|\textbf{u} - \textbf{u}_{h}\|_{h}^{2} + \|\textbf{u} - \boldsymbol{\Pi}_{k,h}\textbf{u}\|_{h}^{2} + \|\bPi_{k,h}(\textbf{u} - \textbf{u}_{h})\|_{h}^{2} \\
\lesssim \left(C(\lambda_{S},\mu_{S})^{2} + \max\{\lambda_{S}^{2}\mu_{S}^{-2},1\}\right)h^{2\ell}|\textbf{u}|_{\ell + 1,\Omega} \\
+ \|\bPi_{k,h}(\textbf{u} - \textbf{u}_{h})\|_{h}^{2}. 
\end{multline}
On the other hand, from the definition of $\|\cdot\|_{h}$, \eqref{cota1}, and Theorem \ref{teorema42}, we have
\begin{equation*}
\begin{split}
\|\boldsymbol{\Pi}_{k,h}(\textbf{u} - \textbf{u}_{h})\|_{h}^{2} &= \displaystyle{\sum_{E \in \mathcal{T}_{h}} a^{E}(\boldsymbol{\Pi}_{k,E}(\textbf{u} - \textbf{u}_{h}),\boldsymbol{\Pi}_{k,E}(\textbf{u} - \textbf{u}_{h}))} \\
&\lesssim \max\{\lambda_{S}^{3}\mu_{S}^{-2},\lambda_{S}^{2}\mu_{S}^{-1},\lambda_{S},\mu_{S}\}\displaystyle{\sum_{E \in \mathcal{T}_{h}} |\textbf{u} - \textbf{u}_{h}|_{1,E}^{2}} \\
&\lesssim C_{5}(\lambda_{S},\mu_{S})h^{2\ell}|\textbf{u}|_{\ell + 1,\Omega}^{2},
\end{split}
\end{equation*}
where $C_{5}(\lambda_{S},\mu_{S}) := \max\{\lambda_{S}^{3}\mu_{S}^{-2},\lambda_{S}^{2}\mu_{S}^{-1},\lambda_{S},\mu_{S}\}(K(\lambda_{S},\mu_{S}))^{2}$. Hence, replacing this into \eqref{stima19} we obtain
\begin{equation*}
\displaystyle{\sum_{E \in \mathcal{T}_{h}} S^{E}(\textbf{u}_{h} - \boldsymbol{\Pi}_{k,E}\textbf{u}_{h},\textbf{u}_{h} - \boldsymbol{\Pi}_{k,E}\textbf{u}_{h})} \lesssim \mathcal{C}(\lambda_{S},\mu_{S})h^{2\ell}|\textbf{u}|_{\ell + 1,\Omega}^{2},
\end{equation*}
where $\mathcal{C}(\lambda_{S},\mu_{S}) := \max\{C(\lambda_{S},\mu_{S})^{2}, C_{5}(\lambda_{S},\mu_{S}), \max\{\lambda_{S}^{2}\mu_{S}^{-2},1\}\}$.
This concludes the proof.
\end{proof}
The following lemma is adapted from \cite[Lemma 4.4]{MR3815658}, and gives us a consistency result.
\begin{lemma}\label{lema44}
Assuming that $\textbf{u} \in \textbf{H}^{\ell + 1}(\Omega)$ with $1 \leq \ell \leq k$. Then, there holds
\begin{equation*}
a(\textbf{u} - \textbf{u}_{h}, \textbf{I}_{k,h}\boldsymbol{\zeta}) \lesssim \mathfrak{D}(\lambda_{S},\mu_{S})h^{\ell + 1}|\textbf{u}|_{\ell + 1,\Omega}|\boldsymbol{\zeta}|_{2,\Omega} \quad \forall \boldsymbol{\zeta} \in \textbf{H}^{2}(\Omega) \cap \textbf{H}_{0}^{1}(\Omega),
\end{equation*}
where $\mathfrak{D}(\lambda_{S},\mu_{S})$ is a positive constant depending on the Lam\'e coefficients.
\end{lemma}

\begin{proof}
From the linearity of $a(\cdot, \cdot)$, the definition of $a_{h}(\cdot, \cdot)$, Problems \ref{continuous} and \ref{discrette}, we have
\begin{multline}\label{stima20}
a(\textbf{u} - \textbf{u}_{h},\textbf{I}_{k,h}\boldsymbol{\zeta}) = a(\textbf{u},\textbf{I}_{k,h}\textbf{u}) - \displaystyle{\sum_{E \in \mathcal{T}_{h}} a^{E}(\textbf{u}_{h},\textbf{I}_{k,h}\boldsymbol{\zeta})} \\
= F(\textbf{I}_{k,h}\boldsymbol{\zeta}) - F_{h}(\textbf{I}_{k,h}\boldsymbol{\zeta}) + \displaystyle{\sum_{E \in \mathcal{T}_{h}} S^{E}(\textbf{u}_{h} - \boldsymbol{\Pi}_{k,E}\textbf{u}_{h}, \textbf{I}_{k,E}\boldsymbol{\zeta} - \boldsymbol{\Pi}_{k,E}\textbf{I}_{k,E}\boldsymbol{\zeta})} \\
+ \displaystyle{\sum_{E \in \mathcal{T}_{h}} a^{E}(\boldsymbol{\Pi}_{k,E}\textbf{u}_{h} - \textbf{u}_{h},\textbf{I}_{k,h}\boldsymbol{\zeta} - \boldsymbol{\Pi}_{k,E}\textbf{I}_{k,E}\boldsymbol{\zeta})}.
\end{multline}
Now, from item $\text{ii)}$ of  Lemma~\ref{lema41} we have 
\begin{equation}\label{stima21}
\left\lvert F(\textbf{I}_{k,h}\boldsymbol{\zeta}) - F_{h}(\textbf{I}_{k,h}\boldsymbol{\zeta})\right\lvert \lesssim \max\{\lambda_{S}\mu_{S}^{-1},1\}^{2}h^{\ell + 1}|\textbf{u}|_{\ell + 1,\Omega}|\boldsymbol{\zeta}|_{2,\Omega}.
\end{equation}
On the other hand, applying Cauchy Schwarz inequality, and Lemmas \ref{lema42} and \ref{lema43}, we obtain
\begin{equation}\label{stima22}
\displaystyle{\sum_{E \in \mathcal{T}_{h}} S^{E}(\textbf{u}_{h} - \boldsymbol{\Pi}_{k,E}\textbf{u}_{h}, \textbf{I}_{k,E}\boldsymbol{\zeta} - \boldsymbol{\Pi}_{k,E}\textbf{I}_{k,E}\boldsymbol{\zeta})} \lesssim C_{6}(\lambda_{S},\mu_{S})h^{\ell + 1}|\textbf{u}|_{\ell + 1,\Omega}|\boldsymbol{\zeta}|_{2,\Omega},
\end{equation}
\noindent with $C_{6}(\lambda_{S},\mu_{S}) := \max\{\lambda_{S}\mu_{S}^{-1},1\}\sqrt{\mathcal{C}(\lambda_{S},\mu_{S})}$. Finally, invoking the Cauchy Schwarz inequality, Theorems \ref{teorema41} and \ref{teorema42}, and part $\text{ii)}$ of Lemma \ref{lema315} yields to
\begin{multline}\label{stima23}
\displaystyle{\sum_{E \in \mathcal{T}_{h}} a^{E}(\boldsymbol{\Pi}_{k,E}\textbf{u}_{h} - \textbf{u}_{h},\textbf{I}_{k,h}\boldsymbol{\zeta} - \boldsymbol{\Pi}_{k,E}\textbf{I}_{k,E}\boldsymbol{\zeta})} \\
\lesssim \max\{\lambda_{S}^{2},\mu_{S}^{2}\}\left(|\textbf{u}_{h} - \boldsymbol{\Pi}_{k,h}\textbf{u}_{h}|_{1,h}\right)\left(|\textbf{I}_{k,h}\boldsymbol{\zeta} - \boldsymbol{\Pi}_{k,h}\textbf{I}_{k,h}\boldsymbol{\zeta}|_{1,h}\right) \\
\leq \max\{\lambda_{S}^{2},\mu_{S}^{2}\}\left(|\textbf{u}_{h} - \textbf{u}|_{1,\Omega} + |\textbf{u} - \boldsymbol{\Pi}_{k,h}\textbf{u}_{h}|_{1,h}\right)\left(|\textbf{I}_{k,h}\boldsymbol{\zeta} - \boldsymbol{\zeta}|_{1,\Omega} + |\boldsymbol{\zeta} - \boldsymbol{\Pi}_{k,h}\textbf{I}_{k,h}\boldsymbol{\zeta}|_{1,h}\right) \\
\lesssim C_{7}(\lambda_{S},\mu_{S})h^{\ell + 1}|\textbf{u}|_{\ell + 1,\Omega}|\boldsymbol{\zeta}|_{2,\Omega},
\end{multline}
where $C_{7}(\lambda_{S},\mu_{S}) := \max\{\lambda_{S}^{3}\mu_{S}^{-1}, \lambda_{S}\mu_{S}, \lambda_{S}^{2}, \mu_{S}^{2}\}\max\{C(\lambda_{S},\mu_{S}),K(\lambda_{S},\mu_{S})\}$. Thus, from \eqref{stima20}, \eqref{stima21}, \eqref{stima22} and \eqref{stima23} allows us to conclude
\begin{equation*}
a(\textbf{u} - \textbf{u}_{h},\textbf{I}_{k,h}\boldsymbol{\zeta}) \lesssim \mathfrak{D}(\lambda_{S},\mu_{S})h^{\ell + 1}|\textbf{u}|_{\ell + 1,\Omega}|\boldsymbol{\zeta}|_{2,\Omega},
\end{equation*}
with $\mathfrak{D}(\lambda_{S},\mu_{S}) := \max\{C_{6}(\lambda_{S},\mu_{S}), C_{7}(\lambda_{S},\mu_{S}), \max\{\lambda_{S}\mu_{S}^{-1},1\}^{2}\}$. This concludes the proof.
\end{proof}

The following result is an adaptation of \cite[Theorem 4.3]{MR3815658} to our case, and gives us an error estimate in $\textbf{L}^{2}(\O)$ norm.

\begin{thm}\label{teorema43}
Assuming that $\textbf{u} \in \textbf{H}^{\ell + 1}(\Omega)$, $1 \leq \ell \leq k$, there holds
\begin{equation*}
\|\textbf{u} - \textbf{u}_{h}\|_{0,\Omega} \lesssim \mathfrak{R}(\lambda_{S},\mu_{S})h^{\ell + 1}|\textbf{u}|_{\ell + 1,\Omega},
\end{equation*}
where the hidden constant depends on $\O$ and not on $h$, and $\mathfrak{R}(\lambda_{S},\mu_{S})$ is a positive constant depending on the Lam\'e coefficients.
\end{thm}
\begin{proof}
Let $\boldsymbol{\zeta} \in \textbf{H}_{0}^{1}(\Omega)$ the solution of
\begin{equation*}
a(\textbf{v}, \boldsymbol{\zeta}) = \displaystyle{\int_{\Omega} \textbf{v}(\textbf{u} - \textbf{u}_{h})} \quad \forall \textbf{v} \in \textbf{H}_{0}^{1}(\Omega).
\end{equation*}
Then, we have
\begin{equation*}
\|\textbf{u} - \textbf{u}_{h}\|_{0,\Omega}^{2} = a(\textbf{u} - \textbf{u}_{h}, \boldsymbol{\zeta}) = a(\textbf{u} - \textbf{u}_{h}, \boldsymbol{\zeta} - \textbf{I}_{k,h}\boldsymbol{\zeta}) + a(\textbf{u} - \textbf{u}_{h}, \textbf{I}_{k,h}\boldsymbol{\zeta}).
\end{equation*}
Since we are assuming that $\O$ is convex, and invoking Lemma \ref{reg}, we have the following estimate for the $\mathbf{H}^2$ seminorm
\begin{equation}\label{convex}
|\boldsymbol{\zeta}|_{2,\Omega} \leq \|\boldsymbol{\zeta}\|_{2,\Omega} \lesssim C_{\Omega}\|\textbf{u} - \textbf{u}_{h}\|_{0,\Omega}.
\end{equation}
Then, from the continuity of  $a(\cdot, \cdot)$, part $\text{ii)}$ of Lemma \ref{lema315}, and  Theorem \ref{teorema42} we obtain 
\begin{multline*}
a(\textbf{u} - \textbf{u}_{h}, \boldsymbol{\zeta} - \textbf{I}_{k,h}\boldsymbol{\zeta}) \lesssim \max\{\lambda_{S},\mu_{S}\}|\textbf{u} - \textbf{u}_{h}|_{1,\Omega}|\boldsymbol{\zeta} - \textbf{I}_{k,h}\boldsymbol{\zeta}|_{1,\Omega}\\
 \lesssim C_{8}(\lambda_{S},\mu_{S})h^{\ell + 1}|\textbf{u}|_{\ell + 1,\Omega}|\boldsymbol{\zeta}|_{2,\Omega},
\end{multline*}
where  $C_{8}(\lambda_{S},\mu_{S}) := \max\{\lambda_{S}^{2}\mu_{S}^{-1},\lambda_{S},\mu_{S}\}K(\lambda_{S},\mu_{S})$.
Therefore, from \eqref{convex} and Lemma \ref{lema44} we obtain
\begin{equation*}
\begin{split}
\|\textbf{u} - \textbf{u}_{h}\|_{0,\Omega}^{2} &\lesssim  
\max\{C_{8}(\lambda_{S},\mu_{S}), \mathfrak{D}(\lambda_{S},\mu_{S})\}h^{\ell + 1}|\textbf{u}|_{\ell + 1,\Omega}|\boldsymbol{\zeta}|_{2,\Omega} \\
&\lesssim \mathfrak{R}(\lambda_{S},\mu_{S})h^{\ell + 1}|\textbf{u}|_{\ell + 1,\Omega}\|\textbf{u} - \textbf{u}_{h}\|_{0,\Omega},
\end{split}
\end{equation*}
where $\mathfrak{R}(\lambda_{S},\mu_{S}) := C_{\Omega}\max\{\mathfrak{D}(\lambda_{S},\mu_{S}), C_{8}(\lambda_{S},\mu_{S})\}$. This concludes the proof.
\end{proof}
Finally, we conclude this section with the following result, that gives us an error estimate for $\boldsymbol{\Pi}_{k,h}^{0}\textbf{u}_{h}$ and $\boldsymbol{\Pi}_{k,h}\textbf{u}_{h}$ in $\|\cdot\|_{0,\Omega}$ norm. This result is adapted from \cite[Theorem 4.4]{MR3815658} to our case.
\begin{thm}\label{teorema44}
Assuming that $\textbf{u} \in \textbf{H}^{\ell + 1}(\Omega)$, $1 \leq \ell \leq k$, there holds
\begin{equation*}
\|\textbf{u} - \boldsymbol{\Pi}_{k,h}^{0}\textbf{u}_{h}\|_{0,\Omega} + \|\textbf{u} - \boldsymbol{\Pi}_{k,h}\textbf{u}_{h}\|_{0,\Omega} \lesssim \mathfrak{C}(\lambda_{S},\mu_{S})h^{\ell + 1}|\textbf{u}|_{\ell + 1,\Omega},
\end{equation*}
where $\mathfrak{C}(\lambda_{S},\mu_{S})$ is a positive constant depending on the Lam\'e coefficients. 
\end{thm}
\begin{proof}
First, using triangular inequality, \eqref{stimaproye2} and Theorem \ref{teorema43}, we obtain
\begin{multline}\label{stima24}
\|\textbf{u} - \boldsymbol{\Pi}_{k,h}^{0}\textbf{u}_{h}\|_{0,\Omega} \leq
 \|\textbf{u} - \boldsymbol{\Pi}_{k,h}^{0}\textbf{u}\|_{0,\Omega} + \|\textbf{u} - \textbf{u}_{h}\|_{0,\Omega} \\
\lesssim \max\{\mathfrak{R}(\lambda_{S},\mu_{S}),1\}h^{\ell + 1}|\textbf{u}|_{\ell + 1,\Omega}.
\end{multline}
On the other hand, applying triangular inequality again, we have
\begin{equation*}\label{stima25}
\|\textbf{u} - \boldsymbol{\Pi}_{k,h}\textbf{u}_{h}\|_{0,\Omega} \leq \|\textbf{u} - \boldsymbol{\Pi}_{k,h}\textbf{I}_{k,h}\textbf{u}\|_{0,\Omega} + \|\boldsymbol{\Pi}_{k,h}(\textbf{I}_{k,h}\textbf{u} - \textbf{u}_{h})\|_{0,\Omega}.
\end{equation*}
Hence, from Lemma \ref{lema34}, \eqref{triplenorm} and \eqref{traceineq0}, we deduce that
\begin{multline*}
\|\boldsymbol{\Pi}_{k,h}(\textbf{I}_{k,h}\textbf{u} - \textbf{u}_{h})\|_{0,\Omega}^{2} \lesssim \max\{\lambda_{S}\mu_{S}^{-1},1\}^{2}\displaystyle{\sum_{E \in \mathcal{T}_{h}} \vertiii{\textbf{I}_{k,E}\textbf{u} - \textbf{u}_{h}}_{k,E}^{2}} \\
\lesssim \max\{\lambda_{S}\mu_{S}^{-1},1\}^{2}\left(\|\textbf{I}_{k,h}\textbf{u} - \textbf{u}\|_{0,\Omega}^{2}\right.\\
\left. + \|\textbf{u} - \textbf{u}_{h}\|_{0,\Omega}^{2} + h^{2}|\textbf{I}_{k,h}\textbf{u} - \textbf{u}|_{1,\Omega}^{2} + h^{2}|\textbf{u} - \textbf{u}_{h}|_{1,\Omega}\right).
\end{multline*}
Therefore, we obtain the estimate
\begin{multline}\label{stima26}
\|\boldsymbol{\Pi}_{k,h}(\textbf{I}_{k,h}\textbf{u} - \textbf{u}_{h})\|_{0,\Omega} \lesssim \max\{\lambda_{S}\mu_{S}^{-1},1\}\left(\|\textbf{u} - \textbf{u}_{h}\|_{0,\Omega} \right.\\
\left.+ \|\textbf{u} - \textbf{I}_{k,h}\textbf{u}\|_{0,\Omega} + h|\textbf{u} - \textbf{I}_{k,h}\textbf{u}|_{1,\Omega} + h|\textbf{u} - \textbf{u}_{h}|_{1,\Omega}\right).
\end{multline}
Now, from Lemma  \ref{lema315} we obtain
\begin{equation}\label{stima27}
\|\textbf{u} - \textbf{I}_{k,h}\textbf{u}\|_{0,\Omega} + h|\textbf{u} - \textbf{I}_{k,h}\textbf{u}|_{1,\Omega} \lesssim \max\{\lambda_{S}^{2}\mu_{S}^{-2},1\}h^{\ell + 1}|\textbf{u}|_{\ell + 1,\Omega},
\end{equation}
and by Theorems \ref{teorema42} and \ref{teorema43}, there holds
\begin{equation}\label{stima28}
\|\textbf{u} - \textbf{u}_{h}\|_{0,\Omega} + h|\textbf{u} - \textbf{u}_{h}|_{1,\Omega} \lesssim C_{9}(\lambda_{S},\mu_{S})h^{\ell + 1}|\textbf{u}|_{\ell + 1,\Omega}.
\end{equation}
with $C_{9}(\lambda_{S},\mu_{S}) := \max\{\mathfrak{R}(\lambda_{S},\mu_{S}),K(\lambda_{S},\mu_{S})\}$. Hence, from \eqref{stima26}, \eqref{stima27} and \eqref{stima28} we conclude that
\begin{equation}\label{stima29}
\|\boldsymbol{\Pi}_{k,h}(\textbf{I}_{k,h}\textbf{u} - \textbf{u}_{h})\|_{0,\Omega} \lesssim \widetilde{C}_{9}(\lambda_{S},\mu_{S})h^{\ell + 1}|\textbf{u}|_{\ell + 1,\Omega},
\end{equation}
where $\widetilde{C}_{9}(\lambda_{S},\mu_{S}) := \max\{\lambda_{S}\mu_{S}^{-1},1\}\max\{C_{9}(\lambda_{S},\mu_{S}),\max\{\lambda_{S}^{2}\mu_{S}^{-2},1\}\}$. Finally, from $\text{i)}$ of Lemma \ref{lema315}, we derive
\begin{equation}\label{stima30}
\|\textbf{u} - \boldsymbol{\Pi}_{k,h}\textbf{I}_{k,h}\textbf{u}\|_{0,\Omega} \lesssim \max\{\lambda_{S}^{2}\mu_{S}^{-2},1\}h^{\ell + 1}|\textbf{u}|_{\ell + 1,\Omega}.
\end{equation}
Therefore, from \eqref{stima29} and \eqref{stima30} we obtain 
\begin{equation}\label{stima31}
\|\textbf{u} - \boldsymbol{\Pi}_{k,h}\textbf{u}_{h}\|_{0,\Omega} \lesssim C_{10}(\lambda_{S},\mu_{S})h^{\ell + 1}|\textbf{u}|_{\ell + 1,\Omega},
\end{equation}
with $C_{10}(\lambda_{S},\mu_{S}) := \max\{\widetilde{C}_{9}(\lambda_{S},\mu_{S}),\max\{\lambda_{S}^{2}\mu_{S}^{-2},1\}\}$. Hence, we conclude from \eqref{stima24} and \eqref{stima31} that
\begin{equation*}
\|\textbf{u} - \boldsymbol{\Pi}_{k,h}^{0}\textbf{u}_{h}\|_{0,\Omega} + \|\textbf{u} - \boldsymbol{\Pi}_{k,h}\textbf{u}_{h}\|_{0,\Omega} \lesssim  \mathfrak{C}(\lambda_{S},\mu_{S})h^{\ell + 1}|\textbf{u}|_{\ell + 1,\Omega},
\end{equation*}
where $ \mathfrak{C}(\lambda_{S},\mu_{S}) := \max\{C_{10}(\lambda_{S},\mu_{S}),\max\{\mathfrak{R}(\lambda_{S},\mu_{S}),1\}\}$. This concludes the proof.
\end{proof}

\section{Numerical experiments}
\label{sec:numerics}
In the following section we report a series of numerical experiments in which 
we illustrate the performance of the proposed numerical method. The results have been obtained with a MATLAB 
code. 

To make matters precise, we are interested in the computation of experimental rates of convergence for the load problem, where we measure the error of  approximation in the $\mathbf{L}^2$ and $\mathbf{H}^1$ norms. Moreover, as we have claimed through our study, the Poisson ratio $\nu$ is not allowed to be equal to $1/2$, since the Lam\'e coefficient $\lambda_S$ tends to infinity. Then, we will focus on values of $\nu$ far enough from $1/2$, in order to obtain the desire results, and avoid the locking phenomenon. Let us remark that in all our numerical tests, we have taken as material density $\varrho=1$ and Young modulus $E=1$.

\subsection{The unit square}
We begin with a simple convex domain as the unit square $\O:=(0,1)^2$. The boundary condition for the forthcoming  tests is $\bu=\boldsymbol{0}$ on the whole boundary $\partial\Omega$. This configuration for the domain leads to solutions that are smooth enough on $\O$, except for values of $\nu$, too close to $1/2$.

The polygonal meshes that we will consider for our tests are the following: 
\begin{itemize}
\item[$\bullet$] $\mathcal{T}_{h}^{1}$: Triangles with small edges.
\item[$\bullet$] $\mathcal{T}_{h}^{2}$: Deformed triangles with middle points.
\item[$\bullet$] $\mathcal{T}_{h}^{3}$: Deformed squares. 
\item[$\bullet$] $\mathcal{T}_{h}^{4}$: Voronoi,
\end{itemize}
and in Figure \ref{fig:meshes} we present examples of these meshes.
\begin{figure}[H]
	\begin{center}
			\centering\includegraphics[height=5cm, width=6cm]{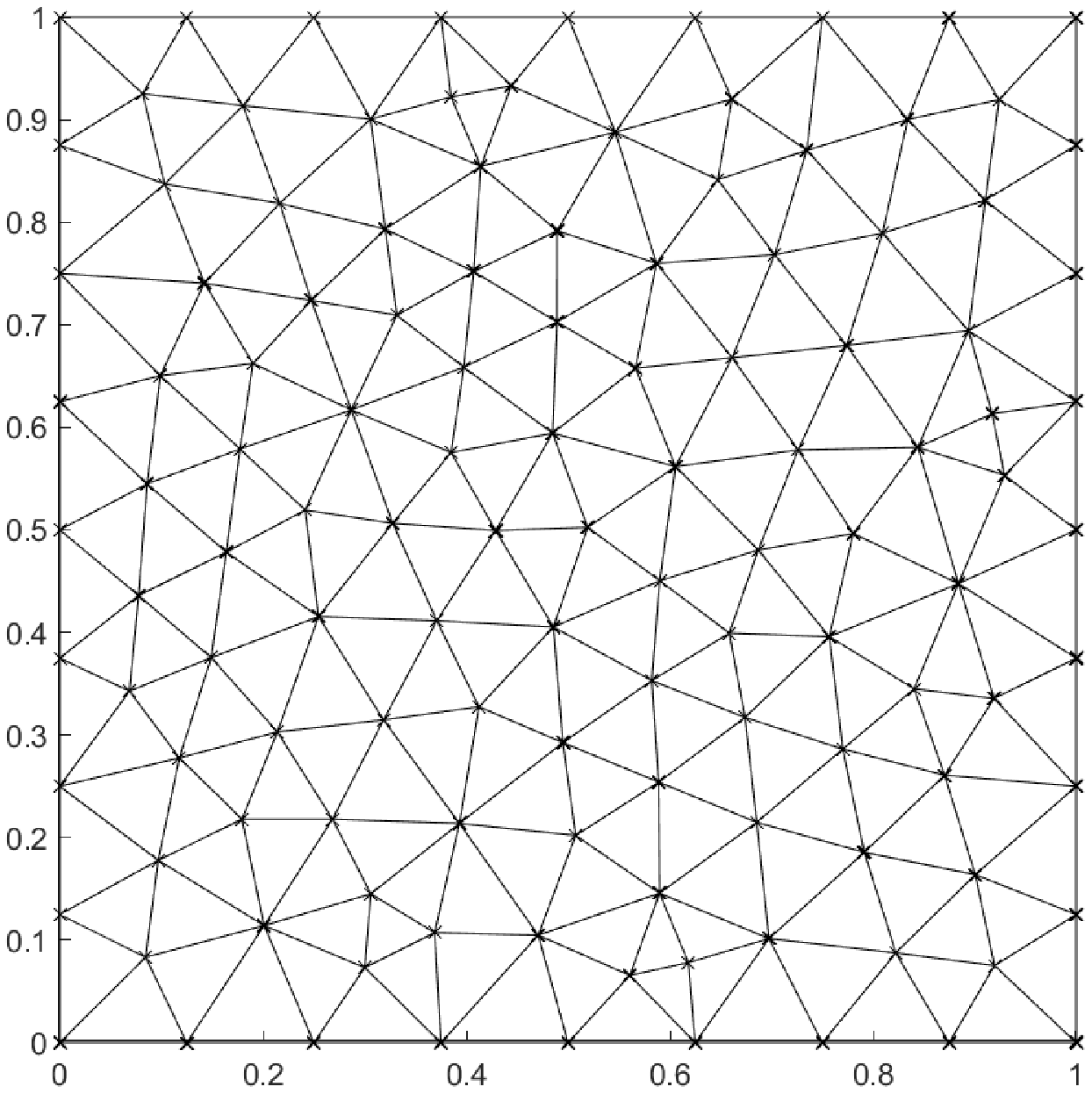}
						\centering\includegraphics[height=5cm, width=6cm]{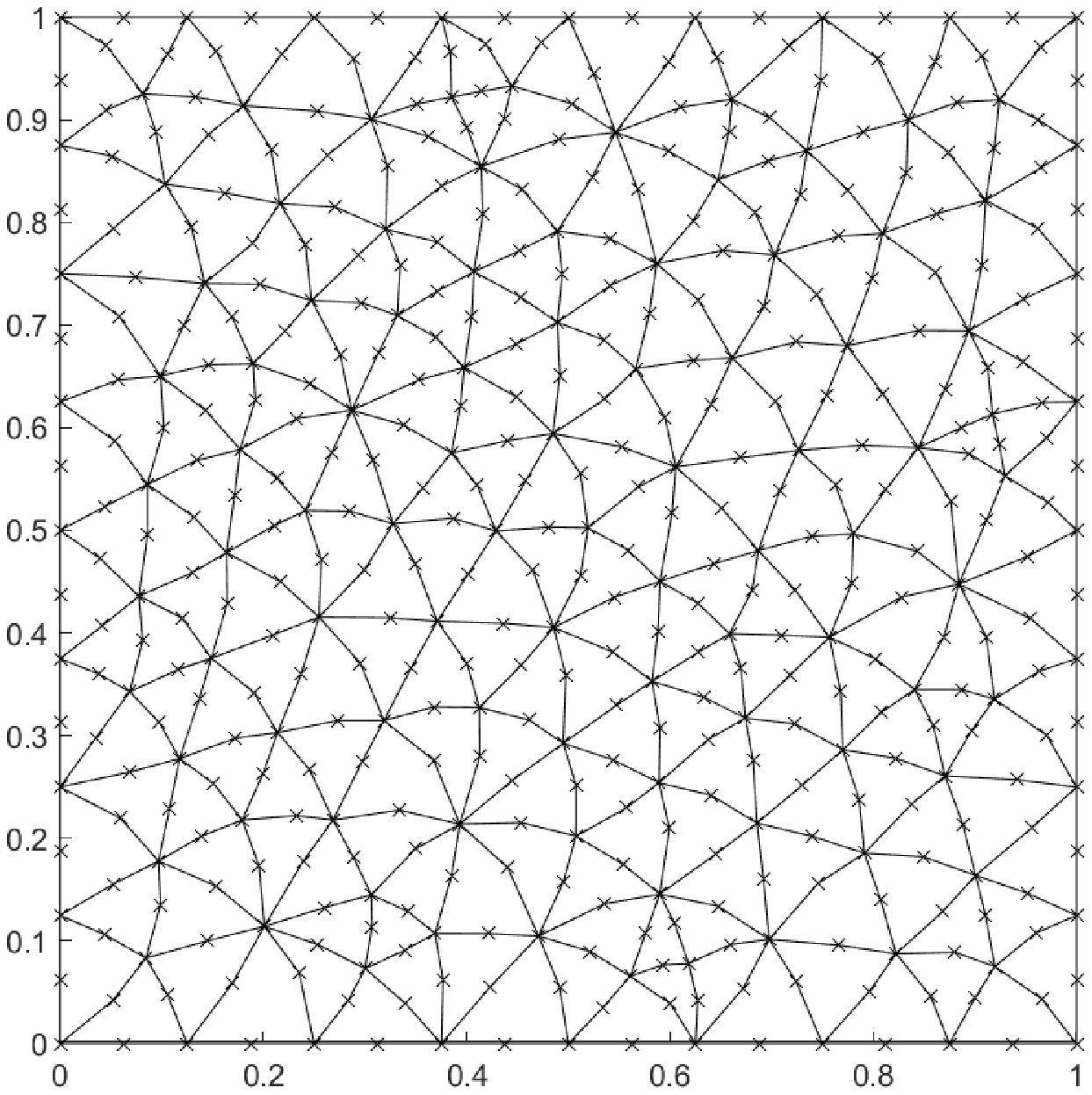}\\
                          \centering\includegraphics[height=5cm, width=6cm]{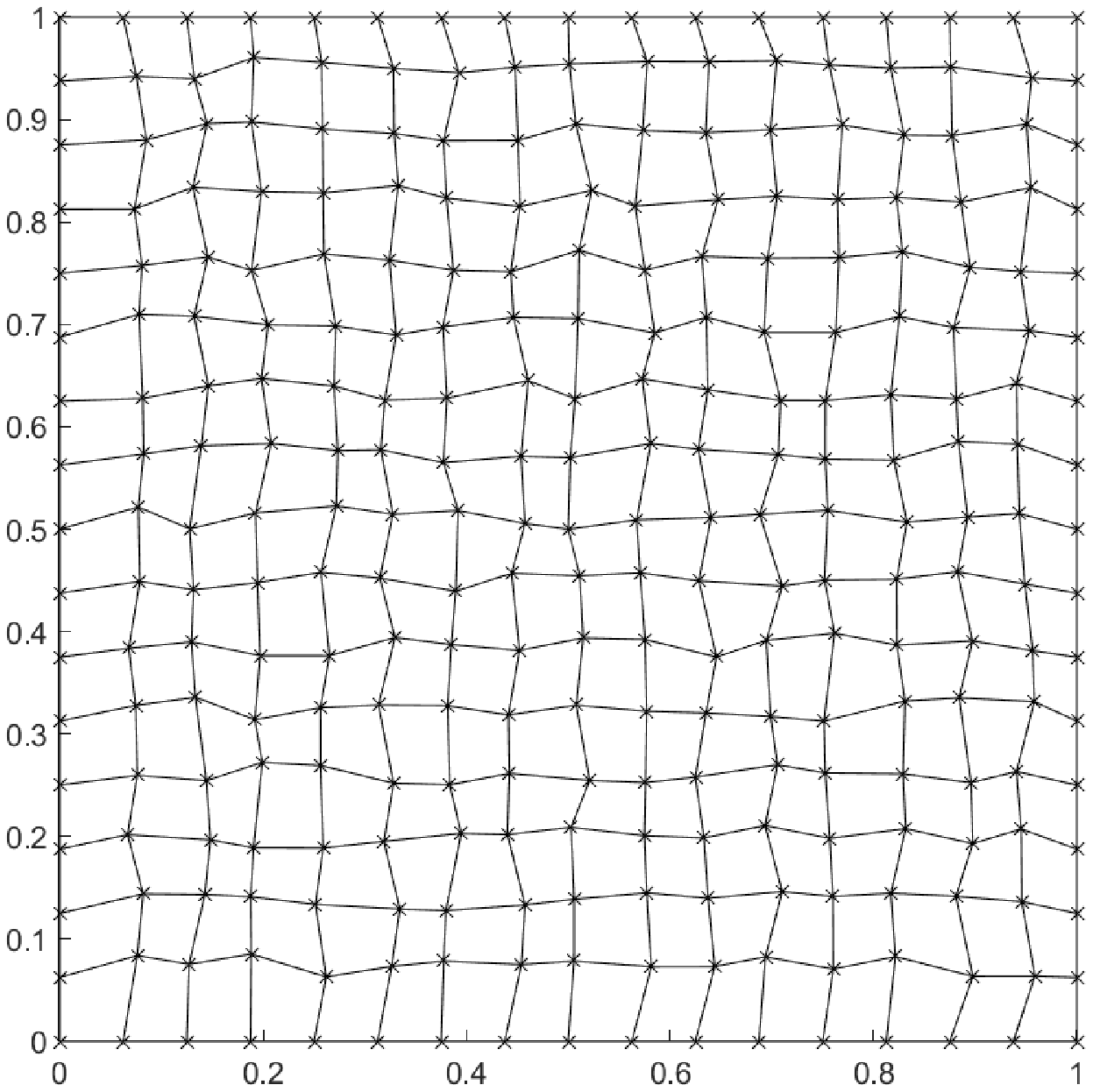}
                          \centering\includegraphics[height=5cm, width=6cm]{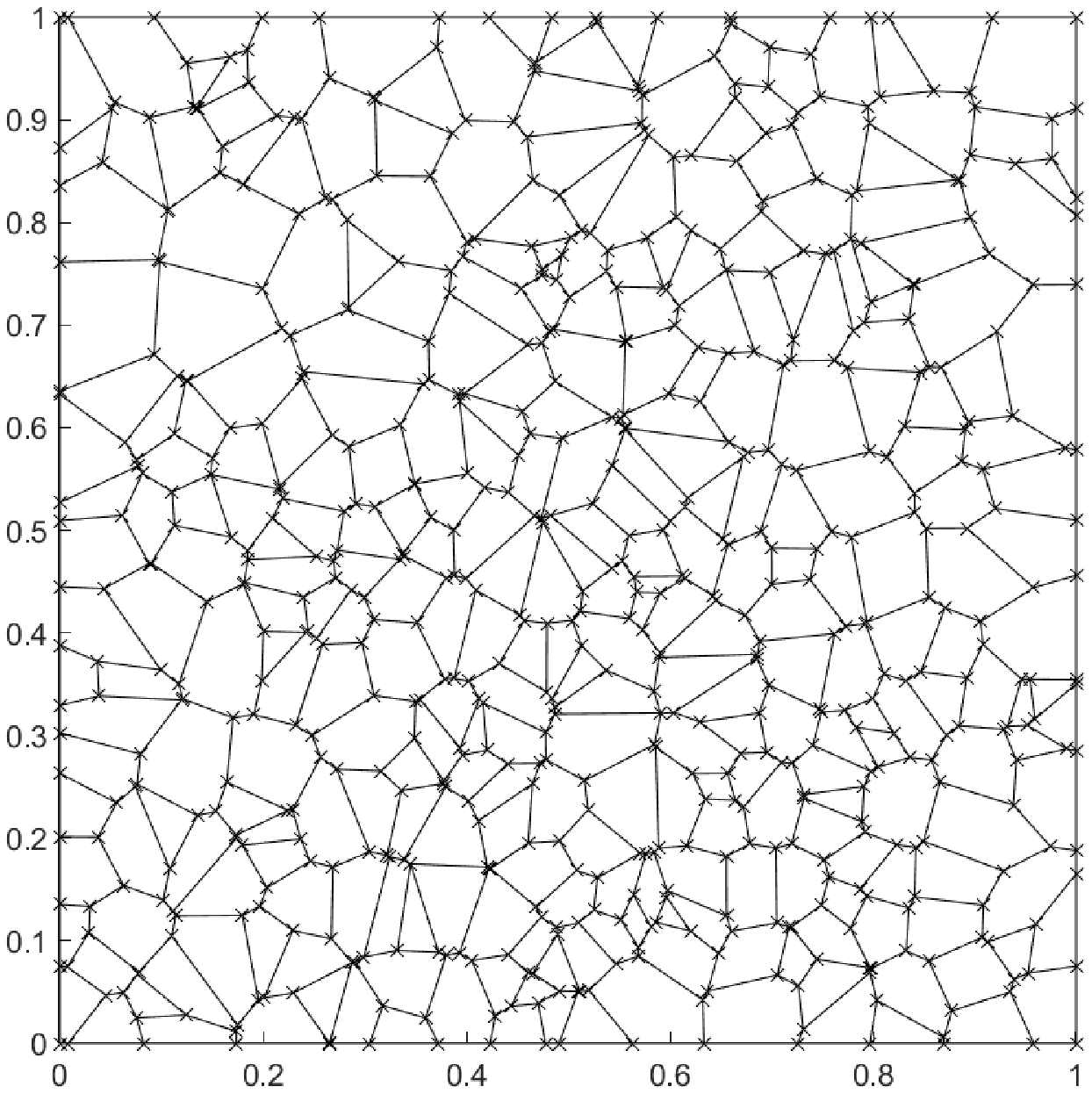}
		\caption{Example of the meshes for our experimentations. Top left: $\mathcal{T}_{h}^{1}$, top right $\mathcal{T}_{h}^{2}$, bottom left $\mathcal{T}_{h}^{3}$ and bottom right $\mathcal{T}_{h}^{4}$.
}
		\label{fig:meshes}
	\end{center}
\end{figure}
Les us mention that $\mathcal{T}_{h}^{1}$ has the particularity that, for some element $E \in \mathcal{T}_{h}^{1}$, the vertices are admissible to be very close between them, in order to almost collapse. This characteristic that can be observed in Figure \ref{zoomT1} is the flexibility that the meshes with small edges allows, making a significant improvement from the classically admissible meshes. In Figure \ref{zoomT1}, we present the way in which the elements on $\mathcal{T}_{h}^{1}$ are constructed: From a triangular mesh, we consider the middle point of each edge of the element $E$ as a new degree of freedom. Then, we move such a point on each edge to a distance $h_{e}/50$ from one vertex of $E$ and $1 - h_{e}/50$ from the other where $h_{e}$ is the length of the edge $e$, resulting a 6 edges polygon.
\begin{figure}[H]
\begin{center}
\centering\includegraphics[height=5cm, width=6cm]{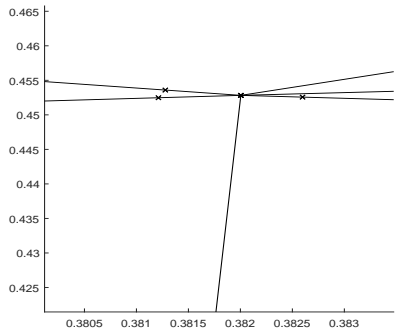}
\caption{\label{zoomT1}Zoom of $\mathcal{T}_{h}^{1}$ mesh}.
\end{center}
\end{figure}

We remark that each of these meshes have the condition of our interest: the polygons are allowed to 
consider small edges as is claimed in assumption \textbf{A1}. 

We begin by comparing the performance of the proposed method with two different stabilizations $S(\cdot,\cdot)$, in order to observe if there are significant differences when the small edges method is considered.

\subsection{The derivative stabilization}
The following results have been obtained with the following stabilization term
\begin{equation}
\label{eq:derivative_stab}
S(\textbf{w}_{h},\textbf{v}_{h}) = \displaystyle{\sum_{E \in \mathcal{T}_{h}} S^{E}(\textbf{w}_{h},\textbf{v}_{h})}, \quad S^{E}(\textbf{w}_{h}, \textbf{v}_{h}) = h_{E}\displaystyle{\int_{\partial E} \partial_{s}\textbf{w}_{h}\cdot \partial_{s}\textbf{v}_{h}}.
\end{equation}
This stabilization is the one considered in, for instance, \cite{MR3714637} for the Laplace source problem and \cite{MR4284360} for the Steklov eigenvalue problem. With this stabilization, we perform the computation of the convergence orders for the elasticity problem for different values of the Poisson ratio and hence, different values of the Lam\'e coefficient $\lambda_S$. We remark that each load term $\boldsymbol{f}$ is different for each value of $\nu$.

In Figures \ref{fig:plot1xy} and \ref{fig:plot1x} we report error curves for our method, where the meshes presented in Figure \ref{fig:meshes} have been considered. In these figures, the Poisson ratio takes the values $\nu\in\{0.35, 0.45, 0.47, 0.49\}$ and we observe that the error of the displacement in the $\mathbf{L}^2$ norm and the seminorm decay when the mesh is refined. Moreover, this decay of the error occurs with the optimal order, as is expected according to our theory. However, we observe that when $\nu=0.49$ and $\CT_h^4$ are considered, the error estimate is not completely optimal. This is expectable since  as we have claimed along our paper, when $\nu$ is close to $1/2$ the method is unstable. Moreover, $\CT_h^4$ is not an uniform mesh compared with the rest of the considered meshes, leading to geometrical issues that also may affect the convergence order for a Poisson ratio close to $1/2$.

These results allow us to infer that the small edges approach, together with the stabilization term \eqref{eq:derivative_stab}, perform an accurate approximation of the solution, independently of the polygonal mesh under consideration.
\begin{figure}[H]
	\begin{center}
			\centering\includegraphics[height=6.3cm, width=7.2cm]{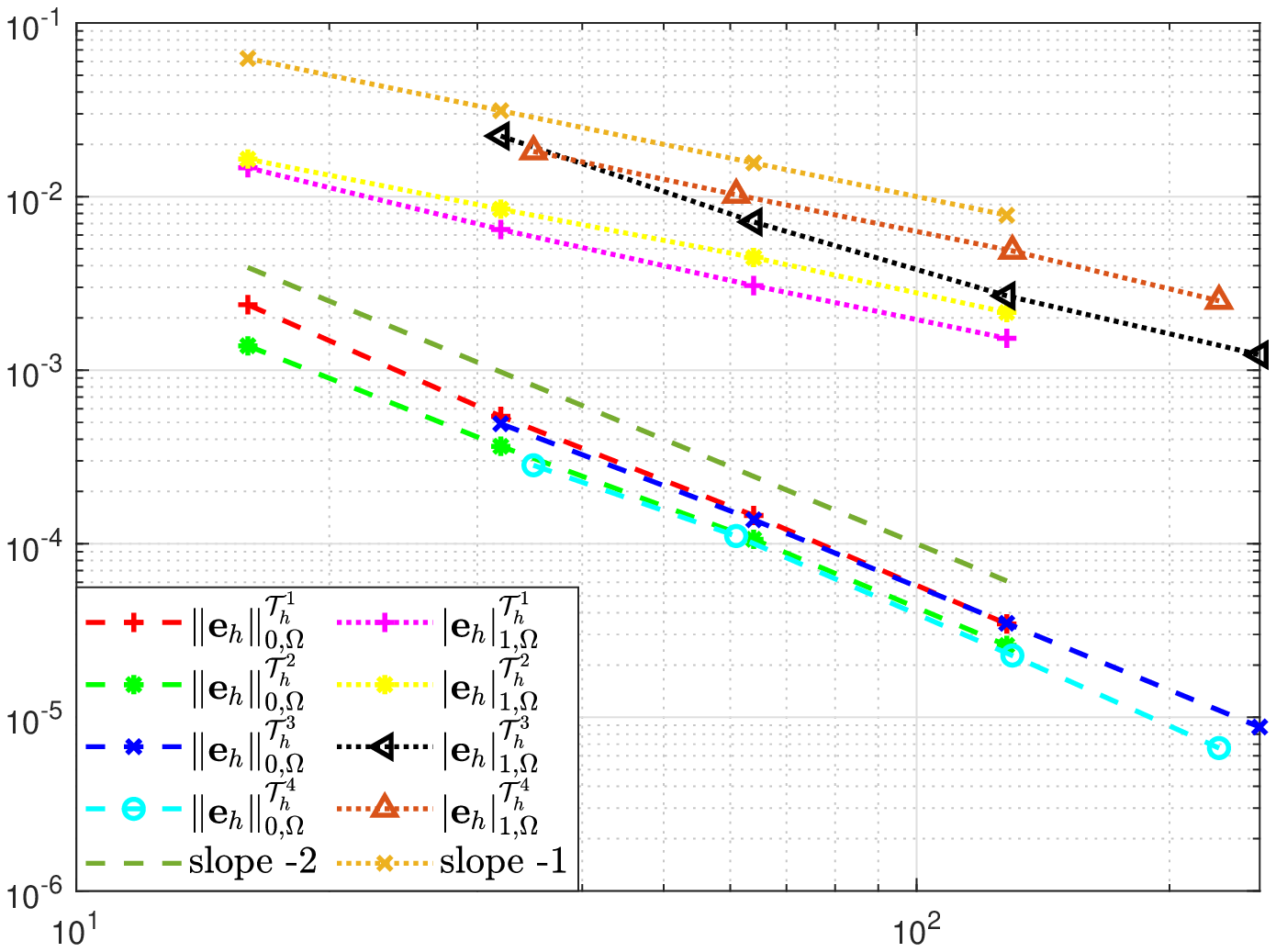}
			\centering\includegraphics[height=6.3cm, width=7.2cm]{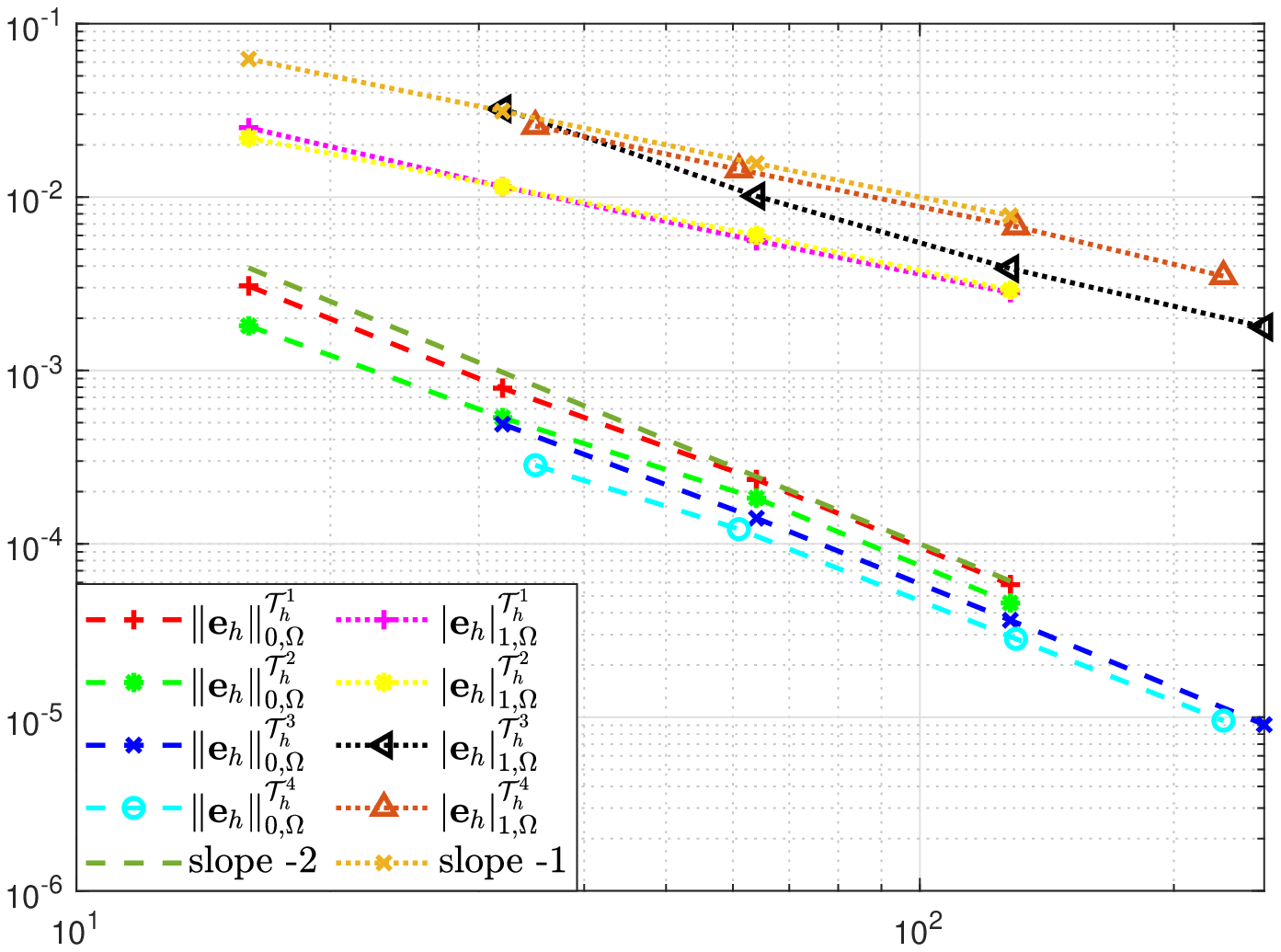}
			\caption{Plots of computed errors for different polygonal meshes. Left: Poisson ratio $\nu=0.35$, right: Poisson ratio $\nu = 0.45$, and stabilization \eqref{eq:derivative_stab}. 
}
\label{fig:plot1xy}
			\end{center}
			\end{figure}
			\begin{figure}[H]
			\begin{center}
			\centering\includegraphics[height=6.3cm, width=7.2cm]{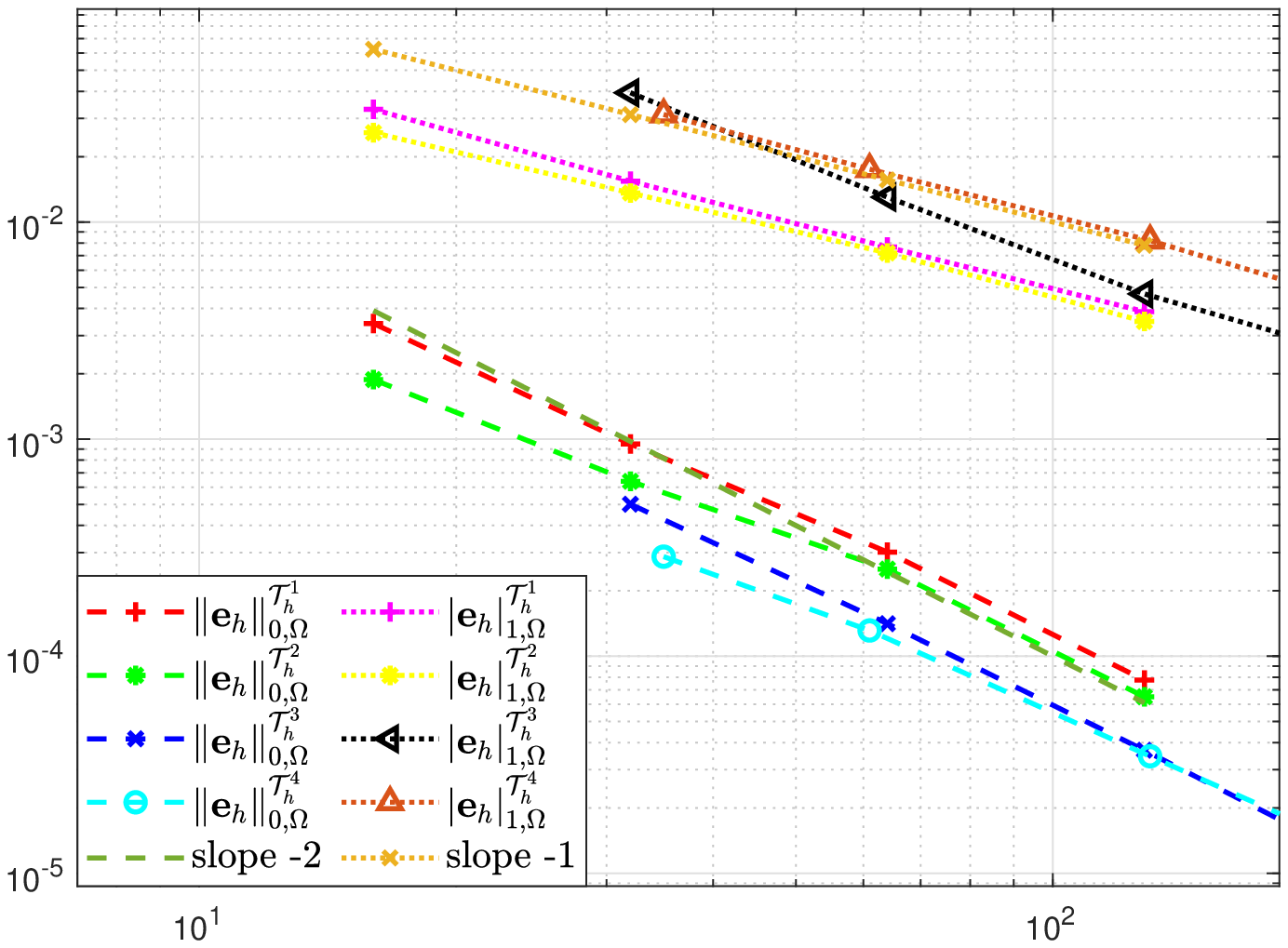}
			\centering\includegraphics[height=6.3cm, width=7.2cm]{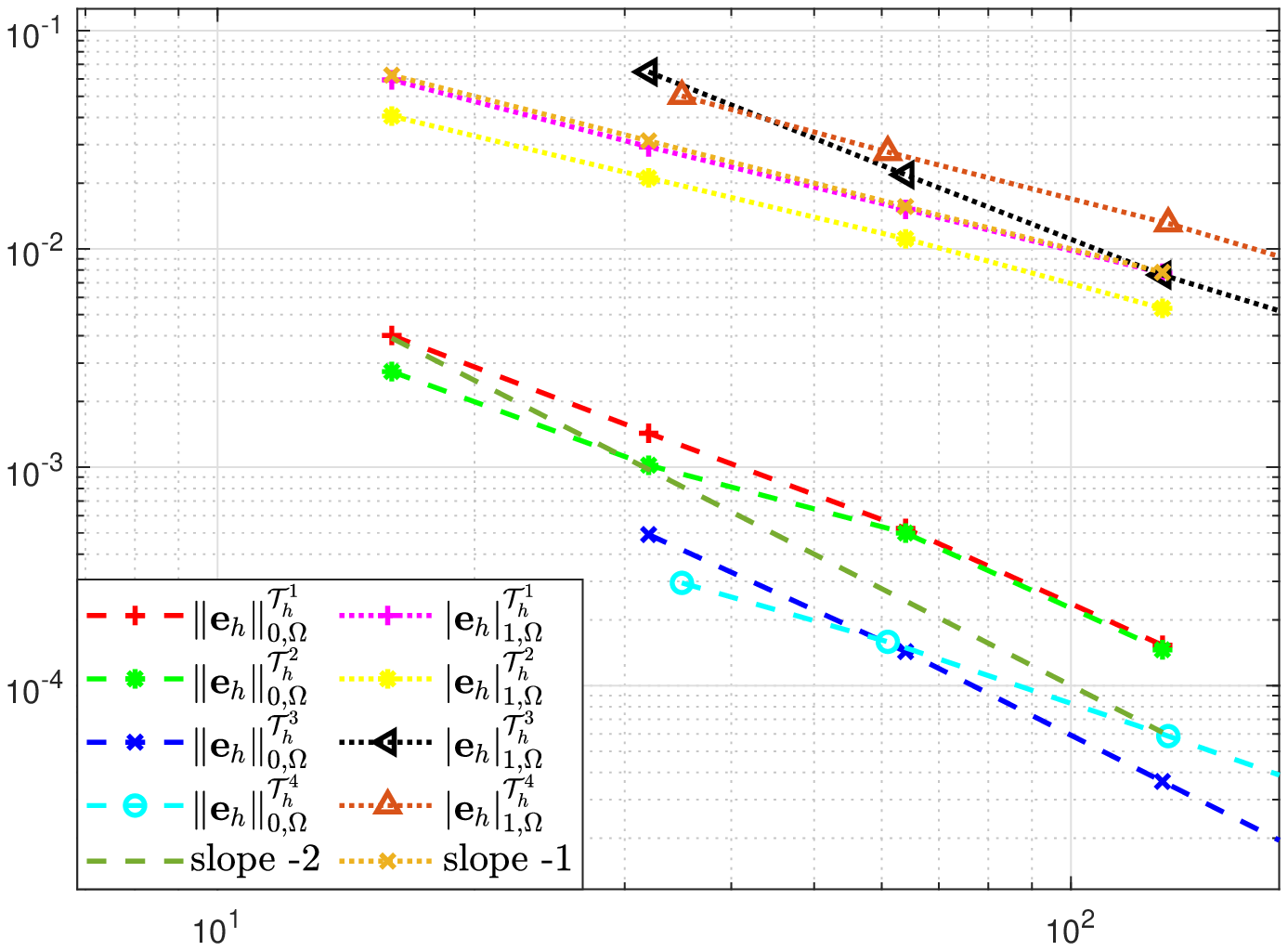}
		\caption{Plots of computed errors for different polygonal meshes. Left: Poisson ratio $\nu = 0.47$, right: Poisson ratio $\nu = 0.49$, and stabilization \eqref{eq:derivative_stab}. 
}
		\label{fig:plot1x}
	\end{center}
\end{figure}

\subsection{The classic stabilization}
Now our aim is to compare the results obtained with the stabilization defined in \eqref{eq:derivative_stab}, with the classic 
stabilization that us used in the VEM.  To make matters precise, the stabilization term for the following tests is given by 
\begin{equation}
\label{eq:classic_stabilization}
S(\textbf{w}_{h},\textbf{v}_{h}) = \displaystyle{\sum_{E \in \mathcal{T}_{h}} S^{E}(\textbf{w}_{h}, \textbf{v}_{h}), \quad S^{E}(\textbf{w}_{h}, \textbf{v}_{h})} = \displaystyle{\sum_{i = 1}^{N_{E}} \textbf{w}_{h}(V_{i})\textbf{v}_{h}(V_{i})},
\end{equation}
which corresponds to the evaluation on the vertices of the polygons. 

In what follows, we report error curves for the $\mathbf{L}^2$ norm and $\textbf{H}^{1}$ seminorm when the stabilization  \eqref{eq:classic_stabilization} is considered. These results have obtained, once again, for $\nu\in\{0.35, 0.45, 0.47, 0.49\}$ and the meshes presented in Figure \ref{fig:meshes}.
\begin{figure}[H]
	\begin{center}
\centering\includegraphics[height=6.3cm, width=7.2cm]{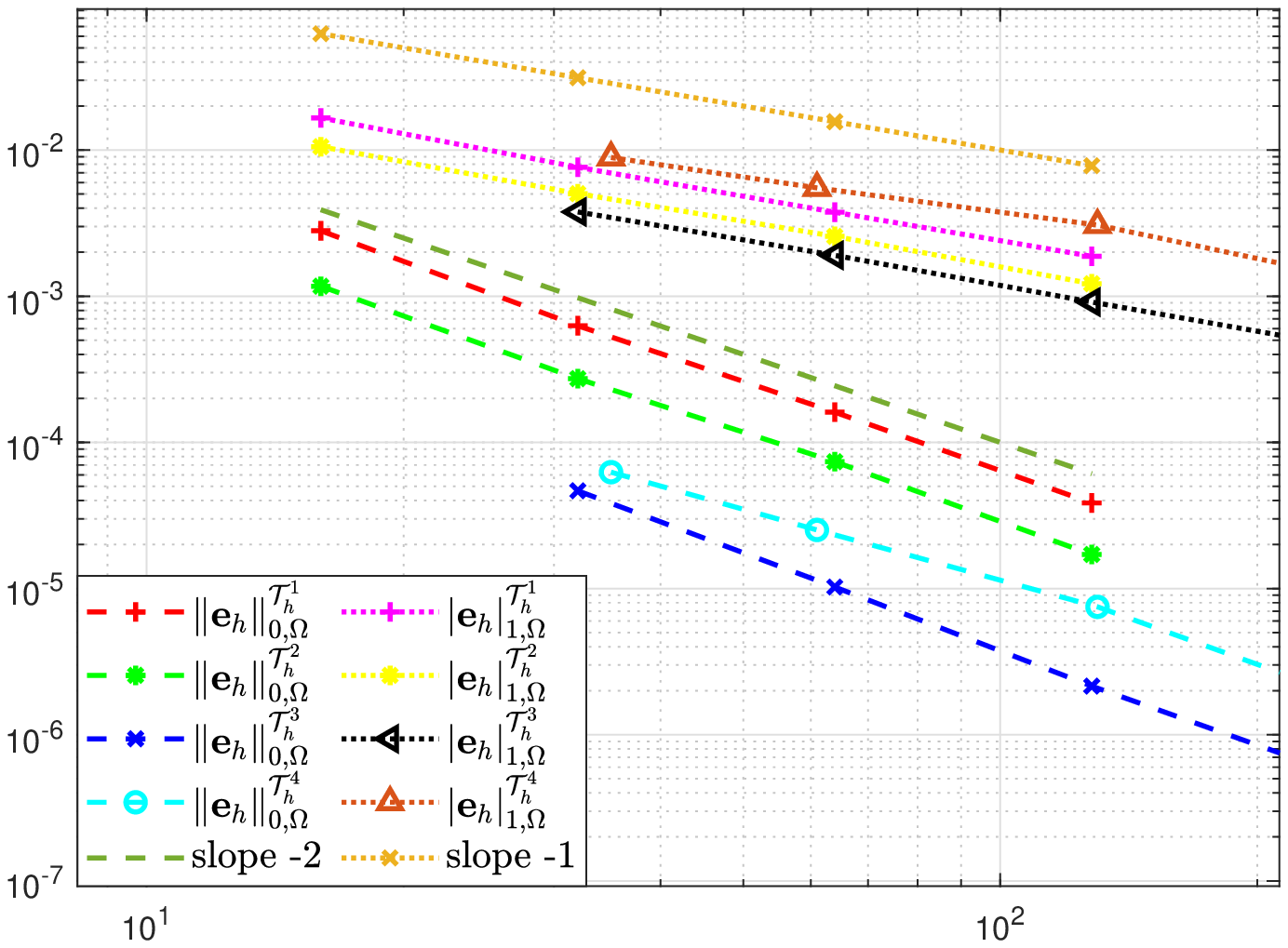}
			\centering\includegraphics[height=6.3cm, width=7.2cm]{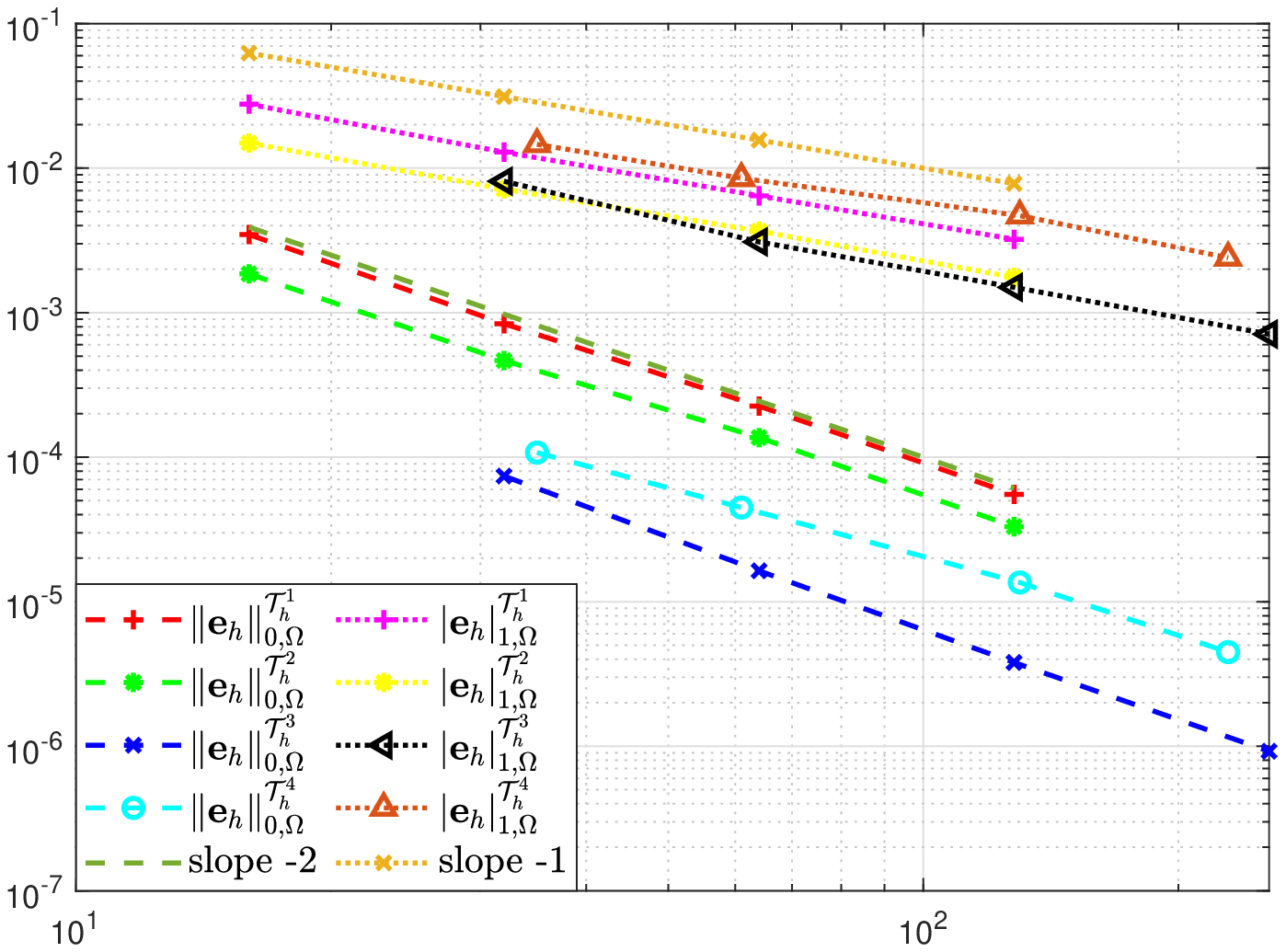}
		\caption{Plots of computed errors for different polygonal meshes. Left: Poisson ratio $\nu=0.35$, right: Poisson ratio $\nu = 0.45$, and stabilization \eqref{eq:classic_stabilization}. 
}
		\label{fig:plot4xy}
			\end{center}
			\end{figure}
			\begin{figure}[H]
			\begin{center}
						\centering\includegraphics[height=6.3cm, width=7.2cm]{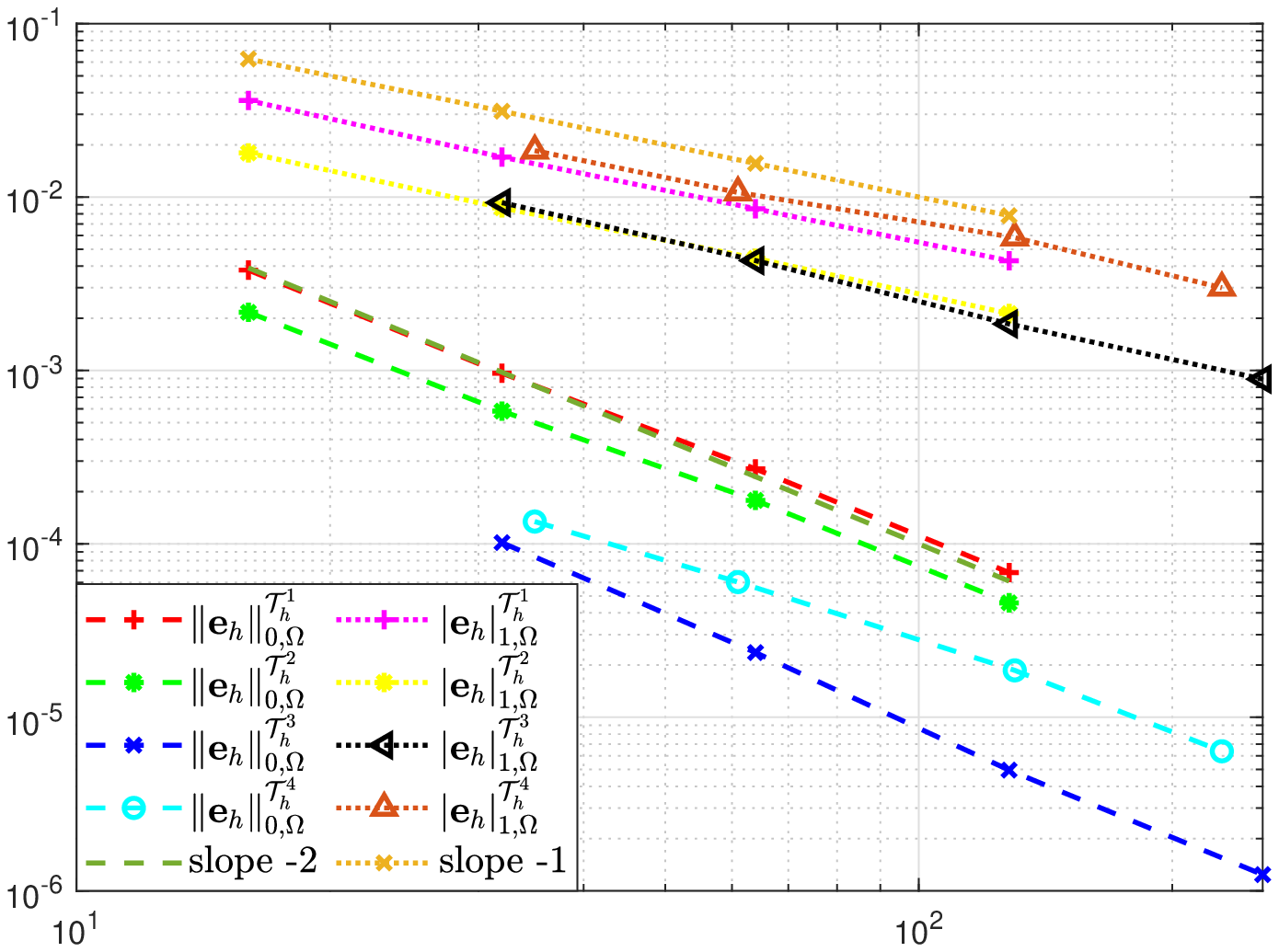}
						\centering\includegraphics[height=6.3cm, width=7.2cm]{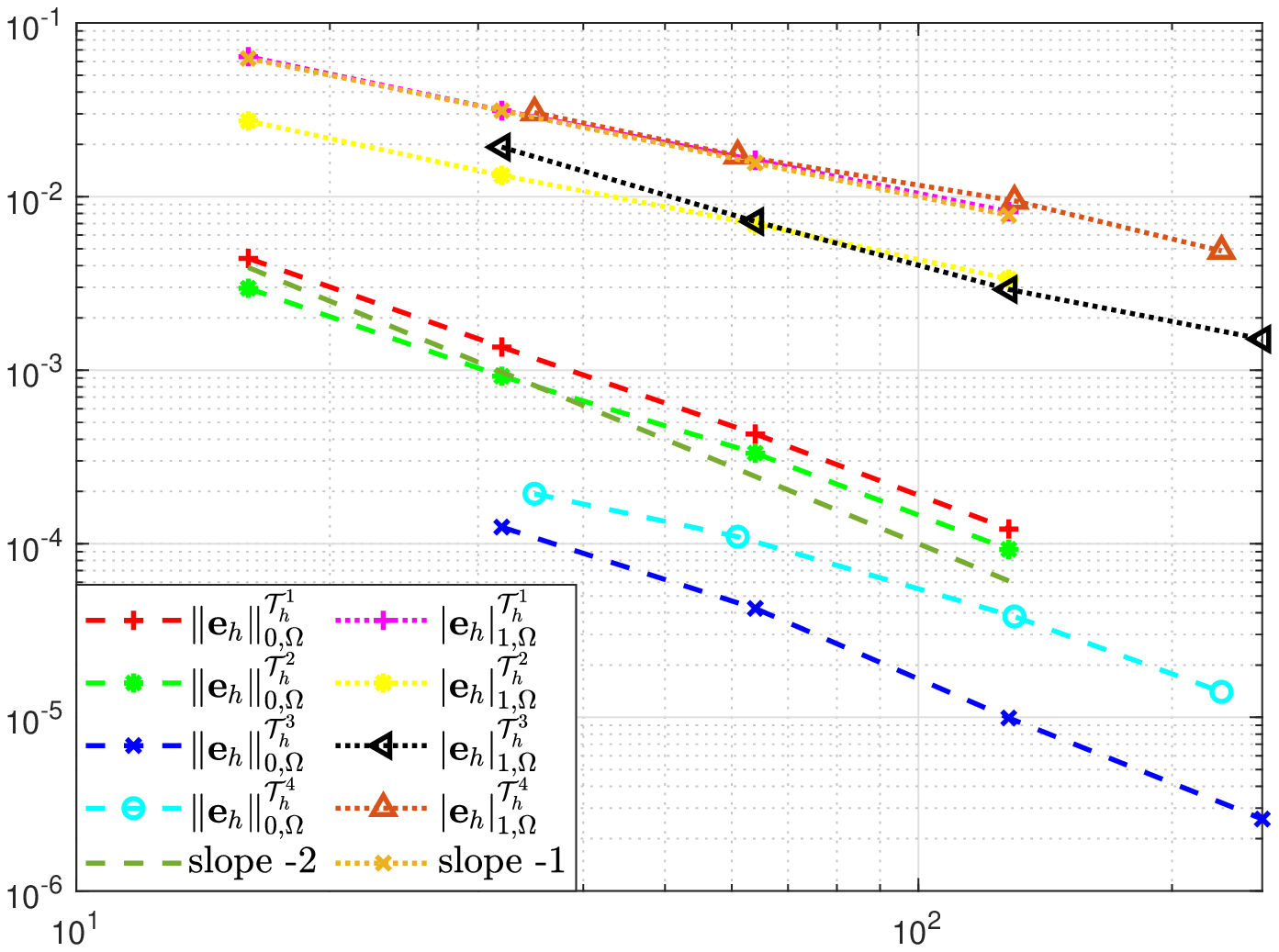}
		\caption{Plots of computed errors for different polygonal meshes. Left: Poisson ratio $\nu = 0.47$, right: Poisson ratio $\nu = 0.49$, and stabilization \eqref{eq:classic_stabilization}. 
}
		\label{fig:plot4}
	\end{center}
\end{figure}

From Figures  \ref{fig:plot4xy} and \ref{fig:plot4} we observe that the error in norm and seminorm
behave according to the theory, when the different Poisson ratios are involved in the computation of the solution. This allows us to conclude that our small edges approach for the elasticity equations, works well independent of the stabilization 
implemented on the computation al codes, reinforcing the idea that the VEM, and particularly the small edges approach, is a   versatile tool that approximates accurately the solutions of the elasticity equations.

Finally in Figures \ref{fig:plot7} and \ref{fig:plot8} we present plots of the magnitude of the displacement, where components of the approximated and exact solutions are presented. 
\begin{figure}[H]
	\begin{center}
			\centering\includegraphics[height=6.5cm, width=6.5cm]{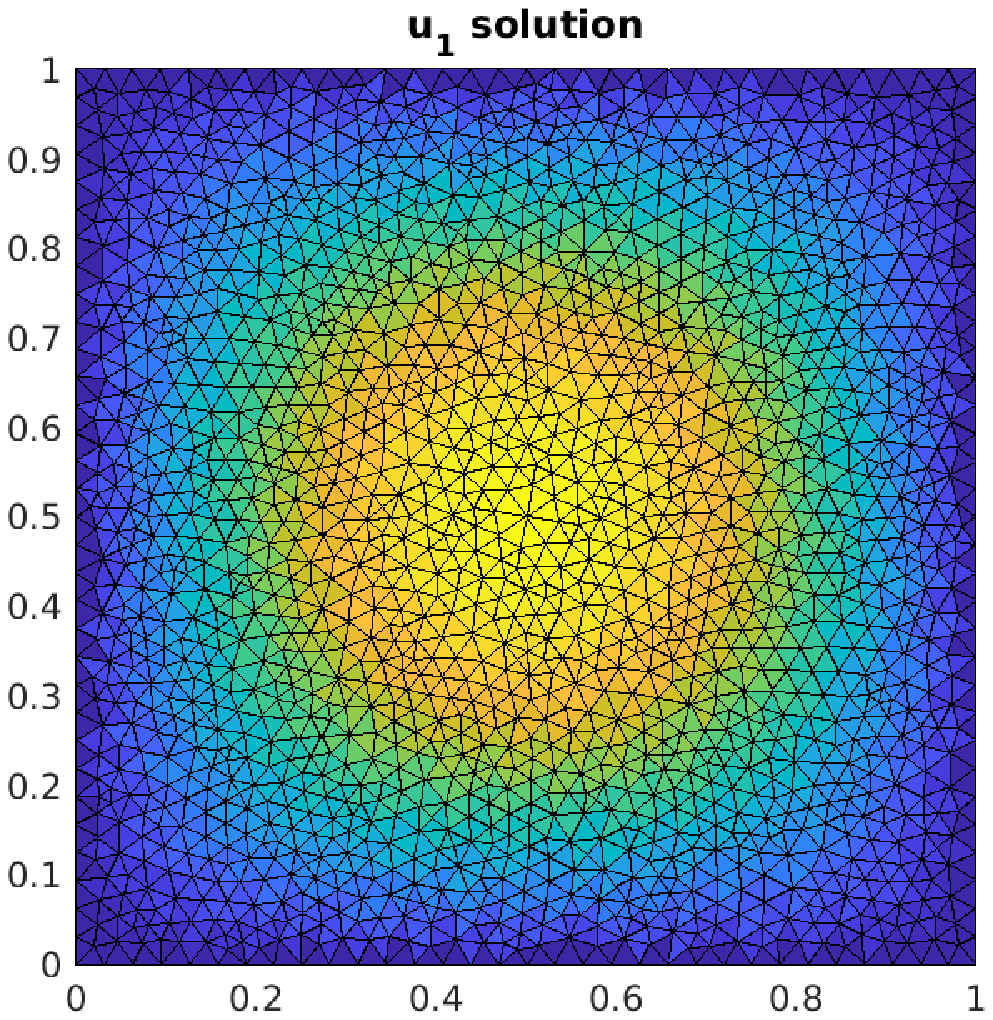}
			\centering\includegraphics[height=6.5cm, width=6.5cm]{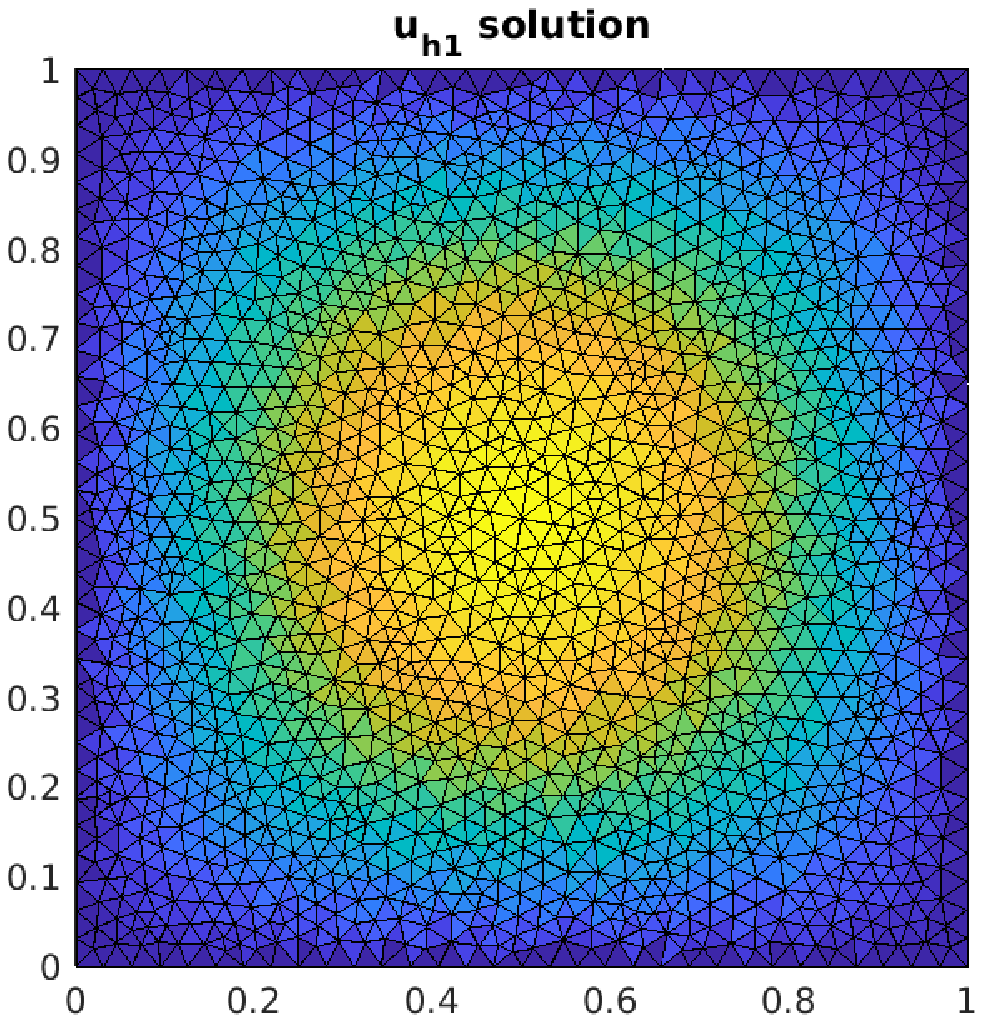}\\
		\caption{First components of the displacement $\boldsymbol{u}_1$. Left: exact solution, right: approximated solution, both computed with $\nu=0.35$.}
		\label{fig:plot7}
	\end{center}
\end{figure}

\begin{figure}[H]
	\begin{center}
                          \centering\includegraphics[height=6.5cm, width=6.5cm]{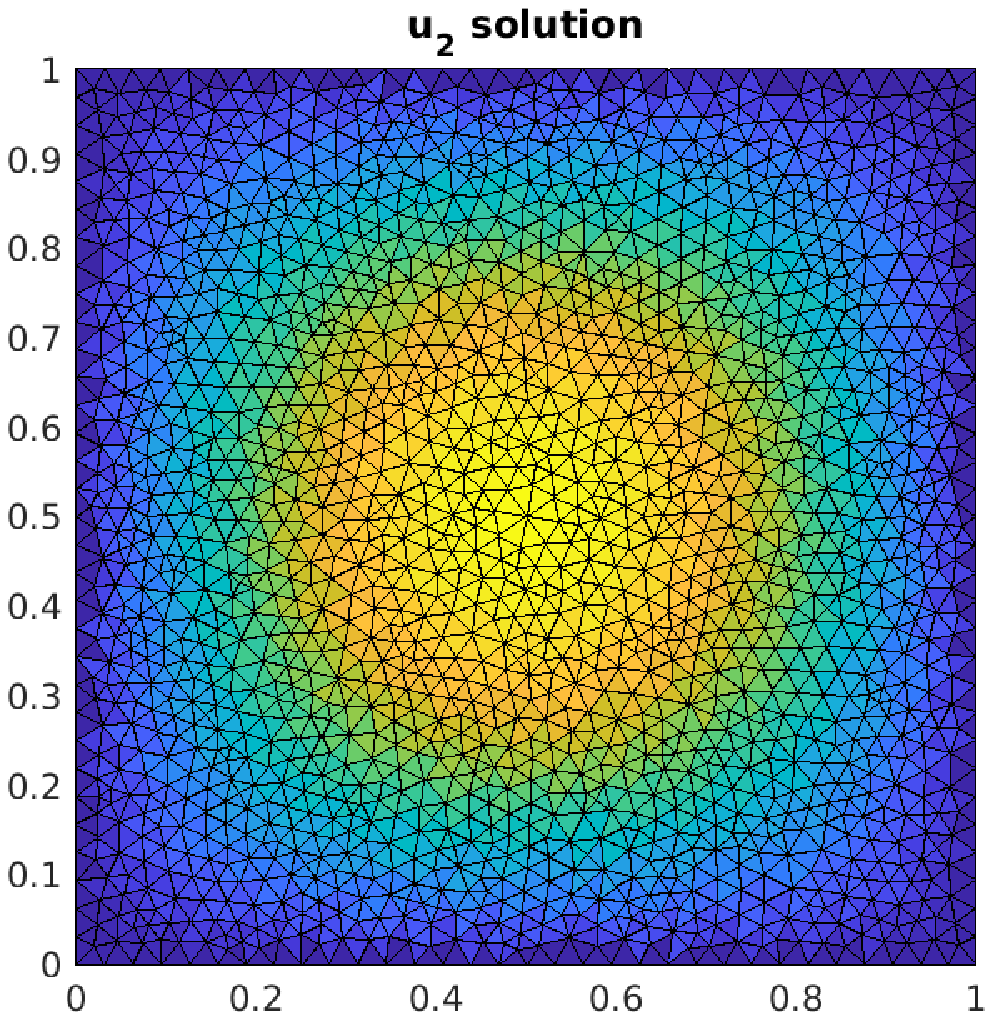}
                          \centering\includegraphics[height=6.5cm, width=6.5cm]{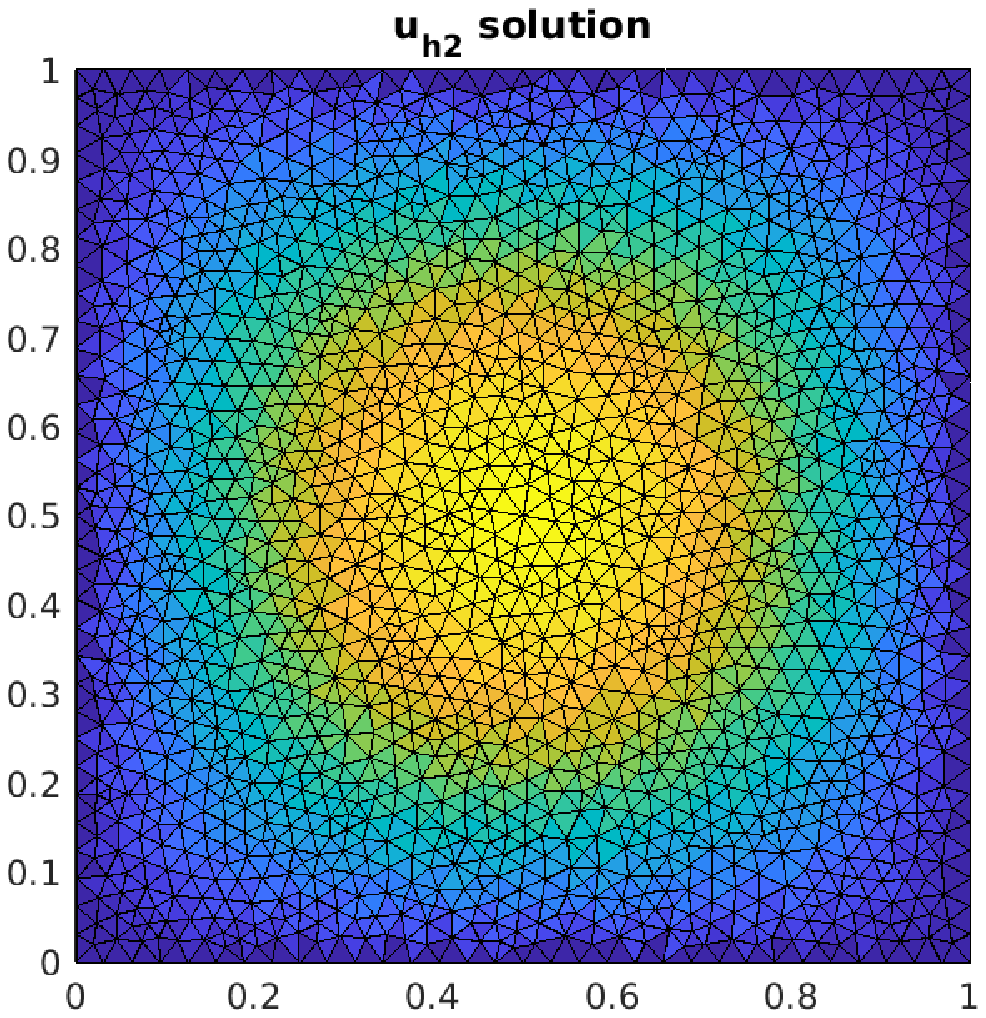}
	\caption{Second  components of the displacement $\boldsymbol{u}_2$. Left: exact solution, right: approximated solution, both computed with $\nu=0.35$.}

		\label{fig:plot8}
	\end{center}
\end{figure}

In order to confirm the robustness of the method, we present an additional experiment where we consider a mesh that we denote by  $\mathcal{T}_{h}^{v}$, which consists in the union of three different Voronoi mesh. This  mesh has the particularity that, where the meshes are glued, the nodes are sufficiently close each other, appearing polygons with arbitrary small edges. In Figure \ref{mesh:referato} we present an example of this mesh.
\begin{figure}[H]
\begin{center}
\centering\includegraphics[height=5cm, width=6cm]{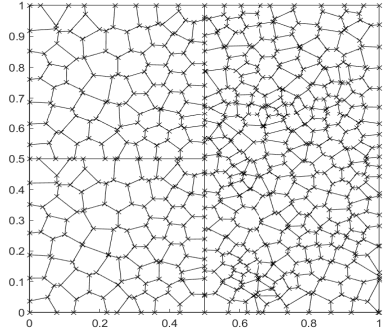}
\caption{\label{mesh:referato} Sample of the mesh $\mathcal{T}_{h}^{v}$, where the three meshes can be observed.} 
\end{center}
\end{figure}
In Figure \ref{curves:referato}, we present the plots of computed error for  $\mathcal{T}_{h}^{v}$ mesh and different values of Poisson ratio, namely $\nu\in\{0.35, 0.45, 0.47, 0.49\}$. Also, we present this experiment considering both stabilization terms \eqref{eq:derivative_stab} and \eqref{eq:classic_stabilization}. 
\begin{figure}[H]
\begin{center}
\centering\includegraphics[height=6.3cm, width=7.2cm]{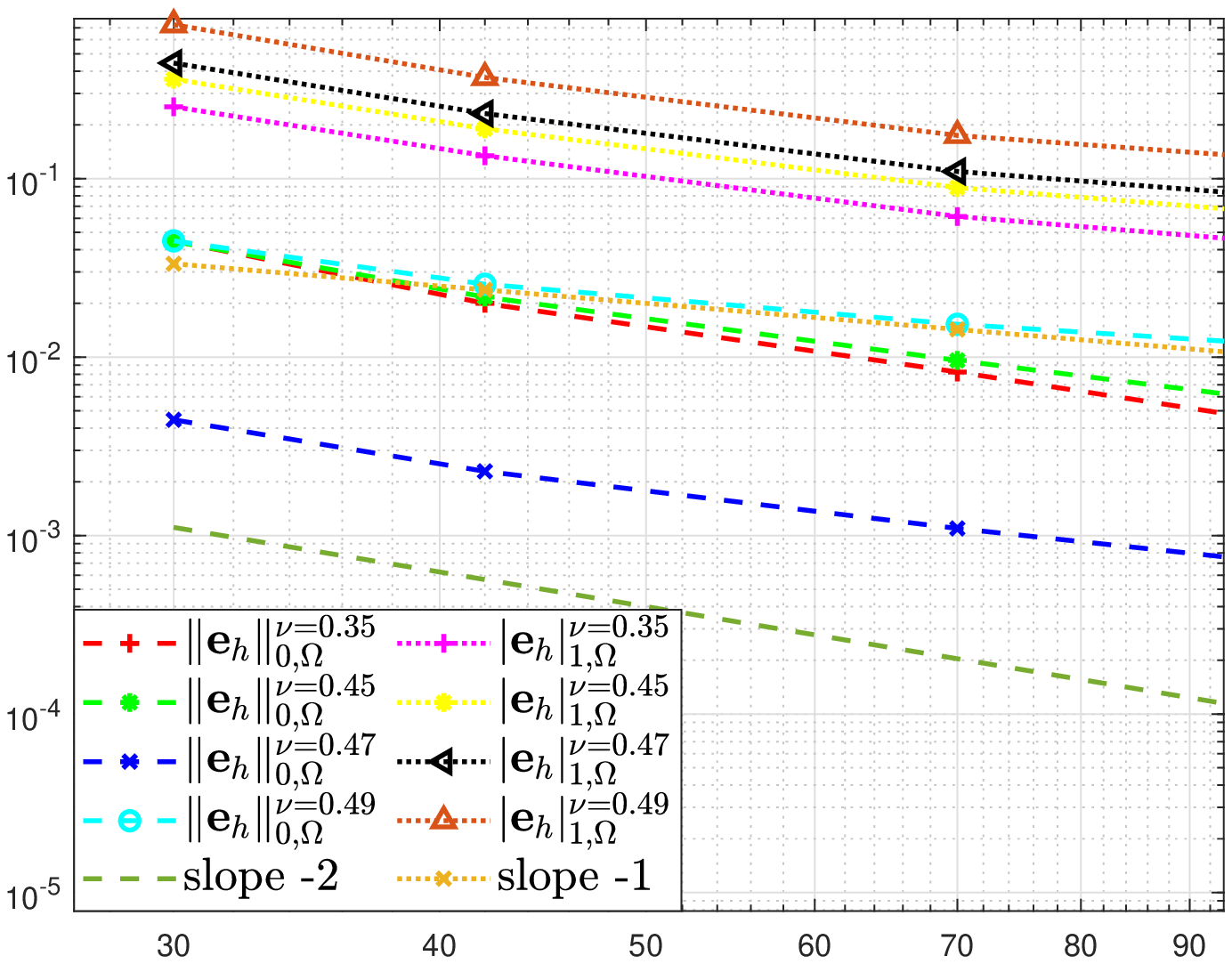}
\centering\includegraphics[height=6.3cm, width=7.2cm]{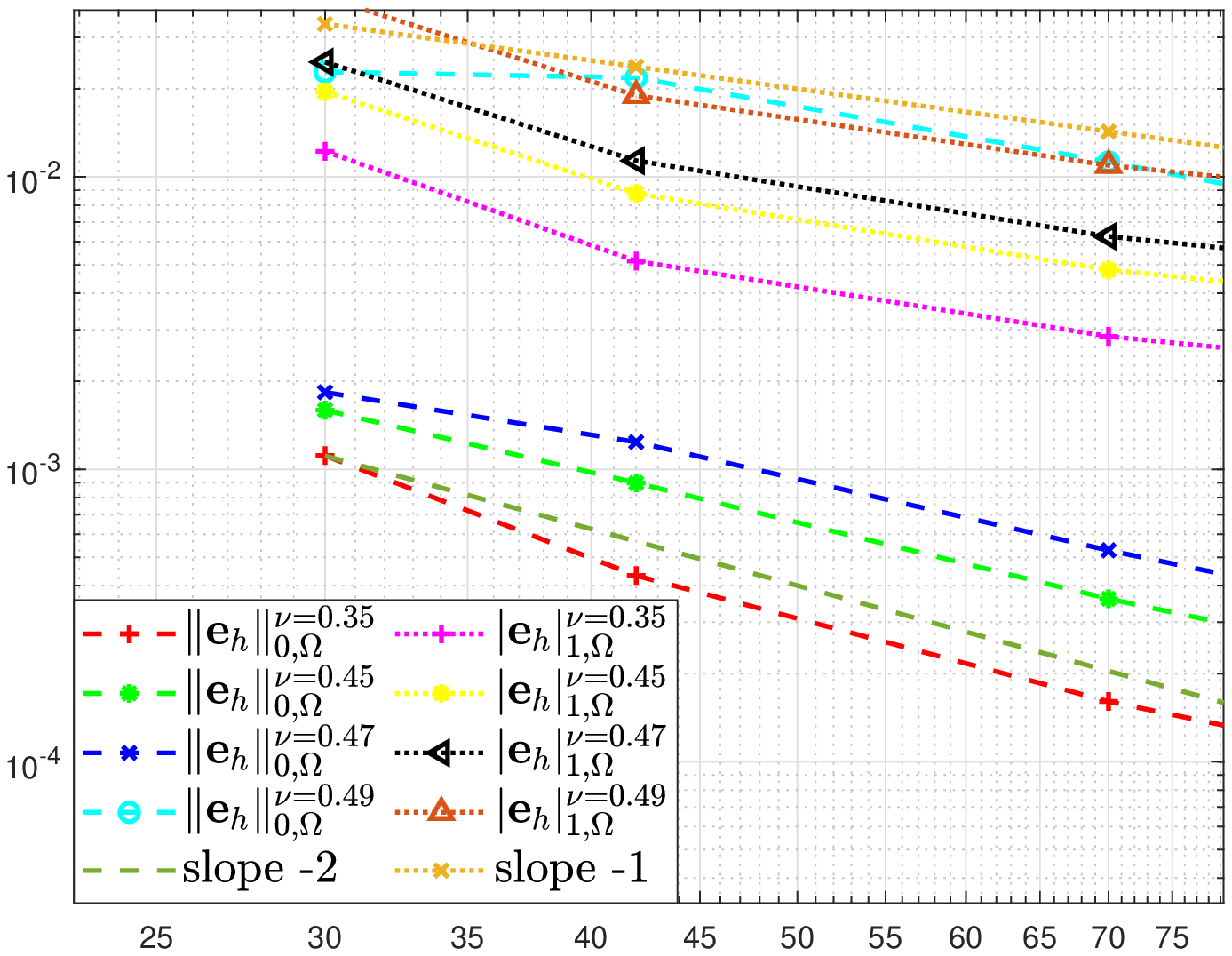}
\caption{\label{curves:referato} Plots of computed errors for different Poisson ratio and $\mathcal{T}_{h}^{v}$. Left: results for the stabilization term \eqref{eq:derivative_stab}, right: results for the stabilization term \eqref{eq:classic_stabilization}.}
\end{center}
\end{figure}

The results clearly show that the method is strongly precise with this more general mesh and for the Poisson ratios. More precisely, there are no significant differences for these meshes when we compare the results with the ones obtained when $\CT_h^1$,  $\CT_h^2$,  $\CT_h^3$, and  $\CT_h^4$ are considered. Moreover, the stabilizations  \eqref{eq:derivative_stab} and \eqref{eq:classic_stabilization} work similarly.

\subsection{Nonconvex domain}
Now we perform a test that goes beyond the theory that we have developed, which assumes the convexity of the domain $\O$ in order to
obtain optimal order of convergence for the proposed VEM method. For this test, we consider a nonconvex domain that we call the L-shaped domain as is defined by $\O:=(0,2)^2\setminus [1,2)^2$. Clearly
the geometry of this domain presents a geometrical singularity that leads to approximate a solution for the elasticity problem that is no sufficiently smooth. Hence, the VEM method will be not capable 
to achieve the optimal order of convergence under this configuration.

%
%

The polygonal meshes that we will consider for our tests are the following: 
\begin{itemize}
\item[$\bullet$] $\mathcal{T}_{h}^{1}$: Triangles with small edges.
\item[$\bullet$] $\mathcal{T}_{h}^{2}$: Deformed triangles with middle points.
\item[$\bullet$] $\mathcal{T}_{h}^{3}$: Voronoi-squares-deformed squares mixed. 
\item[$\bullet$] $\mathcal{T}_{h}^{4}$: Voronoi.
\end{itemize}

In Figure \ref{fig:meshes2} we present examples of these meshes when the L-shaped domain is considered.
\begin{figure}[h !]
	\begin{center}
			\centering\includegraphics[height=5cm, width=6cm]{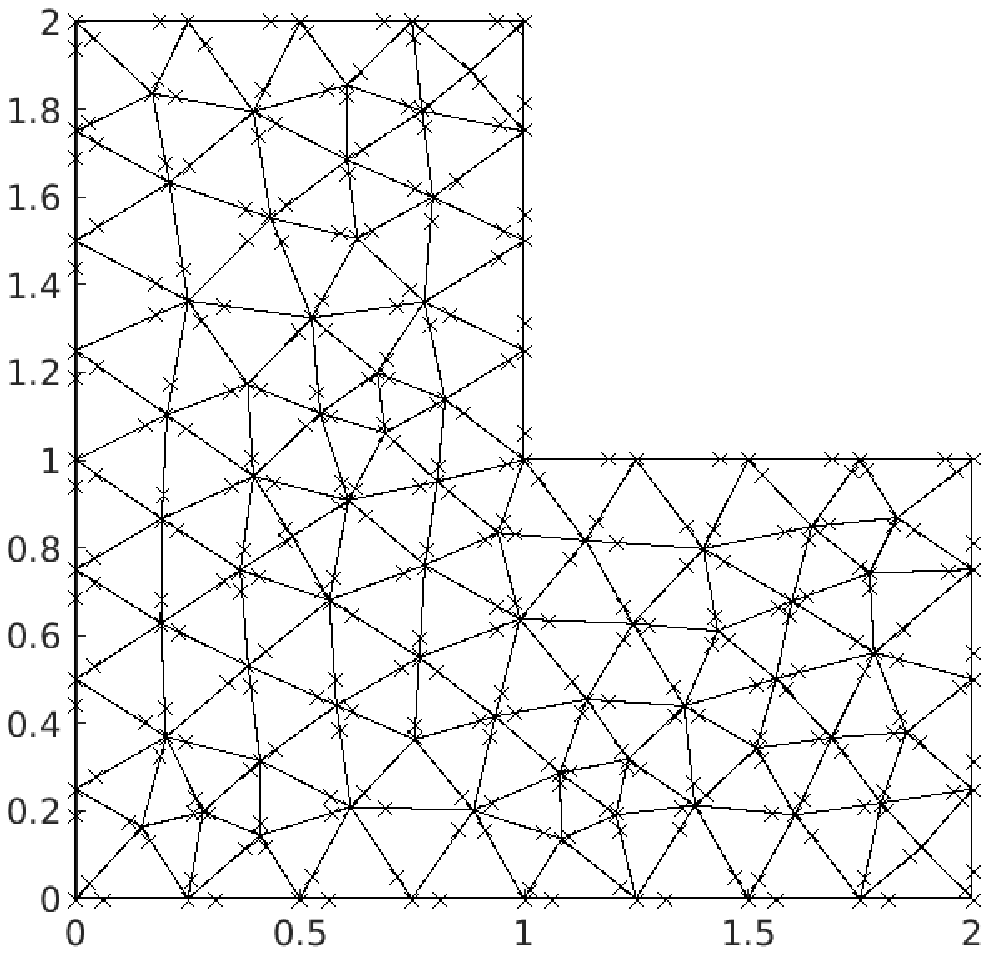}
			\centering\includegraphics[height=5cm, width=6cm]{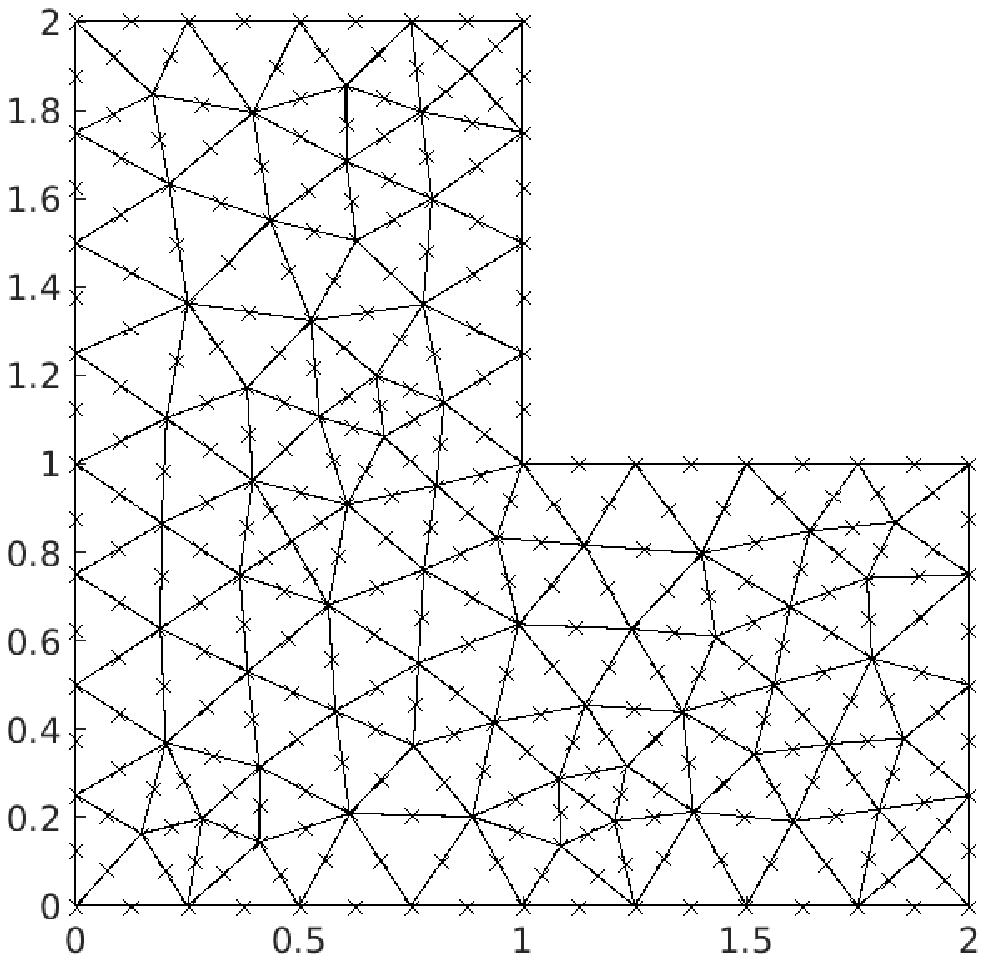}\\
                          \centering\includegraphics[height=5cm, width=6cm]{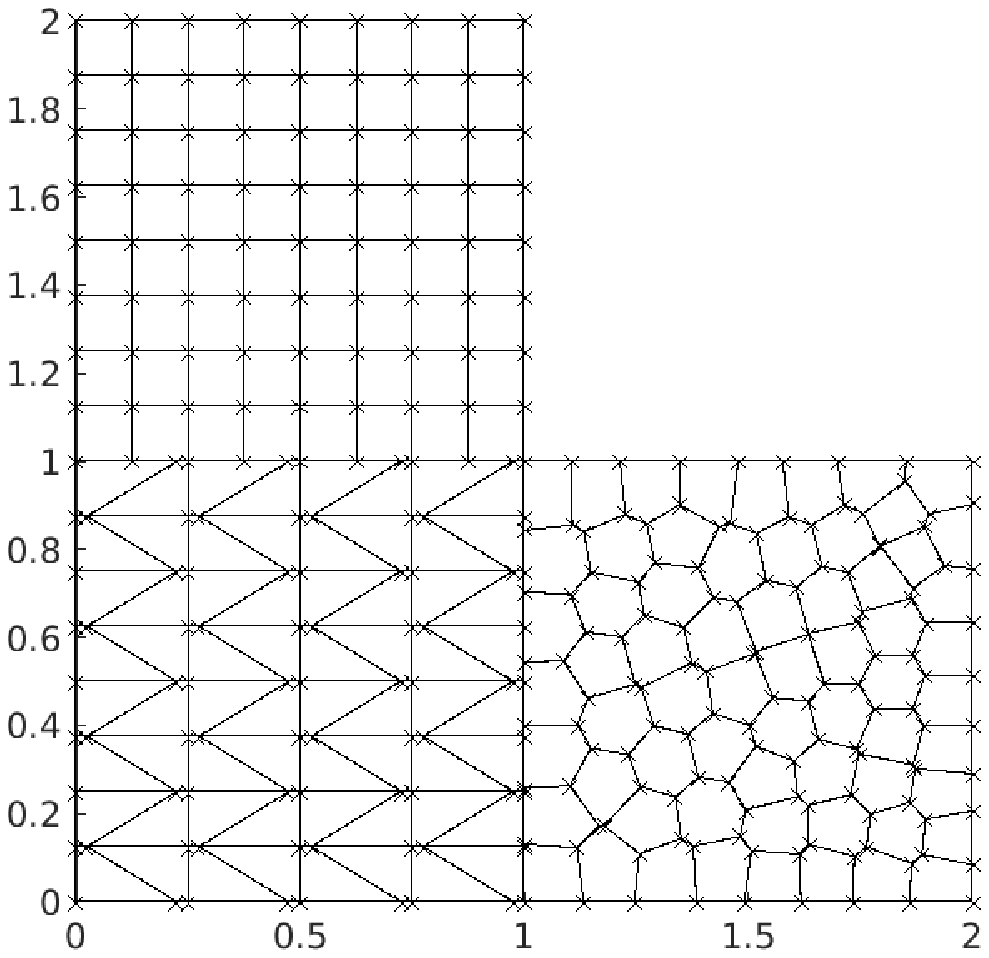}
                          \centering\includegraphics[height=5cm, width=6cm]{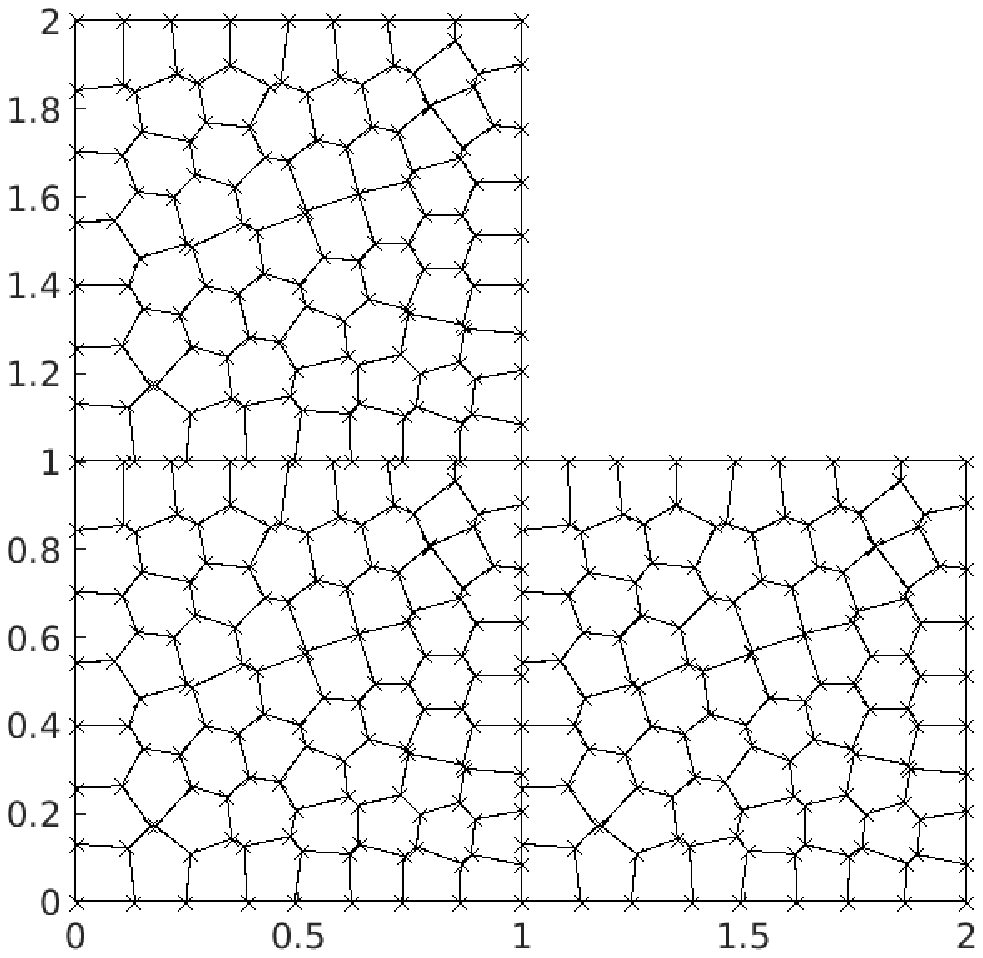}
		\caption{Example of the meshes for our experimentations. Top left: $\mathcal{T}_{h}^{1}$, top right $\mathcal{T}_{h}^{2}$, bottom left $\mathcal{T}_{h}^{3}$ and bottom right $\mathcal{T}_{h}^{4}$.
}
		\label{fig:meshes2}
	\end{center}
\end{figure}

For our experiments, we will consider the classic stabilization term \eqref{eq:classic_stabilization}. In what follows, we report error curves for the $\mathbf{L}^2$ norm and seminorm when the stabilization  \eqref{eq:classic_stabilization} is considered. These results have obtained for $\nu\in\{0.35, 0.45, 0.47, 0.49\}$ and the meshes presented in Figure \ref{fig:meshes2}.
\begin{figure}[H]
	\begin{center}
			\centering\includegraphics[height=6.3cm, width=7.2cm]{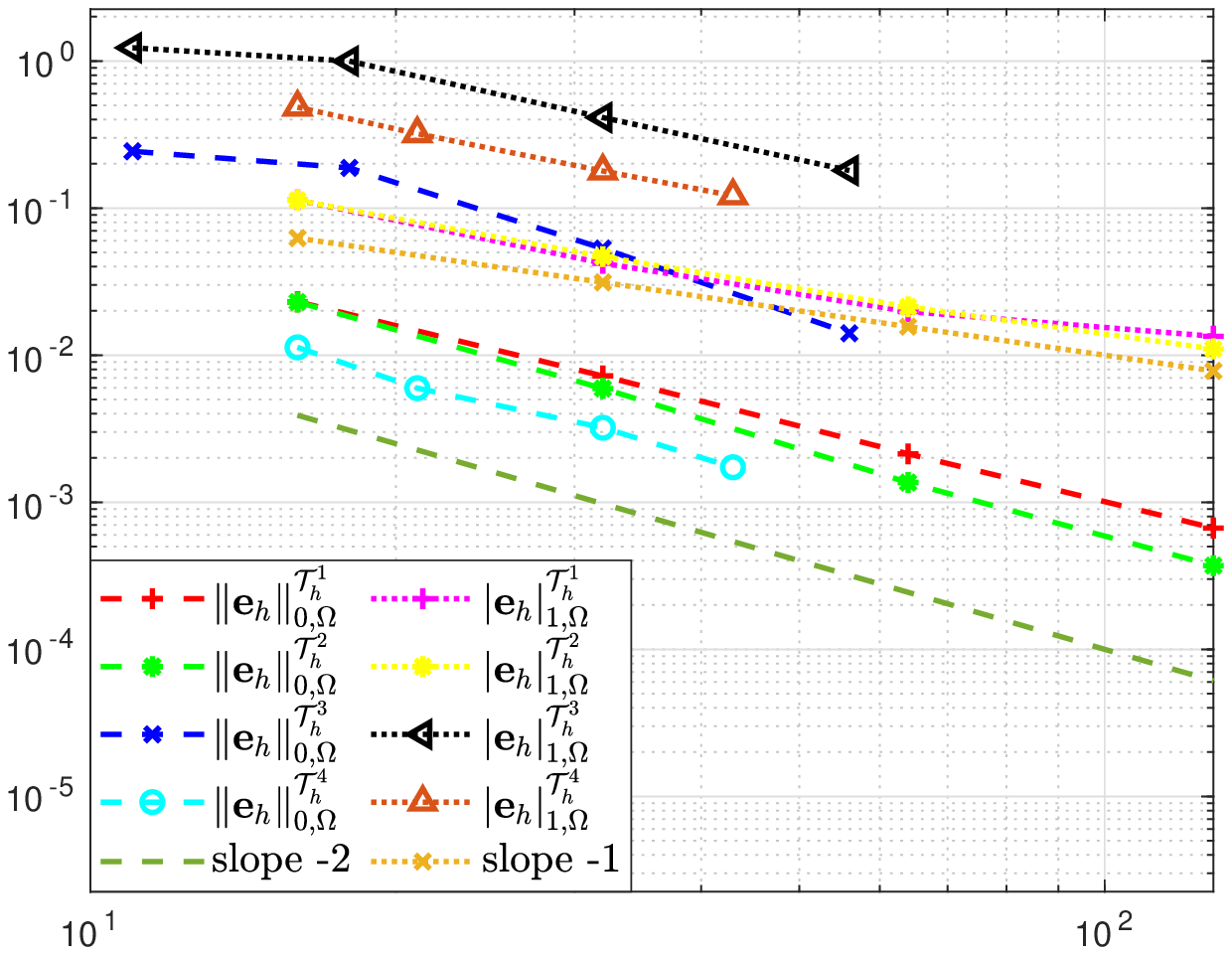}
			\centering\includegraphics[height=6.3cm, width=7.2cm]{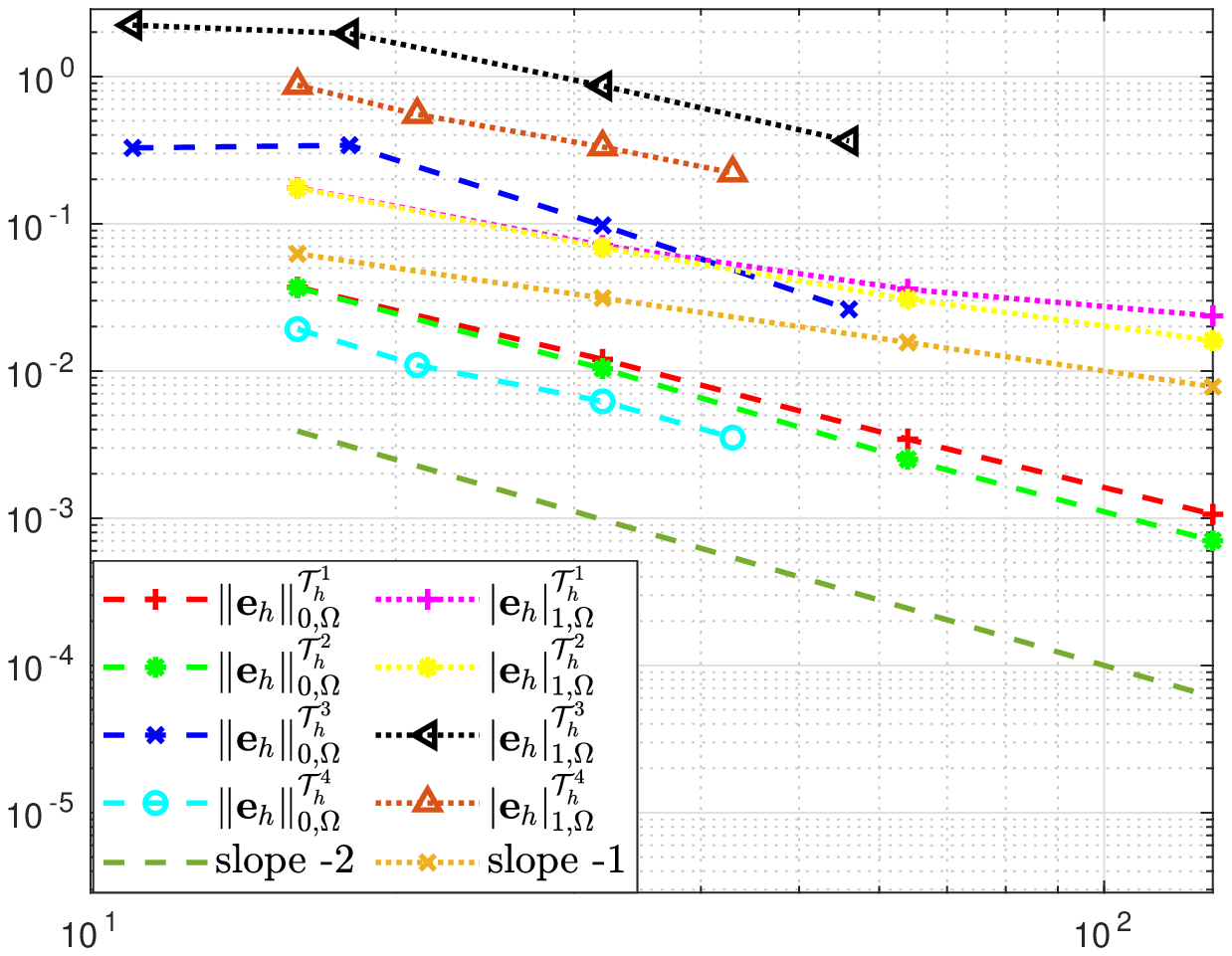}
		\caption{Plots of computed errors for different polygonal meshes. Left: Poisson ratio  $\nu=0.35$, right: Poisson ratio $\nu = 0.45$, and stablization \eqref{eq:classic_stabilization}. 
}
		\label{fig:plot4xzx}
\end{center}
\end{figure}
\begin{figure}[H]
\begin{center}
			\centering\includegraphics[height=6.3cm, width=7.2cm]{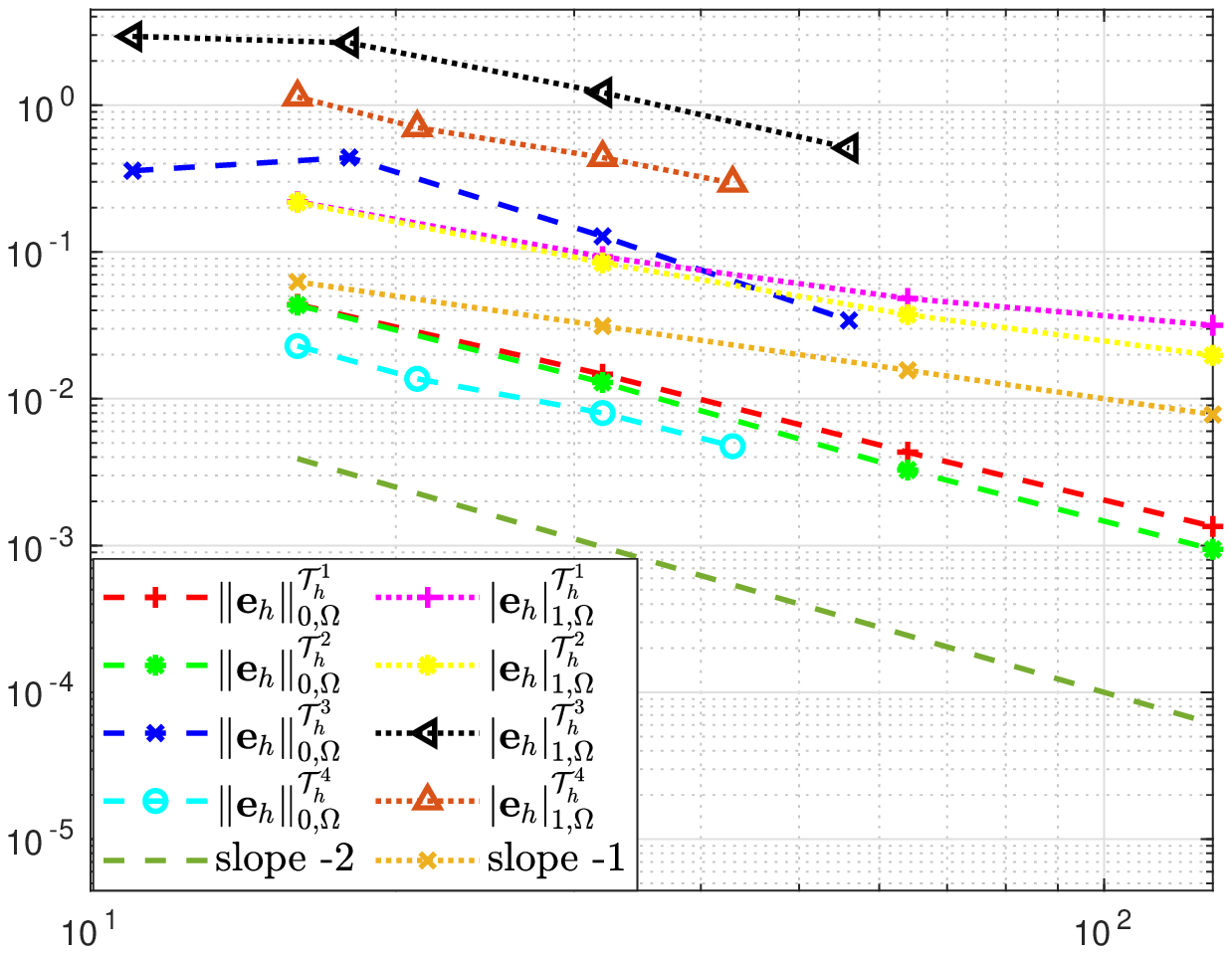} 
			\centering\includegraphics[height=6.3cm, width=7.2cm]{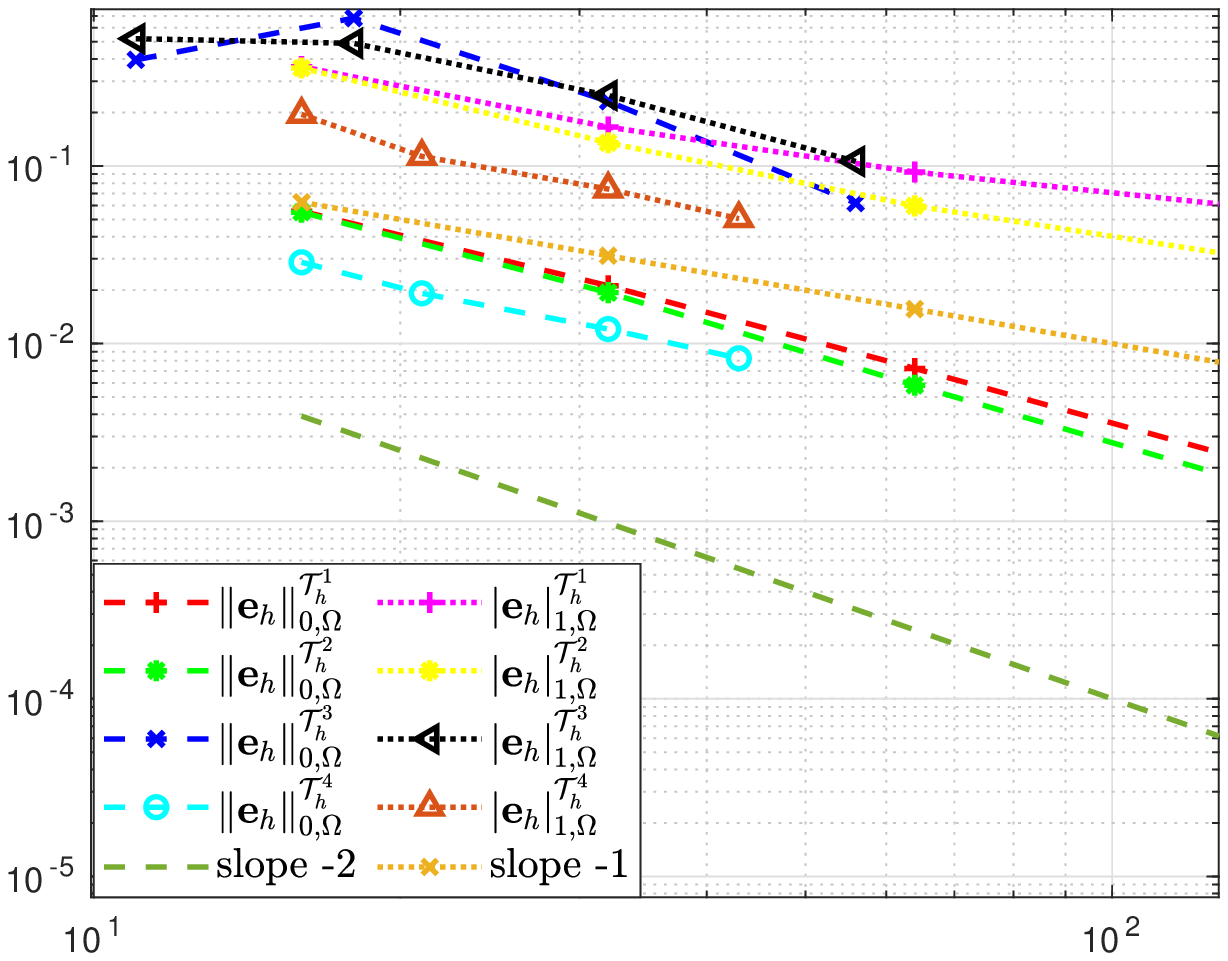}
		\caption{Plots of computed errors for different polygonal meshes. Left: Poisson ratio $\nu = 0.47$, right: Poisson ratio $\nu = 0.49$, and stablization \eqref{eq:classic_stabilization}. 
}
		\label{fig:plot4z}
	\end{center}
\end{figure}


Finally in Figure \ref{fig:plot9} and Figure \ref{fig:plot10} we present plots of the magnitude of the displacement, where components of the approximated and exact solutions are presented. 
\begin{figure}[H]
	\begin{center}
			\centering\includegraphics[height=6.3cm, width=7.2cm]{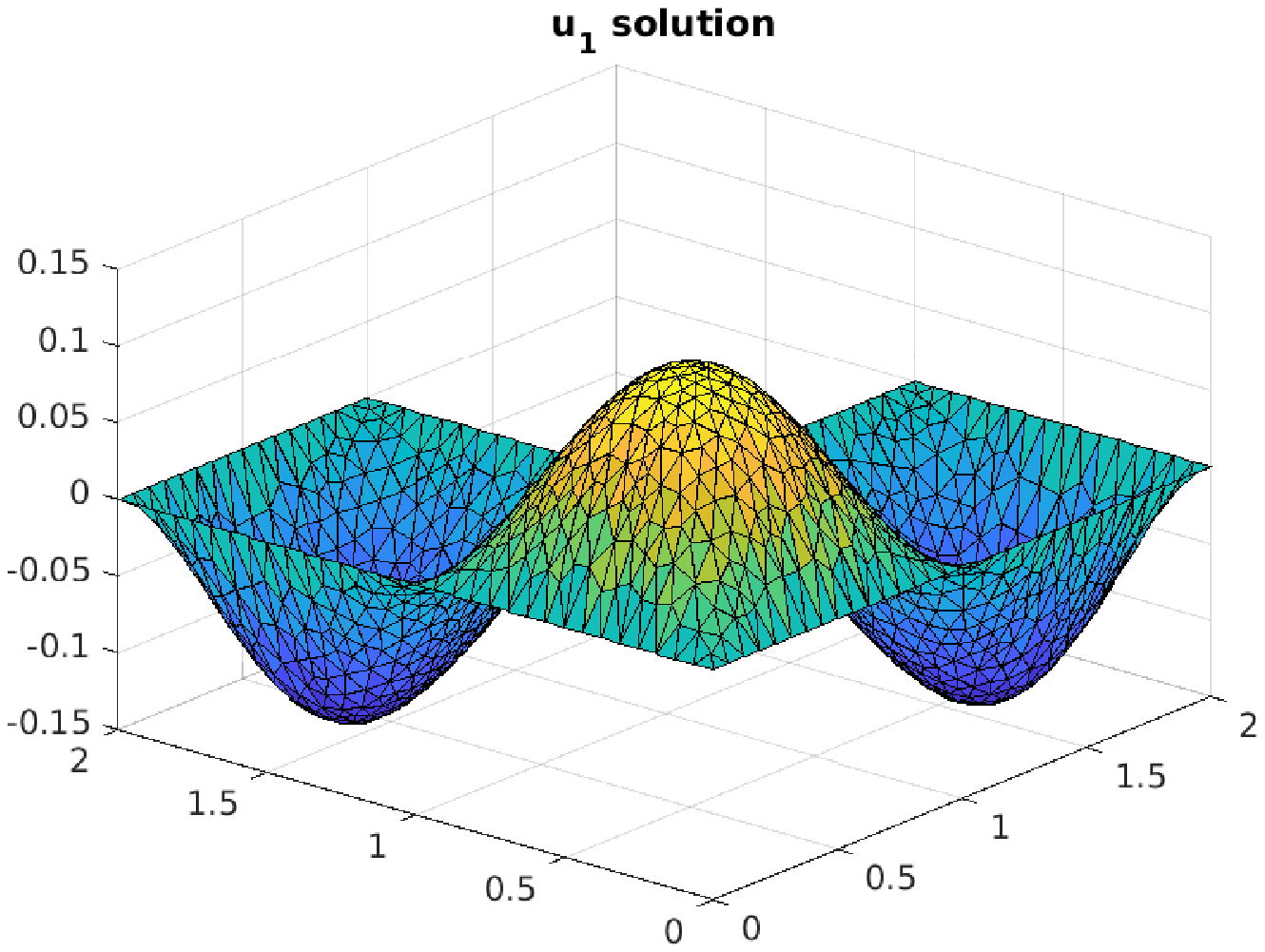}
			\centering\includegraphics[height=6.3cm, width=7.2cm]{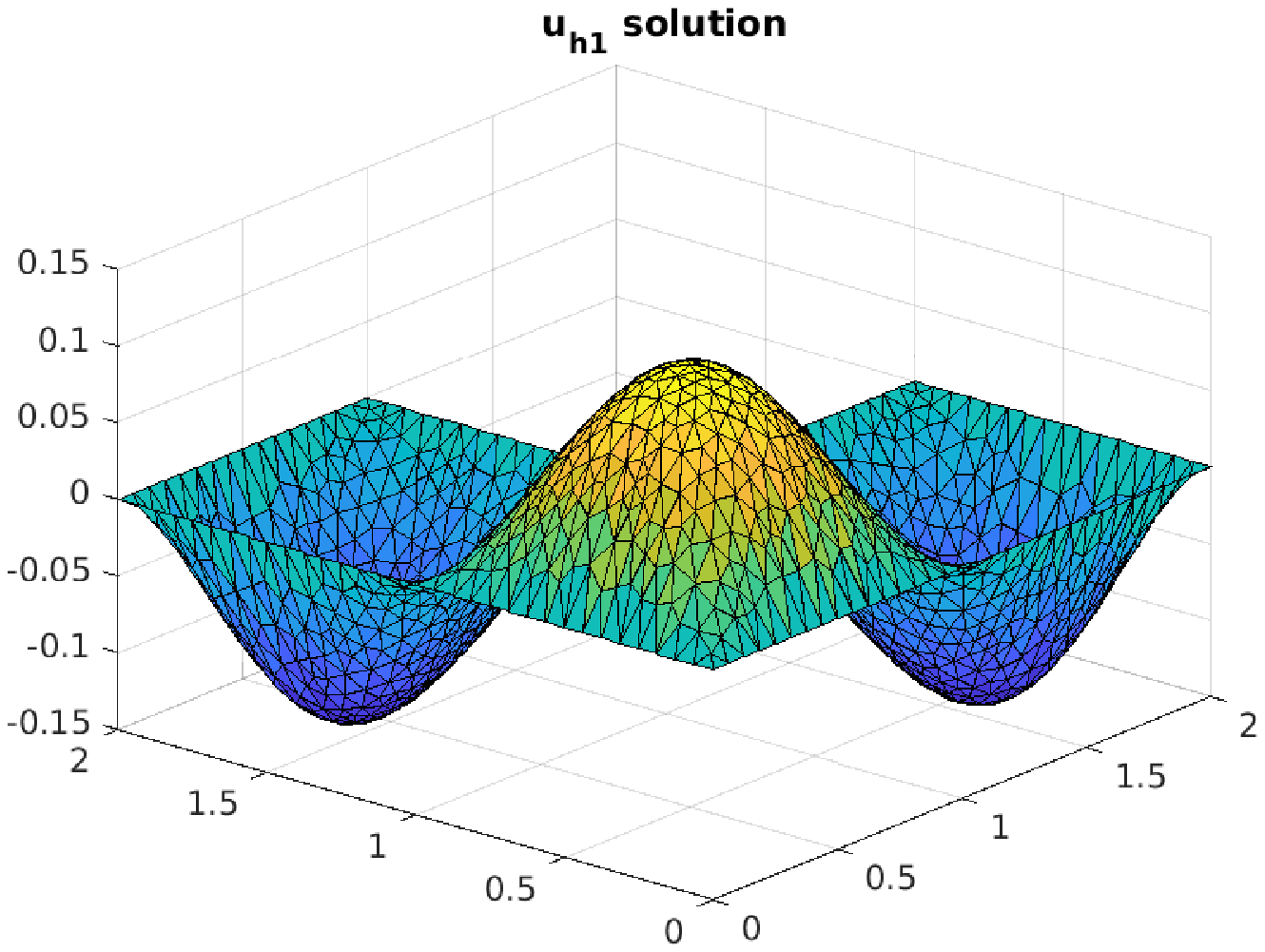}\\
		\caption{First components of the displacement $\boldsymbol{u}_1$. Left: exact solution, right: approximated solution,  both computed with $\nu=0.35$.}
		\label{fig:plot9}
	\end{center}
\end{figure}

\begin{figure}[H]
	\begin{center}
                          \centering\includegraphics[height=6.3cm, width=7.2cm]{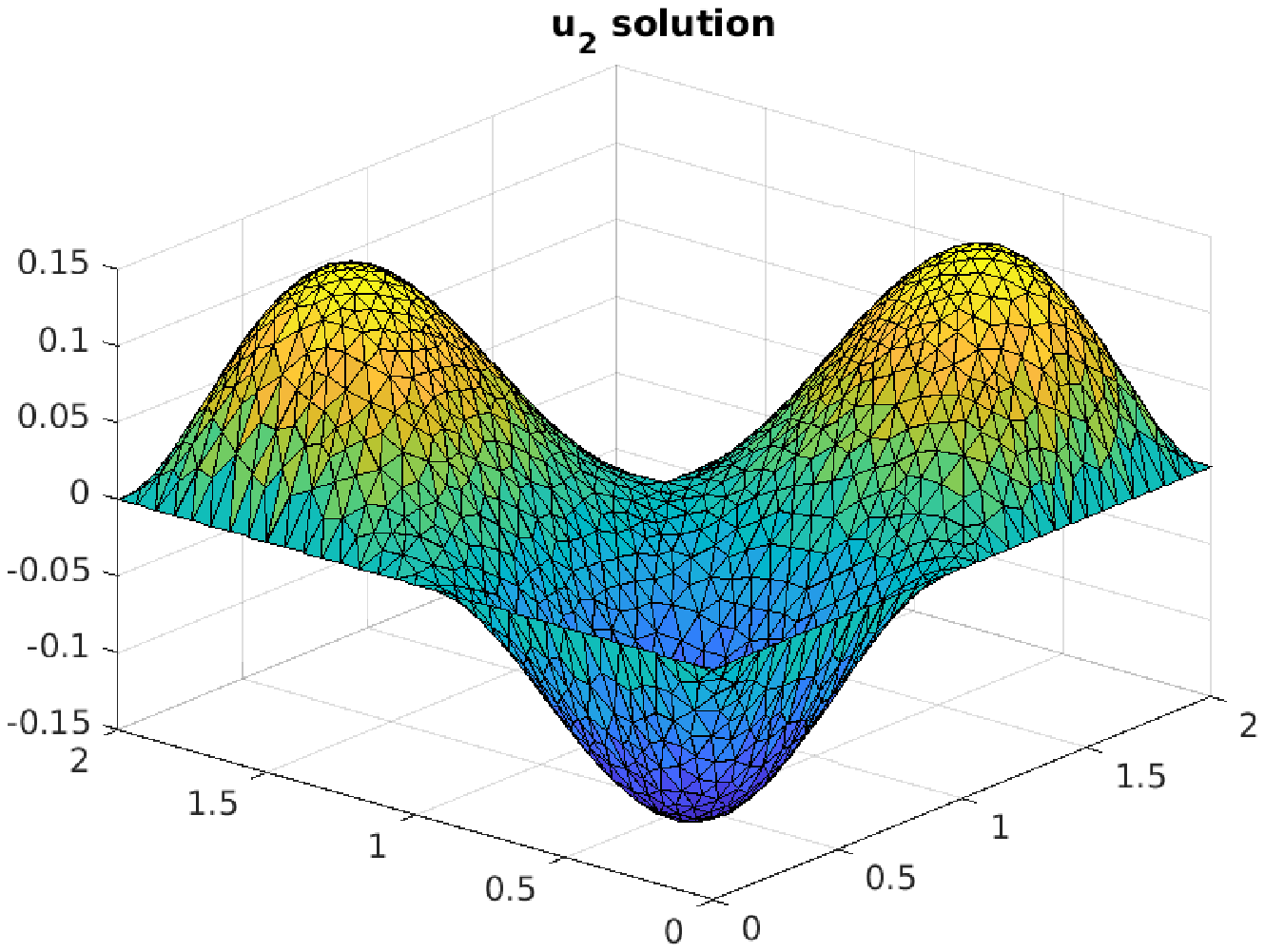}
                          \centering\includegraphics[height=6.3cm, width=7.2cm]{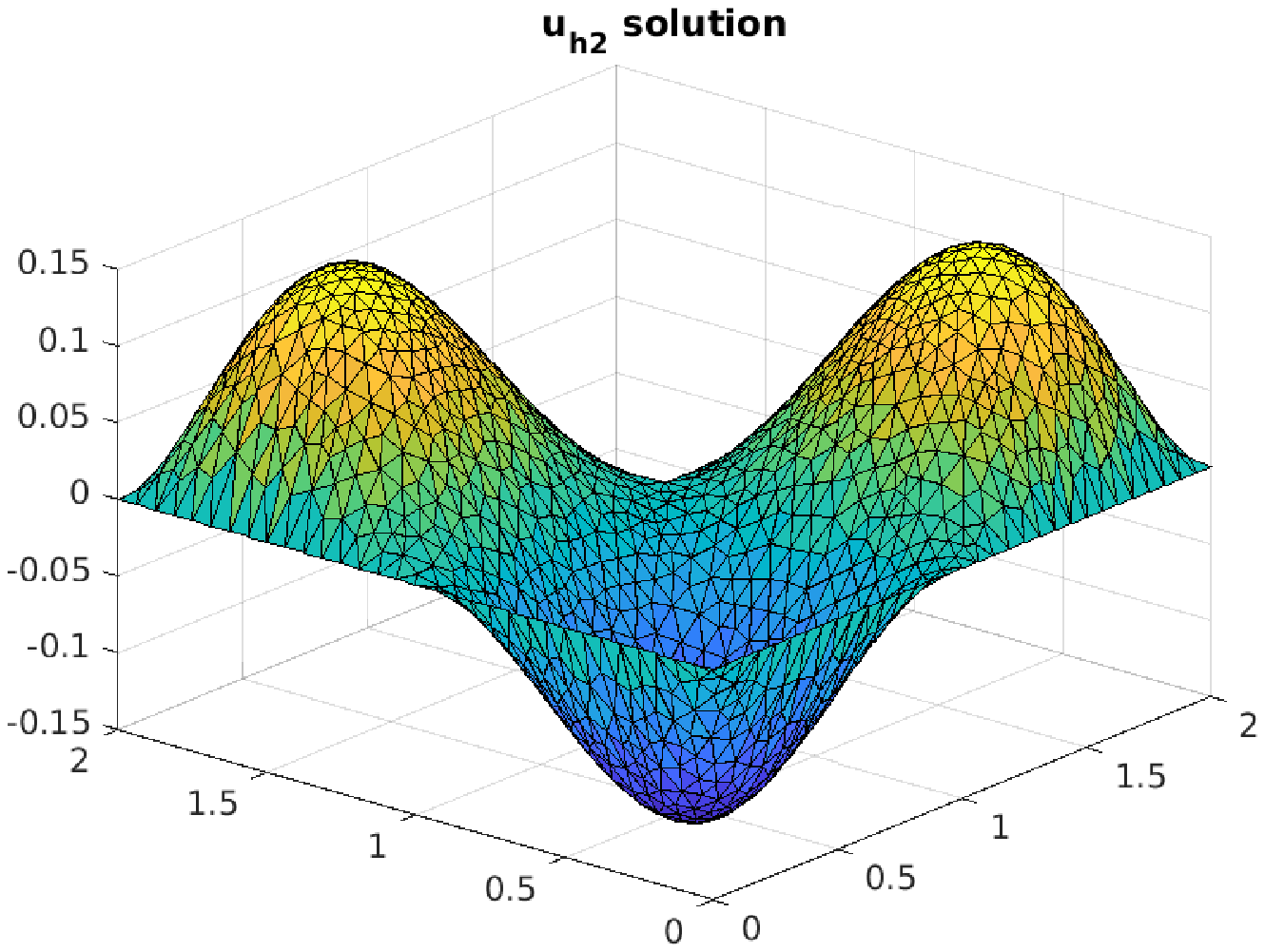}
	\caption{Second  components of the displacement $\boldsymbol{u}_2$. Left: exact solution, right: approximated solution, both computed with $\nu=0.35$.}

		\label{fig:plot10}
	\end{center}
\end{figure}

\bibliographystyle{siam}
\footnotesize
\bibliography{AmLeRi_ver_2}

\end{document}